\documentclass[acmsmall,nonacm]{acmart}
\usepackage[T1]{fontenc}

\usepackage{tikz}
\usepackage{dsfont}
\usepackage{amsfonts}
\usepackage{amsmath}
\usepackage{amsthm}

\usepackage{ragged2e}
\usepackage{bm}
\usepackage{array}
\usepackage{multicol}
\usepackage{multirow}
\usepackage{pifont}
\usepackage{dsfont}

\usepackage{enumitem}

\newlist{thmlist}{enumerate}{1}
\setlist[thmlist]{label=(\roman{thmlisti}), ref=\thethm.(\roman{thmlisti}),noitemsep}

\usepackage{thmtools}

\usepackage[utf8]{inputenc}

\newcommand{\EE}{\mathbb E }
\newcommand{\ti}{(t) }
\newcommand{\limg}{\lim_{\gamma \rightarrow 0}}
\newcommand{\MM}[1]{M_{#1}^\gamma(\phi)}
\newcommand{\lfsmall}[1]{\left( #1 \right )}

\newcommand{\q}{{\bf q}}
\newcommand{\vsum}[1]{\langle {\bf 1}, {\bf #1} \rangle}

\theoremstyle{remark}
\newtheorem{remark}{Remark}
\newtheorem{claim}{Claim}

\AtBeginDocument{%
  }

\setcopyright{acmcopyright}
\copyrightyear{2022}
\acmYear{2022}
\acmDOI{XXXX.XXXXXXX}

\acmJournal{POMACS}
\acmVolume{37}
\acmMonth{10}

\makeatletter
\let\@authorsaddresses\@empty
\makeatother

\setcopyright{none}
\renewcommand\footnotetextcopyrightpermission[1]{}
\makeatletter
\def\@copyrightspace{\relax}
\makeatother

\begin{document}
\title{Join-the-Shortest Queue with Abandonment: Critically Loaded and Heavily Overloaded Regimes}


\author{Prakirt R. Jhunjhunwala}
\email{prakirt@gatech.edu}
\affiliation{%
  \institution{Georgia Institute of Technology}
  \streetaddress{North Avenue}
  \city{Atlanta}
  \state{Georgia}
  \country{USA}
  \postcode{30332}
}

\author{Martin Zubeldia}
\email{zubeldia@umn.edu}
\affiliation{%
  \institution{University of Minnesota}
  \streetaddress{207 Church St SE}
  \city{Minneapolis}
  \state{Minnesota}
  \country{USA}
  \postcode{55455}
}

\author{Siva Theja Maguluri}
\email{siva.theja@gatech.edu}
\affiliation{%
  \institution{Georgia Institute of Technology}
  \streetaddress{North Avenue}
  \city{Atlanta}
  \state{Georgia}
  \country{USA}
  \postcode{30332}
}

\begin{abstract}
We consider a load balancing system comprised of a fixed number of single server queues, operating under the well-known Join-the-Shortest Queue policy, and where jobs/customers are impatient and abandon if they do not receive service after some (random) amount of time. In this setting, we characterize the centered and appropriately scaled steady state queue length distribution (hereafter referred to as limiting distribution), in the limit as the abandonment rate goes to zero at the same time as the load either converges to one or is larger than one.

Depending on the arrival, service, and abandonment rates, we observe three different regimes of operation that yield three different limiting distributions. The first regime is when the system is underloaded and its load converges relatively slowly to one. In this case, abandonments do not affect the limiting distribution, and we obtain the same exponential distribution as in the system without abandonments. When the load converges to one faster, we have the second regime, where abandonments become significant. 
Here, the system undergoes a phase transition, and the limiting distribution is a truncated Gaussian. Further, the third regime is when the system is heavily overloaded, and so the queue lengths are very large. In this case, we show the limiting distribution converges to a normal distribution. 

To establish our results, we first prove a weaker form of State Space Collapse by providing a uniform bound on the second moment of the (unscaled) perpendicular component of the queue lengths, which shows that the system behaves like a single server queue. We then use exponential Lyapunov functions to characterize the limiting distribution of the steady state queue length vector.
\end{abstract}




\maketitle
\thispagestyle{plain}

\section{Introduction}

With the ever-increasing human dependence on the internet, the past decade has seen an explosion in the number of data centers with up to hundreds of thousands of servers in each one of them. Increasingly more businesses and industries are dependent on these data centers to perform their compute-intensive tasks. As competition fuels business dynamics, there is a pressing need to make information exchange time-efficient. A five millisecond delay on Amazon’s website, can result in a loss of up to \$4 billion per millisecond in revenues \cite{yoav2019amazon}. If the delays become large, customers may abandon the system and move to other platform. The phenomenon of job abandonment is apparent in many systems. For example, customer abandonment \cite{garnett_designing_2002} plays a key role in the performance of call centers \cite{koole_queueing_2002},
where customers have limited patience. In food delivery systems or Perishable Inventory Management \cite{nahmias1982perishable}, the delivery or demand service needs to be fulfilled before the item perishes. A similar phenomenon is observed in Quantum Communication Networks \cite{yurke1984quantum, quantum2-towsley}, where entangled qubits can become useless after some time due to quantum decoherence \cite{schlosshauer2019quantum}. Therefore, understanding the effect of abandonment on the performance of general Stochastic Processing Networks \cite{williams_survey_SPN} is crucial to characterize the quality of service. 

Motivated by applications to data centers, in this paper we analyse the effect of abandonment on a Load Balancing system \cite{khiyaita2012load} operating under the Join-the-Shortest Queue (JSQ) policy. One of the most important performance metrics in Load Balancing systems is the queue lengths, which informs design decisions such as buffer dimensioning, and serves as a proxy for the delay. However, finding closed form expressions for the queue lengths can be quite challenging in general, so one often looks at the behaviour of such systems in various limiting regimes. The Heavy traffic regime \cite{kingman1962_brownian, bramson2001heavy} is one 
such regime that is well-studied in the literature 
especially 
in context of Load Balancing systems. With that in mind, we aim to analyze the steady state queue length distribution in a load balancing system with abandonments operating either in heavy traffic regimes, or in overload.

We consider a JSQ system with abandonments (dubbed JSQ-A), and the question we pose is the following: \textit{What is the limiting distribution of the queue length vector when the abandonment rate goes to zero in heavy traffic or overload?} 
We show that the limiting behavior of JSQ system is similar to that of a single-server queue with abandonments. Our main contributions are the following. 

\subsection{Main Contributions}
In this paper, we consider a parameterized family (with parameter $\gamma$) of discrete-time load balancing systems with a fixed number 
of single-server FIFO queues with total processing rate $\mu_\gamma$. Jobs arrive as an i.i.d. process of rate $\lambda_\gamma$, and all jobs that arrive in the same time slot join a shortest queue immediately upon arrival. Moreover, we assume that each job in a queue abandons the system in each time slot with probability $\gamma$, independently of other jobs. For this system, we characterize the limiting distribution of the (appropriately scaled) steady state queue length vector as the abandonment probability $\gamma$ goes to zero, in three different regimes. \\

\begin{itemize}
    \item[1.] \textbf{Classic heavy traffic regime:}  In the first regime, the load $\rho_\gamma = \lambda_\gamma/\mu_\gamma$ is such that the slack $1-\rho_\gamma$ is positive, and of order $\gamma^\alpha$ for some $\alpha \in (0,1/2)$. Here, the system is  underloaded 
    and we have $1-\rho_\gamma >> \sqrt{\gamma}$, i.e., the heavy-traffic parameter is large relative to abandonment rate. 
    For this regime, we show that the limiting distribution of JSQ-A is the same as the limiting distribution of JSQ without abandonments. In particular, we show that the scaled steady state queue length $\gamma^\alpha\q$ converges in distribution to an exponential random variable, i.e., that
    \begin{equation*}
        \gamma^\alpha\q \stackrel{d}{\rightarrow} \text{Exp}(\cdot) \times \bf{1},
    \end{equation*}
    where $\bf{1}$ is an $n$-dimensional vector of all ones.
    \item[2.] \textbf{Critical heavy traffic regime:} In the second regime, the load $\rho_\gamma$ is such that the slack $|1-\rho_\gamma|$ is of order $\gamma^\alpha$ for some $\alpha\in [1/2,\infty)$. Here, we allow the system to be underloaded, critically loaded, and overloaded. For this regime, we show that the limiting distribution of $\sqrt{\gamma}\q$ undergoes a phase transition, and changes from exponential to a truncated normal distribution. That is, we show that
    \begin{equation*}
        \sqrt{\gamma}\q \stackrel{d}{\rightarrow} \text{Truncated-Normal}(\cdot,\cdot) \times \bf{1},
    \end{equation*}
    where the mean of the underlying normal distribution is negative if $\rho_\gamma <1$ and $\alpha=1/2$, zero if $\rho_\gamma =1$ or $\alpha>1/2$, and positive if $\rho_\gamma>1$ and $\alpha=1/2$.
    \item[3.] \textbf{ Heavily overloaded regime:} In the third regime, the load $\rho_\gamma$ is such that the slack $1-\rho_\gamma$ is negative and of order $\gamma^\alpha$ for some $\alpha\in[0,1/2)$. Here, the system is heavily overloaded, with $\rho_\gamma-1>>\sqrt{\gamma}$. For this regime, we show that the scaled queue length $\sqrt{\gamma}\q$ diverges as $\gamma\to 0$, and that the distribution of the centered and appropriately scaled queue length converges to a normal distribution of zero mean. That is, we show that
    \begin{equation*}
        \sqrt{\gamma}(\q-\EE[\q]) \stackrel{d}{\rightarrow} \text{Normal}(0,\cdot) \times \bf{1}.
    \end{equation*}
\end{itemize}

In all the above regimes, in addition to the weak convergence result, we also establish convergence of moments (see Lemma \ref{COR: MGF_CONVERGENCE}). In essence, we show convergence of the moment generating functions (MGFs), which implies convergence in distribution as well as the moments. 

The above results can be interpreted as follows in the special case of 
a SSQ-A system. The service process and the abandonment process are similar to the discrete time analogue of an $M/M/1$ queue and an $M/M/\infty$ queue, respectively. Further, the ratio between the number of customers 
served by these analogous $M/M/1$ and $M/M/\infty$ queues (i.e., customers departing due to service and abandonment respectively) depends on the load and abandonment rates. 
In the classic heavy traffic regime, the 
behavior is completely dominated by the service process alone, and the abandonments are insignificant. Therefore, the 
overall behavior is similar to that of an $M/M/1$ queue. As we move to the critical heavy traffic regime, the abandonment, i.e., the $M/M/\infty$ behavior starts playing a significant role, and so, we observe a phase transition. Finally in the heavily overloaded regime, abandonment plays the primary role, and so the behavior is analogous to that of an $M/M/\infty$ queue. 
More details and precise mathematical definitions of the regimes are provided in Section \ref{sec: regimes}, and  details regarding the limiting distribution for JSQ-A are provided in Theorem \ref{THM: JSQ_LIMIT_DIS}. A pictorial representation of the results is given in Table \ref{tbl: dist}.


\newcolumntype{P}[1]{>{\centering\arraybackslash}p{#1}}
\newcolumntype{N}{@{}m{0pt}@{}}
\begin{table}[ht]
     \begin{center}
     \small\addtolength{\tabcolsep}{-3pt}
     \begin{tabular}{| P{2.6cm} | P{7.8cm} | P{2.8cm} | N }
     \hline
     \rule{0pt}{11pt}
     Classic-HT
     & Critical-HT   &  Heavily Overloaded   &  \\
     \hline
     \rule{0pt}{12pt}
     $\rho_\gamma<1, \alpha \in \big(0,\frac{1}{2}\big)$ & (a) $\rho_\gamma< 1, \alpha =\frac{1}{2}$, (b) $\rho_\gamma=1 \text{ or } \alpha >\frac{1}{2}$, (c) $\rho_\gamma>1, \alpha =\frac{1}{2}$ & $\rho_\gamma>1, \alpha \in \big[0,\frac{1}{2}\big)$ &\rule{0pt}{18pt} \\
     \hline
      $\gamma^\alpha \q \stackrel{d}{\rightarrow}$ & $\sqrt{\gamma}\q \stackrel{d}{\rightarrow}$ & $\sqrt{\gamma}(\q-\EE[\q]) \stackrel{d}{\rightarrow}$ &\rule{0pt}{16pt} \\ 
      \hline
     \includegraphics[width=0.16\textwidth]{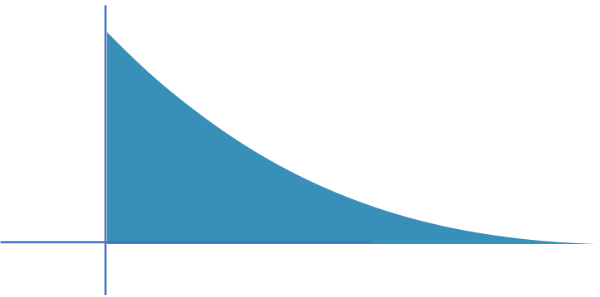}
      & \includegraphics[width=0.16\textwidth]{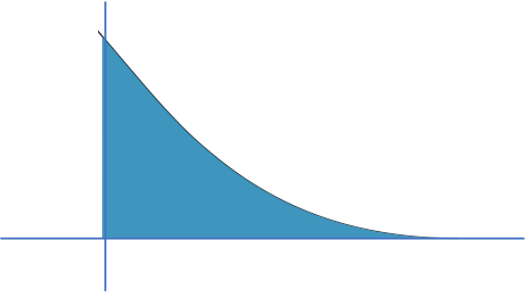} \includegraphics[width=0.16\textwidth]{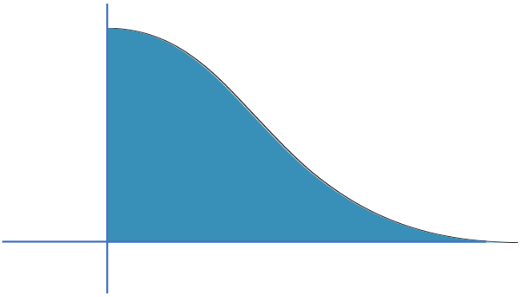} \includegraphics[width=0.16\textwidth]{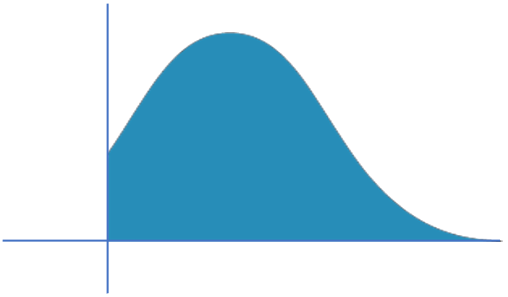}
      & \vspace{-34pt}
      \includegraphics[width=0.2\textwidth]{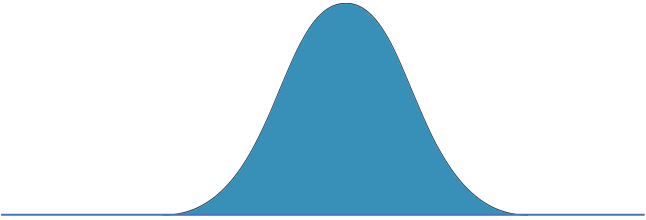} &
      \\ 
      $\text{Exp}(\cdot) \times \bf{1}$ & $\text{Truncated-Normal}(\cdot,\cdot) \times \bf{1}$ & $\text{Normal}(0,\cdot) \times \bf{1}$ & \\ \hline
      \end{tabular}
      \vspace{1mm}
      \caption{Limiting steady state distribution of the scaled queue length vector in the three regimes, as the abandonment probability $\gamma$ goes to zero. }
      \label{tbl: dist}
      \end{center}
\end{table}

A key step in characterizing the limiting distributions is to show that JSQ-A exhibits a \textit{State Space Collapse} (SSC), where the queue length vector collapses to the subspace where all of its coordinates are equal. This implies that the limiting behaviour of JSQ-A is similar to that of a SSQ-A. In this paper, we show a form of SSC where the second moment of the difference between the individual queue lengths and their average (i.e., the ``perpendicular component'' in SSC parlance) is uniformly bounded for all $\gamma>0$, even though the queue lengths diverge as $\gamma\rightarrow 0$. 
The formal statement for the SSC is provided in Theorem \ref{LEM: JSQ_SSC}.
While such SSC results for JSQ in the absence of abandonement are known, 
to the best of our knowledge, ours is the first work to show SSC for a JSQ-A. The SSC result plays a crucial role in 
characterizing its limiting distribution.

\subsection{Summary of Our Techniques}

Our key approach for proving the results is setting the drift of an exponential Lyapunov function (test function) to zero in steady-state, an approach dubbed the Transform Method that was first introduced in \cite{hurtado2020transform}. We work with the exponential of an 
appropriately centered and scaled steady state queue length, and perform a Lyapunov drift analysis to obtain our results. Our approach comprises three main components, each of which presents several technical challenges that arise due to the abandonments. These challenges set our work apart from existing literature. We mention some of the fundamental challenges below.

\begin{enumerate}
    \item \textbf{Boundedness of the MGF:} A primary requirement of setting the drift to zero 
    is that the expected value of the chosen test function should exist and be finite. Further, such a result has to hold even asymptotically (here, as $\gamma\rightarrow 0$). Since we use exponential test functions, we need the boundedness of the MGF of the steady state queue length. 
    While an upper bound on the MGF of steady state queue length of the JSQ system was presented in \cite{hurtado2020transform}, it is not bounded asymptotically, and it is not applicable in presence of abandonments. 
    In order to handle abandonments, we prove that the MGF of the (centered and appropriately scaled) steady state queue length is bounded in an interval around $0$, uniformly for all $\gamma>0$. 
    
    This turns out to be a major technical challenge, especially in the heavily overloaded regime, where the 
    abandonment rate diverges since the
    scaled mean queue length itself diverges as $\gamma \rightarrow 0$. 
    We address this challenge as follows: We first identify the rate at which the mean diverges. Based on this rate, we  couple the abandonment process with two processes
    where the abandonments are truncated. 
    We finally use Lyapunov drift arguments on the coupled processes to 
    obtain a uniform bound on the MGF. More details on the boundedness of the MGF are provided in Lemma \ref{lem: jsq_mgfexist}.
    
    \item \textbf{State Space Collapse:} 
    In the absence of abandonments, SSC for the JSQ system was established using the drfit arguments in \cite{atilla} (or in \cite{hurtado2020transform}), where the authors show that all the moments of a well designed scaled perpendicular component (the difference between the queue length vector and its lower dimensional representation) converge to zero. However, their technique does not extend to the case when there are abandonments. 
    In this work, as presented in Theorem \ref{LEM: JSQ_SSC}, we show that JSQ-A exhibits a form of SSC. A major challenge here, unlike in \cite{atilla}, is that there is an inter-dependency between the perpendicular component and the abandonments. Usually,  SSC is proved using the drift arguments from 
    \cite{hajek1982hitting} which relies on a bounded drift condition. Since the number of abandonments  is proportional to the queue lengths, they are unbounded, and so the bounded drift condition is not satisfied. As such, abandonments create an (unbounded) additional variability in the perpendicular component, which is challenging to deal with. Thus, arguing that the perpendicular component has a negative drift (as in \cite{atilla}) is not enough. To tackle this, we use additional steady state arguments. We first get a handle on the number of abandonments in steady state using the boundedness of MGF result (previous paragraph). 
    We then decouple the variability due to abandonments from the perpendicular component to get a bound on its second moment. 
    
 \item \textbf{Differential equation for MGF:}
The primary approach in our proof is to set the drift of an exponential function to zero, after which we use 
the boundedness of MGF,  SSC, and second order Taylor approximations.  
In the absence of abandonment, one then directly obtains the MGF of the limiting distribution \cite{hurtado2020transform}. 
We similarly obtain the limiting MGF in the classical heavy-traffic regime. However, the other two regimes are more challenging, where
we obtain differential equations on the MGF  (one each for each of the two regimes), where the derivative term appears because the abandonment process acts like an $M/M/\infty$ queue. Thus, the differential equation captures the trade-off between $M/M/1$  and $M/M/\infty$ like behaviors in the JSQ-A system. 
These differential equations involve certain terms that correspond to unused service (idleness of servers due to lack of customers). Characterizing these terms in each of the two regimes leads is a technical  challenge that we overcome. We finally solve these differential equations which gives the MGFs of the limiting distributions, thus proving the results. 
\end{enumerate}

\subsection{Related Work}
The study of the limiting behaviour of the steady state queue length distribution was first done by Kingman \cite{kingman1961_charfunction, kingman1962_brownian}, where he studied a single server queue in heavy traffic. In the past couple of decades, there have been quite a large body of work about the analysis of a SSQ-A with extensions to Multi-Server Queues with Abandonment. A popular tool to analyse such systems is a process level diffusion limit approximation \cite{harrison1978diffusion}, where the process level convergence to a (reflected) Brownian motion \cite{morters2010brownian} is shown. Closest to our results are the results for continuous time SSQ-A in \cite{ward_diffusion_2003}, with follow-up work in \cite{ward_diffusion_2005, reed_approximating_2008}. There, the authors showed that, under a certain scaling of the system parameters, the scaled queue length process converges weakly to an Ornstein-Uhlenbeck (O-U) process \cite{uhlenbeck1930theory}. Since then, there have been some efforts towards studying Multi-server queues with abandonment  \cite{dai_customer_2010, dai_many-server_2010, zeltyn_call_2005, kang2010fluid}. We refer the readers to \cite{dai2012many} for a more detailed survey on many-server queues with abandonment. In a different setting, in \cite{P2P}, the authors study a peer-to-peer content distribution network. There, the peer-to-peer nature of the network makes it behave as a mixture of a single-server and an infinite server queue for the users. In that setting, they show that the steady state distribution of the number of users converges to either the distribution of a single server queue or of a infinite server queue, depending on the (fluid) scaling applied.

In the literature that relies on process level diffusion limits (such as \cite{ward_diffusion_2003}), they would need to show a certain interchange of limits in order to show that the limit of the steady state distribution of the scaled queue length process matches the steady state distribution of the limiting process. This interchange of limits is generally hard to prove, and it is generally not done. In that regard, our work is complementary to the work based on process level diffusion limits in the sense that we directly show the convergence of the steady state distribution. Even without the complication brought by the need to prove interchange of limits, the process level diffusion limit approach is hard to generalize to load balancing systems. With no abandonment it is known that, for many load balancing systems, the queue length process behaves like a SSQ. This phenomenon is called \textit{State Space Collapse} \cite{Williams_state_space, Bramson_state_space, rei_state_space}, as the state space of the multi-dimensional system collapses to a one-dimensional subspace. The idea of state space collapse was first introduced in \cite{foschini1978basic}, where the authors used the diffusion limit approach to show that Join-the-Shortest Queue is heavy traffic optimal for a two-server system. JSQ \cite{Gamarnik_JSQ} has been popular and well-studied example of a load balancing system, commonly used to model queueing systems in supermarkets and Data Centers. 
Under the JSQ policy, the state space collapses to a subspace in which all the queues are equal. A major challenge in generalizing the diffusion limit approach to JSQ-A system is to prove that the state space collapse is still achieved.

Several alternatives to the diffusion limit approach have been developed recently to study the steady state distribution a queueing system in heavy traffic. For example, in \cite{braverman_steins_2017}, the Stein's method (previously introduced in \cite{braverman2017stein}) was used to study a $M/Ph/n+M$ system. Another alternative to a diffusion limit approach is the drift method, which was first introduced in \cite{atilla}. The idea behind the drift method is to choose a Lyapunov function and equate its drift to zero in steady state to derive bounds on meaningful quantities, like on the moments of the queue lengths. As an extension of the drift method, the Transform method was developed in \cite{hurtado2020transform} and applied to different queueing models in \cite{Jhun_heavy_traffic, twosided_ht}, which uses exponential Lyapunov functions to characterize the limiting MGF (or characteristic function) of the queue length process. 

\subsection{Basic Notations}
We use $\mathbb R$ to denote the set of real numbers and $\mathbb R_+$ to denote the set of non-negative real numbers. Also, $\mathbb N$ denotes the set of natural numbers. Similarly, $\mathbb R^d$ denotes the set of $d$-dimensional real vectors. We use bold letters to denote vectors and, for any vector ${\bf x}$, we use $x_i$ to denote the $i^{th}$ coordinate of ${\bf x}$. The inner product of two vectors ${\bf x}$ and ${\bf y}$ in $\mathbb R^d$ is defined as $\langle {\bf x},{\bf y}\rangle = {\bf x}^T {\bf y} = \sum_{i=1}^d x_iy_i$. For any vector $x\in \mathbb R^d$, the $\ell_2$-norm is denoted by $\|{\bf x}\| = \sqrt{\langle {\bf x} , {\bf x}\rangle }$. For any positive natural number $d$, $\bf{1}_d$ and $\bf{0}_d$ denotes the vector of all ones and vector of all zeros of size $d$ respectively. For ease of notation, at most places, we drop the subscript and just use $\bf{1}$ and $\bf{0}$ instead of $\bf{1}_d$ and $\bf{0}_d$.  For any set $A$, $\mathds{1}_A$ denotes the indicator random variable for set $A$. We use $\lfloor \cdot \rfloor$ and $\lceil \cdot \rceil$ to denote the floor and ceiling functions respectively. For any random variable $X$ and real number $\gamma>0$, $\MM{X} = \EE[e^{\sqrt{\gamma}\phi X}]$ denotes the Moment Generating Function (MGF) of $\sqrt{\gamma}X$. For a sequence of random variables $\{X_{\gamma}\}_{\gamma \in \Gamma}$, we use $X_\gamma \stackrel{d}{\rightarrow} X$ to denote that  $\{X_{\gamma}\}_{\gamma \in \Gamma}$ converges in distribution to random variable $X$.\\

\section{System model}
\label{sec: model}

In this section, we present the mathematical model for our JSQ-A system, and then introduce the different limiting regimes in Subsection \ref{sec: regimes}.\\

We consider a discrete time queueing system consisting of $n$ single-server FIFO queues of infinite capacity. We denote the queue length vector at the beginning of time slot $t$ by $\q(t)$, where $q_i(t)$ is the length of the $i^{th}$ queue. Jobs are assumed to be impatient and, at the start of each time slot, any job waiting in the queue chooses to abandon the queue with probability $\gamma>0$, independently of other jobs. The total number of jobs that abandon the queue in time slot $t$ is denoted by ${\bf d}(t)$, with ${\bf d}(t) \sim Bin(\q\ti,\gamma)$.

Jobs arrive to the system as an i.i.d. process $\{a(t)\}_{t\geq 0}$, with $\EE [a(t)] = \lambda_\gamma$, Var$(a(t)) = \sigma_{\gamma,a}^2$, and $a(t)\leq A$ almost surely. When $a(t)$ jobs arrive, they are all dispatched according to the Join-the-Shortest Queue policy to the same queue. That is, they are all dispatched to a queue with index
\begin{equation*}
    i^*(t) \in \underset{i\in\{1,\dots,n\}}{\arg\min} \big\{q_i(t)\big\},
\end{equation*}
where ties are broken uniformly at random. Once the jobs join the queue, the servers serve up to ${\bf s}(t)$ jobs waiting in the queues, with $\EE[{\bf s}(t)] = \boldsymbol{\mu}_\gamma$\footnote{For all discussions and interpretations throughout the paper, we assume that ${\bf \mu}_\gamma$ is of constant order. However, our results hold in more generality, even when ${\bf \mu}_\gamma\to 0$ as $\gamma\to 0$.}, Var$(s_i(t)) = \sigma_{\gamma,i}^2$, and $s_i(t)\leq A$ almost surely, for any $i\in \{1,\dots,n\}$. Similar to the arrival process, the potential services are also independent and identically distributed across time slots. Moreover, these potential services are independent of the queue length vector. 

To describe the queue dynamics, we denote the ``dispatching action" chosen by the dispatcher by ${\bf Y}(t) \in \{0,1\}^n$ such that, $Y_{j}(t) = 1$ for $j = i^*(t) $ and $Y_{j}(t) = 0$ otherwise. We denote the difference between arrivals and potential services as ${\bf c}(t) = a(t) {\bf Y}(t) -{\bf s}(t)$, and for ease  of notation, we denote $c(t)= \langle {\bf c}(t), {\bf 1}\rangle$. Note that, even though ${\bf c}(t)$ depends on the current queue length vector ${\bf q}(t)$, the sum $c(t)= \langle {\bf c}(t), {\bf 1}\rangle$ does not, as $\langle {\bf Y}(t), {\bf 1}\rangle =1$ for any time $t$. Then, $\EE[c(t)] = \nu_\gamma = \lambda_\gamma -\langle \boldsymbol \mu_\gamma, {\bf 1}\rangle$ and Var$(c(t)) = \sigma^2_\gamma = \sigma^2_{\gamma,a} +\langle \bm{\sigma}^2_{\gamma}, {\bf 1} \rangle$. Using this, the queue length process is given by
\begin{equation}
\label{eq: jsq_lindley}
    \q(t+1) = [\q(t) +a(t){\bf Y}(t) -{\bf s}(t) -{\bf d}(t)]^+ = \q(t) +{\bf c}(t) -{\bf d}(t)+ {\bf u}(t) ,
\end{equation}
where the operation $[\,\cdot\,]^+$ above is used because the queue lengths cannot be negative, and the term ${\bf u}(t)$ are the unused services that arise because there might not be enough jobs to serve. Note that, the unused service term $u_i\ti $ is positive only if $q_i(t+1) =0$, which implies $q_i(t+1)u_i(t) =0$ for all $i$, or simply $\langle \q(t+1),{\bf u}(t) \rangle =0$ for any $t>0$. Also, the unused service cannot be larger than the service itself, so we have $0\leq u_i(t)\leq s_i(t) \leq A$.  

We drop the dependence on $t$ to denote the variables in steady state, i.e., $\q$ follows the steady state distribution of the queue length process $\{\q(t)\}_{t=0}^\infty$, ${\bf Y}$ gives the destination of incoming jobs under JSQ for the state $\q$, and ${\bf d} \sim Bin(\q,\gamma)$. We use $\q^+$ to denote the state that comes after $\q$, i.e., $\q^+ = \q + {\bf c} - {\bf d} +{\bf u}$, where ${\bf c} = a {\bf Y} - {\bf s}$, with $a$ and ${\bf s}$ being distributed as $a(t)$ and ${\bf s}(t)$, respectively. The steady state distribution of the random variables $(\q^+,\q,{\bf c},{\bf d},{\bf u})$ depends on the parameter $\gamma$ but, for ease of notation, we do not make this explicit.

When $n=1$, the system described above corresponds to a Single Server Queue with Abandonment (SSQ-A). A major distinction between SSQ-A and JSQ-A is that the dispatcher has only one queue to send jobs to, and so there is no decision to be made. For simplicity and distinction from JSQ-A, we do not use bold face fonts in the notation for SSQ-A.\\

\subsection{Regimes of interest} \label{sec: regimes}

Since the number of abandonments grows with the queue lengths, then the expected number of abandonments is larger than the expected number of arrivals if the queue lengths are large enough. This ensures the stability of the system for any values of $\lambda_\gamma$, $ \mu_\gamma$, and $\gamma>0$. Formally, the stability can be shown by using the Foster-Lyapunov Theorem with $f(q)=q^2$ as the Lyapunov function. Our goal is to find the limit of the steady state distribution of the (appropriately scaled and centered) queue length vector. Depending on how $\nu_\gamma=\lambda_\gamma -\langle \boldsymbol \mu_\gamma, {\bf 1}\rangle$ behaves as a function of $\gamma$, we have three regimes:
\begin{enumerate}
    \item \textit{Classic Heavy Traffic Regime}: The first regime is when there exist constants $C_f >0$, $\alpha\in \lfsmall{0,1/2}$, and $\beta >0$ such that $|\nu_\gamma +C_f\gamma^{ \alpha }| \leq \gamma^{\alpha + \beta}$. In particular, this means that $\nu_\gamma=-C_f\gamma^{ \alpha }+\mathcal{O}\big(\gamma^{\alpha+\beta}\big)$, and thus the the system is in underload (i.e., we have $\nu_\gamma<0$ for all $\gamma$ small enough). We call this regime `\textit{Classic Heavy Traffic}'  because, in this regime, the system is in a light enough heavy traffic so that the queues are not large enough for the drift due to abandonments to be significant. 
    
    \item \textit{Critical Heavy Traffic Regime}: The second regime is when there exist constants $C_c\in \mathbb R$ and $\beta > 0$ such that $|\nu_\gamma -C_c\sqrt{\gamma}| \leq \gamma^{\frac{1}{2}+\beta}$. In particular, this means that $\nu_\gamma=C_c\sqrt{\gamma}+\mathcal{O}\big(\gamma^{\frac{1}{2}+\beta}\big)$, and thus the system can be in underload if $C_c<0$, in overload if $C_c>0$, or in a very heavy traffic if $C_c=0$. We call this regime `{\it Critical Heavy Traffic}' because, in this regime, the system is very close to being critically loaded. 
    
    \item \textit{Heavily Overloaded Regime}: The third and final regime is when there exist constants $C_s>0$, $\alpha \in [0,1/2)$ and $\beta >0$ such that $|\nu_\gamma - C_s \gamma^{\alpha}| \leq \gamma^{\alpha +\beta}$. In particular, this means that $\nu_\gamma=C_s \gamma^{\alpha}+\mathcal{O}\big(\gamma^{\alpha +\beta}\big)$, and thus the system is in overload (i.e., we have $\nu_\gamma>0$ for all $\gamma$ small enough). We call this regime `{\it Heavily Overloaded}' because, in this regime, the system is far enough in overload so that the queue length has to become very large for the (negative) drift due to abandonments to be in equilibrium with the (positive) drift due to arrivals/services.
\end{enumerate}
The results (and more discussions) about each of the regimes are provided in the following section.\\

\section{Main results}
\label{sec: results}

In this section we present our main results for JSQ-A.  In particular, we provide our main result about the limiting queue length distributions in Subsection \ref{sec: jsq_limit_thm_discussion}, and our State Space Collapse result in Subsection \ref{sec: jsq_ssc}.

\subsection{Limiting distribution of JSQ-A}\label{sec: jsq_limit_thm_discussion}

The following theorem provides the limiting distribution of the (appropriately scaled and centered) steady state queue length vector for JSQ-A operating under each of the regimes introduced in Section \ref{sec: regimes}. Its proof is given in Section \ref{sec: proofs}.

\begin{theorem}
\label{THM: JSQ_LIMIT_DIS}
Consider the JSQ-A system as in Section \ref{sec: model}. We have the following results.
\begin{enumerate}[label=(\alph*), ref=\ref{THM: JSQ_LIMIT_DIS}.\alph*]
    \item \label{thm: jsq_fast}  Classic Heavy Traffic: Suppose $ |\nu_\gamma +C_f\gamma^{ \alpha }| \leq \gamma^{\alpha + \beta}$, where $ \alpha\in \lfsmall{0,1/2}$, $\beta >0$, and $C_f >0$ are constants. Let $\Upsilon_f$ be an exponential random variable with mean $\frac{\sigma^2}{2nC_f}$, where $\sigma^2 = \limg \sigma^2_\gamma$. Then, for any $m_1,\dots,m_n \in \mathbb N$ with $\sum_{i=1}^n m_i=m$, we have that as $\gamma\rightarrow 0$,
     \begin{align*}
         \gamma^\alpha {\bf q} \stackrel{d}{\rightarrow} \Upsilon_f \bf 1, && \limg \gamma^{\alpha m} \EE \left[ \prod_{i=1}^n q_i^{m_i} \right] = \EE[\Upsilon_f^m].
     \end{align*}
    \item \label{thm: jsq_crit} Critical Heavy Traffic: Suppose $|\nu_\gamma-C_c\sqrt{\gamma}| \leq \gamma^{\frac{1}{2} + \beta}$, where $\beta >0$, and $C_c \in \mathbb R$ are constants. Let $\Upsilon_c$ be a Gaussian random variable with mean $\frac{1}{n} C_c$ and variance $\frac{1}{2n^2}\sigma^2$. Then, for any $m_1,\dots,m_n \in \mathbb N$ with $\sum_{i=1}^n m_i=m$, we have that as $\gamma\rightarrow 0$,
    \begin{align*}
        \sqrt{\gamma} {\bf q} \stackrel{d}{\rightarrow} [\Upsilon_c  | \Upsilon_c>0] {\bf 1}, && \limg \gamma^{\frac{m}{2}}\EE\left[ \prod_{i=1}^n q_i^{m_i} \right] = \EE[\Upsilon_c^m | \Upsilon_c>0],
    \end{align*}
    where $\Upsilon_c | \Upsilon_c >0$ denotes the random variable $\Upsilon_c$ conditioned on the event $\{\Upsilon_c>0\}$.
    \item \label{thm: jsq_slow} Heavily Overloaded: Suppose $|\nu_\gamma - C_s \gamma^{\alpha}| \leq \gamma^{\alpha +\beta}$, where $\alpha \in [0,1/2)$, $\beta >0$, and $C_s >0$ are constants. Let $\Upsilon_s$ is a Gaussian random variable with zero mean and variance $\frac{1}{2n^2}\bar\sigma^2$, where $\bar\sigma^2 = \lim_{\gamma\rightarrow 0} \sigma^2_\gamma +\nu_\gamma$. Then, for any $m_1,\dots,m_n \in \mathbb N$ with $\sum_{i=1}^n m_i=m$, we have that as $\gamma\rightarrow 0$,
\begin{align*}
     \sqrt{\gamma} \left( {\bf q} - \frac{\nu_\gamma}{n\gamma} \bf 1 \right) \stackrel{d}{\rightarrow} \Upsilon_s \bf 1, &&  \limg \gamma^{\frac{m}{2}}\EE \left[\prod_{i=1}^n  \left( q_i - \frac{\nu_\gamma}{n\gamma}\right)^{m_i} \right] = \EE[\Upsilon_s^m].
\end{align*}
\end{enumerate}
\end{theorem}

We now provide insights into the results and proof techniques for each of the regimes presented in Theorem \ref{THM: JSQ_LIMIT_DIS}. For the sake of simplicity, the intuitions are given for the case $n=1$, which corresponds to the SSQ-A system. However, the same intuitions hold in general when we combine them with the State Space Collapse detailed in Subsection \ref{sec: jsq_ssc}, which implies that all coordinates of the scaled queue length vector converge to the same random variable.

\subsubsection{Classic Heavy Traffic Regime}
Theorem \ref{thm: jsq_fast} states that the limiting distribution in the classic heavy traffic regime is the same exponential distribution as the one for the same system (but without abandonments) in heavy traffic \cite{hurtado2020transform}, with the same scaling $\gamma^\alpha$. 
To understand why the limit is not affected by the abandonments note that, without abandonments, the queue length would be of order $(1-\rho_\gamma)^{-1} \approx \gamma^{-\alpha}$. Therefore, the magnitude of the (negative) drift due to abandonments, given by $-\gamma q$ (as $d\sim Bin(q,\gamma)$), is at most of order $\gamma^{1-\alpha}$. On the other hand, the (negative) drift due to the arrivals/services is $\nu_\gamma$, which is of order $\gamma^\alpha$. Since $\alpha<1/2$, we have $\gamma^{1-\alpha}<<\gamma^\alpha$, and thus the drift due to abandonments is negligible compared to the drift due to arrivals/services. Hence, the limiting distribution is unaffected by the abandonments.\\

{\it Technical note:} In this regime, using the transform method, we establish the equation 
\begin{equation*}
    \left( \frac{1}{\gamma^\alpha}\nu_\gamma + \frac{1}{2}\phi \big( \sigma^2_\gamma + \nu_\gamma^2 \big) \right) \EE\left[{e^{\gamma^\alpha\phi \vsum{q} }}\right] + \frac{1}{\gamma^\alpha} \EE[\vsum{u}] = o(1),
\end{equation*}
which is valid for all $\phi$ in a neighbourhood of zero. Essentially, we show that the term corresponding to the abandonments is approximately zero as $\gamma \rightarrow 0$. Here, we can show that $\EE[\vsum{u}] \approx -\nu_\gamma$. Therefore, we observe that the solution of the equation corresponds to the MGF of an exponential distribution. Additionally, using the State Space Collapse result (presented in Subsection \ref{sec: jsq_ssc}) we prove that, for suitably chosen $\boldsymbol \phi$, $\EE[e^{\gamma^\alpha \langle \boldsymbol \phi, \q \rangle}] \approx \EE[e^{\frac{1}{n}\gamma^\alpha \phi \langle \bf 1, \q \rangle}]$ for $\gamma$ small enough, where $\phi = \langle \bf 1, \boldsymbol \phi \rangle$. Thus, we only need to find the limiting distribution of the total queue length. The proof of Theorem~\ref{thm: jsq_fast} is provided in Section \ref{sec: jsq_proofsketch1}. 

\subsubsection{Critical Heavy Traffic Regime}

Theorem \ref{thm: jsq_crit} states that the limiting steady state queue length distribution in the critical heavy traffic regime is a truncated Gaussian distribution, with scaling factor $\sqrt{\gamma}$. Therefore, abandonments do affect the limiting distribution in this regime. 
To understand why this is the case, we consider the cases $C_c<0$, $C_c>0$, and $C_c=0$ separately.
\begin{itemize}
    \item If $C_c<0$, the system is in underload and, without abandonments, it would have a queue length of order $(1-\rho_\gamma)\approx \gamma^{-\frac{1}{2}}$. For these queue lengths, the magnitude of the drift due to abandonments is of order $\sqrt{\gamma}$. This is comparable to the drift due to arrivals/services $\nu_\gamma \approx \sqrt{\gamma}$, but it is not enough to make the queue smaller than of order $\gamma^{-\frac{1}{2}}$. 
    \item If $C_c>0$, the system is in overload and would not be stable without abandonments. In order for the abandonments to compensate for the positive drift $\nu_\gamma\approx\sqrt{\gamma}$, the queue length must be of order $\gamma^{-\frac{1}{2}}$.
    \item If $C_c=0$, the system is in a very heavy traffic $|1-\rho_\gamma|\approx \gamma^{\frac{1}{2}+\beta}$. However, this system can be upper and lower bounded by systems with $C_c>0$ and $C_c<0$, respectively. Therefore, its queue length must also be of order $\gamma^{-\frac{1}{2}}$.
\end{itemize}
In all three cases, there is a critical inter-play between abandonments and arrivals/services, which is missing in the classic heavy traffic regime. Moreover, unlike in the classic heavy traffic regime (where abandonments are negligible), in this regime the abandonments make the drift become smaller (larger) as the queue length becomes larger (smaller). This results in much faster decay of the p.d.f. of the steady state distribution. In particular, this leads to a phase transition in the limiting distribution, from an exponential distribution in the classic heavy traffic regime, to a truncated normal distribution in the critical heavy traffic regime.\\ 

{\it Technical note:} In terms of the MGF equation, the term introduced due to abandonment acts as a derivative of $\MM{\vsum{q}}$ for small values of $\gamma$ and so, instead of directly getting the MGF as in the classic heavy traffic regime, we get a differential equation on the MGF $\MM{\vsum{q}} = \EE\big[e^{\sqrt{\gamma}\phi \vsum{q}}\big]$ given by
\begin{equation*}
   \left(\frac{1}{\sqrt{\gamma}}\nu_\gamma + \frac{1}{2} \phi \big(\sigma^2_\gamma+\nu_\gamma^2\big) \right)\MM{\vsum{q}} + \frac{1}{\sqrt{\gamma}}\EE[\vsum{u}]  -\frac{d}{d\phi} \MM{\vsum{q}}  = o(1),
\end{equation*}
which is valid for all $\phi$ in a neighbourhood of zero, and for any negative value of $\phi$. Here, $\EE[\vsum{u}]$ acts as an unknown constant. The reason why we establish the above differential equation for all negative values of $\phi$ is so that we can solve for $\EE[\vsum{u}]$. The solution of the differential equation matches with that of a truncated normal distribution. The proof of Theorem \ref{thm: jsq_crit} is provided in Section \ref{sec: jsq_proofsketch2}. 

\subsubsection{Heavily Overloaded Regime}

Theorem \ref{thm: jsq_slow} states that the limiting distribution of the (appropriately centered and scaled) steady state queue length is Gaussian. In this case, the centering and scaling used imply that the mean queue length is of order $\nu_\gamma/\gamma$, and that the variance is of order $\gamma$. We provide intuitive explanations for these magnitudes below.
\begin{itemize}
    \item {\it Mean queue length magnitude:} Note that the drift due to arrivals/services is $\nu_\gamma$, which is positive and much larger than $\sqrt{\gamma}$. As a result, for the drift due to abandonments (given by $-\gamma q$) to compensate for this, the queue length must be of order $\nu_\gamma/\gamma >> \gamma^{-\frac{1}{2}}$. 
    \item {\it Variance magnitude:} Recall that, in the critical heavy traffic regime, $\sqrt{\gamma}\q$ converges to a normal random variable (with mean $C_c$) truncated at zero. Thus, $\sqrt{\gamma} \bar{\q} : = \sqrt{\gamma} \big( \q - \frac{\nu_\gamma}{n\gamma} \bf 1\big) $ converges to a zero mean normal random variable truncated at $-C_c$, where $C_c = \lim_{\gamma\rightarrow 0} \nu_\gamma/\sqrt{\gamma}$. The heavily overloaded regime can be thought of as a critical heavy traffic regime where the parameter $C_c$ goes to $\infty$, and thus the truncation happens at $-\infty$ (which is the same as no truncation). Hence, the variance has the same scaling.\\ 
\end{itemize}

{\it Technical note:} In the heavily overloaded regime, unlike in the other two regimes, we observe that the unused services in the system (divided by $\sqrt{\gamma}$) are approximately zero. Intuitively, this holds because the scaled queue length $\sqrt{\gamma}\q$ (without centering) goes to infinity under the heavily overloaded regime, and so there is almost no unused services. In order to prove the result, we establish a differential equation on the appropriately scaled and centered queue length 
given by
\begin{equation*}
    \frac{1}{2}\phi \bar \sigma^2_{\gamma}\MM{\langle \bf 1 , \bar{\q}\rangle}  -  \frac{d}{d\phi} \MM{\langle \bf 1 , \bar{\q}\rangle} \in o(1),
\end{equation*}
which is valid for all $\phi$ in a neighbourhood of zero. One can check that the solution to the above differential equation is the MGF of a zero mean normal distribution. The proof of Theorem \ref{thm: jsq_slow} is provided in Section \ref{sec: jsq_proofsketch3}. 

\subsection{State Space Collapse for JSQ-A}
\label{sec: jsq_ssc}
In this subsection, we present our State Space Collapse (SSC) result for the JSQ-A system. In order to state this result, we define the subspace $\mathcal{S} \subset \mathbb R_{+}^n$ by
\begin{equation*}
    \mathcal{S} = \big\{{\bf x} \in \mathbb R^n_+ : \exists w \in \mathbb R_+ \ s.t. \  x_i = w, \ \forall i \big\}. 
\end{equation*}
We define $\q_{\|}$ as the projection of the queue length vector $\q$ onto the subspace $\mathcal{S}$, that is 
\begin{align*}
    \q_{\|} = \arg\min_{\mathbf x\in \mathcal S} \|\mathbf q - \mathbf x\|.
\end{align*}
Moreover, the perpendicular component is denoted by $\q_{\perp} = \q - \q_{\|}$. Since $\mathcal{S}$ is a one-dimensional subspace, we have $q_{\|i} = \frac{1}{n} \sum_{i=1}^n q_i =  \frac{1}{n}\langle {\bf q},{\bf 1} \rangle$ or $\q_{\|} = \frac{1}{n}\langle {\bf q},{\bf 1} \rangle \mathbf 1$. In the theorem below, we show that the second moment of $\|\q_{\perp}\|$ is bounded by a constant, uniformly for any $\gamma>0$. \\

\begin{theorem}
\label{LEM: JSQ_SSC}
Consider the JSQ-A system as defined in the Section \ref{sec: model}. Suppose $\nu_\gamma \geq -\frac{1}{2}n\mu_{\gamma,\min}$, where $\mu_{\gamma,\min} = \min_i\{\mu_{\gamma,i}\}$ Then, for any $\gamma \in (0,1)$, in steady state, we have
\begin{equation*}
    \EE\left[ \|{\q}_{\perp}\|^2 \right] \leq M_{\perp},
\end{equation*}
where $M_{\perp}$ is a constant independent of $\gamma$.\\
\end{theorem}

{\it Proof sketch:} In order to prove this, we first get a bound on the first order drift of the perpendicular component, i.e., on $\EE[\Delta \|\q_\perp \| | \q]$, by using drift arguments on $\|\q\|^2$ and $\|\q_{\|}\|^2$, and then using that $2\|\q_{\perp}\|\Delta\|\q_{\perp}\| \leq \Delta \|\q_{\perp}\|^2 $ and the Pythagoras identity, $\|\q_\perp \|^2 = \|\q\|^2 - \|\q_{\|}\|^2$. Next, we equate the drift of the third moment, i.e., $\EE[\Delta \|\q_\perp \|^3|\q]$ to zero in steady state, and use the negative drift of $\|\q_\perp\|$ (i.e., $\EE[\Delta \|\q_\perp \| | \q]$) to get a bound on the second moment $\EE[\|\q_\perp \|^2]$. The proof of Theorem~\ref{LEM: JSQ_SSC} is provided in Section \ref{sec: jsq_ssc_proof}.\\

{\it Intuitive interpretation:} 
Theorem \ref{LEM: JSQ_SSC} states that the second moment of the norm of the projection $\q_{\perp}$ (i.e., the norm of the difference between the individual queue lengths and their average) is uniformly bounded by a constant for any $\gamma>0$. Furthermore, Theorem \ref{THM: JSQ_LIMIT_DIS} states that the moments of the scaled queue length vector converge to a positive number, and thus the unscaled moments diverge as $\gamma\rightarrow 0$ (e.g., in the critical heavy traffic regime, we have that $\EE[\|\q\|^2]$ is of order $\gamma^{-1}$). Combining these two results, we conclude that the scaled queue length vectors are approximately equal. Formally, this phenomenon is called \textit{State Space Collapse} (SSC). 

To understand why this happens note that, under the JSQ policy, arrivals join the shortest queue, and thus shorter queues tend to become larger (due to arrivals) and larger queues tend to become shorter (due to services). Moreover, due to the nature of the abandonments, large queues have a larger number of abandonments than short queues. Thus, both the JSQ policy and the abandonments push the queue lengths towards becoming equal. This is even more effective in heavy traffic or in overload, where the queue lengths are inherently large, which makes the random perturbations small in comparison.\\

{\it Comparisson with the SSC of JSQ without abandonments:} The SSC result given in Theorem~\ref{LEM: JSQ_SSC} is weaker than the SSC result for the classic JSQ system without abandonments \cite{hurtado2020transform}, where the authors show that all moments of the perpendicular component $\|\q_{\perp}\|$ are uniformly bounded by a (possibly moment-dependent) constant. 
Specifically, the authors show that there exists a neighbourhood of zero, such that the MGF of the unscaled $\|\q_{\perp}\|$ is uniformly bounded in that interval, for any $\gamma>0$. To prove this, they use arguments as in \cite{hajek1982hitting}, which require two conditions: a negative drift condition, and a (stochastically) bounded drift assumption, which requires the drift to be bounded by an independent geometric random variable. In the JSQ-A system, due to the state-dependent abandonment, the bounded drift assumption is not satisfied. Therefore, proving a bound on the MGF of the unscaled $\|\q_{\perp}\|$ when we have abandonments turns out to be technically challenging, and so we only provide a weaker form of SSC. We observe that, in order to characterize the limiting joint distribution of JSQ-A, one does not need to derive a bound on the MGF of $\|\q_{\perp}\|$ in a neighbourhood of zero. In this work, we show that our weaker SSC result given in Theorem \ref{LEM: JSQ_SSC}, in combination with the boundedness of the MGF of the scaled queue length (presented later in Lemma \ref{lem: jsq_mgfexist}), is enough to establish Theorem \ref{THM: JSQ_LIMIT_DIS}.

\def\jsqlimitdis{\ref{THM: JSQ_LIMIT_DIS}}

\section{Proof of Theorem \jsqlimitdis} 
\label{sec: proofs} 

For clarity purposes, we condense all the tedious details and calculations into lemmas, and relegate their proofs to Appendix \ref{app: jsq}. This allows us to focus the presentation on the high level ideas of the proof without getting lost in the technicalities. Moreover, the proof of these lemmas are oftentimes based on the behavior of an equivalent SSQ-A system. Thus, we have included the corresponding proofs for the SSQ-A system in Appendix \ref{app: ssq_results}.\\ 

Before getting into the proof of Theorem \ref{THM: JSQ_LIMIT_DIS}, we start by presenting a technical Lemma which states that the pointwise convergence of an MGF in an interval around zero implies convergence in distribution and convergence of moments. \\

\begin{lemma}[Implications of Convergence of MGF]
\label{COR: MGF_CONVERGENCE}
Suppose $\{X_n\}_{n\in \mathbb N}$ is a sequence of random variables, and $M_n(s) = \EE[e^{sX_n}]$ are their corresponding MGFs. Suppose the following two conditions hold:
\begin{itemize}
    \item[(a)] There exists $s_0>0$, such that $M_n(s) < \infty$ for all $s\in [-s_0,s_0]$ and $n \in \mathbb N$.
    \item[(b)] The sequence of MGF $\{M_n\}_{n\in \mathbb N}$ converges point wise to $M(\cdot)$ for $s\in [-s_0,s_0]$, where $M(\cdot)$ is MGF of a random variable $X$.
\end{itemize}
Then, $X_{n} $ converges weakly to $X$ and all the moments of $X_n$ converge to the corresponding moments of $X$, i.e.,
\begin{align*}
    X_n \stackrel{d}{\rightarrow} X, && \lim_{n\rightarrow \infty} \EE[X_n^m] = \EE[X^m], \quad \forall m\in \mathbb N.
\end{align*}
\end{lemma}
The proof of Lemma \ref{COR: MGF_CONVERGENCE} relies on results regarding tight sequences of probability measures, as described in \cite{Bill86}. However, for the sake of completeness, we provide its proof in Appendix \ref{app: mgf_convergence}.\\

It is important to note that Lemma \ref{COR: MGF_CONVERGENCE} does not hold if we replace the MGFs with characteristic functions. While convergence of characteristic functions implies convergence in distribution, which in turn implies convergence of expectations for bounded continuous functions (by Portmanteau's theorem), convergence of moments is not guaranteed since moments are unbounded functions.
Fortunately, as shown in Lemma \ref{COR: MGF_CONVERGENCE}, convergence of MGFs not only implies convergence in distribution, but also convergence of moments. Moreover, convergence of MGFs in an interval around zero can be extended to the convergence of characteristic functions through analytic continuation. Therefore, convergence of MGFs provides a stronger mathematical argument.
However, proving convergence of MGFs can be more challenging as we must first establish the existence of MGFs, which is not always guaranteed unlike for characteristic functions. In this paper, we establish the existence and uniform boundedness of appropriately scaled and centered queue length MGFs, which in turn implies the non-divergence of the limit. These results are presented in the following subsection.

\subsection{Boundedness of the MGF}

To prove Theorem \ref{THM: JSQ_LIMIT_DIS}, the first step is to establish bounds on the MGFs of the steady-state queue lengths. This is necessary because the limiting distributions in Theorem \ref{THM: JSQ_LIMIT_DIS} are derived as solutions of differential equations on the MGFs of the steady-state queue lengths. The MGF bounds are initially established for an SSQ-A system in Subsection \ref{sec:MGF_bounds_SSQ-A}, and then extended to the JSQ-A system in Subsection \ref{sec:MGF_bounds_JSQ-A}.\\

\begin{remark}
    If the limiting distribution is the only quantity of interest (ignoring the convergence of moments), one could just work with characteristic functions, and establish a differential equation in terms of characteristic functions. However, even though working with characteristic functions might simplify some arguments, one would still need to prove strong enough moment bounds on the steady state queue length, which in turn bring similar challenges as the ones we encountered proving the boundedness of the MGF. 
\end{remark}

\subsubsection{Boundedness of the MGF for SSQ-A} \label{sec:MGF_bounds_SSQ-A}
We have the following result for SSQ-A.\\

\begin{lemma}
\label{lem: ssq_mgfexist}
Consider the SSQ-A setting (i.e., JSQ-A with $n=1$). There exists constants $M_0 >0, \phi_0 \in (0,1)$ independent of $\gamma$ and $\gamma_0 \in (0,1)$, such that, the following hold
\begin{enumerate}[label=(\alph*), ref=\ref{lem: ssq_mgfexist}.\alph*]
    \item \label{lem: ssq_mgfexist_a}  For any  $0<\phi < \phi_0$ and $\gamma\in(0,1)$,
    \begin{equation*}
        \MM{q} \leq M_0\exp\Bigg(\frac{\phi \nu_\gamma^+}{\sqrt{\gamma}}\Bigg),
    \end{equation*}
    for any $\nu_\gamma \in \mathbb R$, where $\nu_\gamma^+ = \max\{\nu_\gamma,0\}$.
    \item \label{lem: ssq_mgfexist_b} For any  $0<\phi < \phi_0$ and $\gamma\in(0,\gamma_0)$, we have
    \begin{equation*}
        M_{q}^\gamma(-\phi) \leq  M_0 \exp\Bigg(-\frac{\phi \nu_\gamma}{\sqrt{\gamma}}\Bigg).
    \end{equation*}
    \item \label{lem: ssq_mgfexist_c} Suppose $ |\nu_\gamma +C_f\gamma^{ \alpha }| \leq \gamma^{\alpha + \beta}$, where $ \alpha\in \lfsmall{0,1/2}$, $\beta >0$, and $C_f >0$ are constants. Then, for any $0<\phi<\phi_0$ and $\gamma\in (0,1)$, we have
    \begin{align*}
        \EE\left[e^{\gamma^\alpha \phi q}\right] \leq M_0.
    \end{align*}
\end{enumerate}
\end{lemma}
The complete proof is given in Appendix \ref{app: ssq_mgf}, but we provide a proof sketch below.\\

Lemma \ref{lem: ssq_mgfexist_a} and Lemma \ref{lem: ssq_mgfexist_b} are established in three steps, as follows.
\begin{enumerate}
    \item The first step follows by coupling the queue length process $\{q(t)\}_{t=0}^\infty$ with two different queue length processes. For the first one, the maximum number of abandonments is bounded, which results in fewer abandonments and thus, in higher queue lengths than in the original process. This results in a queue length process which path-wise dominates $\{q(t)\}_{t=0}^\infty$. Similarly, for the second one, the minimum number of abandonments is bounded, which results in more abandonments and thus, in smaller queue lengths than in the original queue length process. This yields a queue length process which is path-wise dominated by $\{q(t)\}_{t=0}^\infty$. The key idea is to properly choose the bound on the abandonments for both processes to precisely compensate for the drift due to arrivals/services ($\nu_\gamma$), and to ensure that the queue length process does not deviate too much from $\nu_\gamma/\gamma$.
    \item In the second step, we establish the existence of the exponential Lyapunov function. We consider the function $\EE[e^{\eta \phi q(t)}]$ (for appropriately chosen $\phi$ and scaling factor $\eta$) and show that its drift is negative outside a bounded set. This implies that the MGF of the original, as well as of the coupled processes, exist in steady state for any value of $\gamma$.
    \item Finally, in the third step, we use an exponential Lyapunov function, and set its drift to zero in steady state. Here, we use the existence of the exponential Lyapunov function, established in the previous step. Then, we do a second order approximation using the Taylor expansion of the exponential function, and use the first order drift of $q(t)$ (i.e., $\EE[\Delta q(t)|q(t)]$) to create a bound on the MGF. 
\end{enumerate}
The proof of Lemma  \ref{lem: ssq_mgfexist_c} is simpler, as we only need to create a coupled process for which there are no abandonments, which results in a queue length process that dominates the original one. Now, for this new process, we perform a similar drift analysis as for the previous parts of the lemma. 

\subsubsection{Boundedness of the MGF for JSQ-A} \label{sec:MGF_bounds_JSQ-A}

We now generalize Lemma \ref{lem: ssq_mgfexist} to the JSQ-A system with $n>1$.\\

\begin{lemma}
\label{lem: jsq_mgfexist}
Consider the JSQ-A system as described in the Section \ref{sec: model}. Then, there exists constants $M_1 >0, \phi_1 \in (0,1)$, and $\gamma_1 \in (0,1)$, such that the following hold.
\begin{enumerate}[label=(\alph*), ref=\ref{lem: jsq_mgfexist}.\alph*]
    \item \label{lem: jsq_mgfexist_a}  For any  $0<\phi < \phi_1$ and $\gamma\in(0,1)$,
    \begin{align*}
      \sum_{i=1}^n M_{nq_i}^\gamma (\phi) \leq nM_1 \exp\Bigg(\frac{\phi \nu_\gamma^+}{\sqrt{\gamma}}\Bigg), &&    M_{\langle {\bf q},{\bf 1} \rangle}^\gamma (\phi ) \leq M_1\exp\Bigg(\frac{\phi \nu_\gamma^+}{\sqrt{\gamma}}\Bigg),
    \end{align*}
    for any $\nu_\gamma$, where $\nu_\gamma^+ = \max\{\nu_\gamma,0\}$.
    \item \label{lem: jsq_mgfexist_b} For any  $0<\phi < \phi_1$ and $\gamma\in(0,\gamma_1)$,
    \begin{align*}
        \sum_{i=1}^n M_{nq_i}^\gamma (-\phi) \leq nM_1 \exp\Bigg(-\frac{\phi \nu_\gamma}{\sqrt{\gamma}}\Bigg), && M_{\langle {\bf q},{\bf 1} \rangle}^\gamma (-\phi ) \leq  M_1 \exp\Bigg(-\frac{\phi \nu_\gamma}{\sqrt{\gamma}}\Bigg).
    \end{align*}
    \item \label{lem: jsq_mgfexist_c} Suppose $ |\nu_\gamma +C_f\gamma^{ \alpha }| \leq \gamma^{\alpha + \beta}$, where $ \alpha\in \lfsmall{0,1/2}$, $\beta >0$, and $C_f >0$ are constants. Then, for any $0<\phi<\phi_1$ and $\gamma\in (0,1)$,
    \begin{align*}
        \EE\left[e^{\gamma^\alpha \phi \langle {\bf q},{\bf 1}\rangle }\right] \leq M_1.
    \end{align*}
\end{enumerate}
\end{lemma}

For JSQ-A, the proof is more challenging than for SSQ-A, as we cannot create a coupled process for JSQ-A that pathwise dominates the original queue length vector, as the load balancing process is state dependent. Thus, we first show the existence of the MGF for JSQ-A by using the bound for the MGF for SSQ-A. Afterwards, we do a `one-step coupling' (instead of pathwise coupling) on the abandonments to replicate the arguments from SSQ-A. Instead of working directly with the MGF of the total queue length $\EE\big[e^{\eta \phi \vsum{\q}}\big]$ (where $\eta$ is a scaling factor that depends on the regime), we work with the sum of the MGFs of the individual scaled queue lengths $\sum_{i=1}^n \EE\big[e^{\eta \phi q_i}\big]$, and then use Jensen's inequality to get a bound on $\EE\big[e^{\frac{1}{n}\eta \phi \vsum{\q}}\big]$. Due to the weaker form of SSC for JSQ-A presented in this paper, we cannot directly claim that $\EE\big[e^{\frac{1}{n}\eta \phi \vsum{\q}}\big] \approx \sum_{i=1}^n \EE\big[e^{\eta \phi q_i}\big]$. Such a claim can be made only after proving a bound for $\sum_{i=1}^n \EE\big[e^{\eta \phi q_i}\big]$ and then combining it with the SSC in Theorem \ref{LEM: JSQ_SSC}. In order to prove the bound for $\sum_{i=1}^n \EE\big[e^{\eta \phi q_i}\big]$, we exploit the fact that, under the JSQ policy, the arrivals join a shortest queue. Thus, this gives sufficiently large negative drift for the exponential Lyapunov function, i.e., for $ \sum_{i=1}^n \EE\big[e^{\eta \phi q_i}\big]$. Further, since we do not have pathwise coupling, we have to use additional second order drift arguments to prove Lemma \ref{lem: jsq_mgfexist} for JSQ-A. The complete proof is given in Appendix \ref{app: jsq_mgfexist}. \\

\begin{remark}
    Lemma \ref{lem: jsq_mgfexist} provides a bound on the MGF of the queue length process in all three regimes. When the system is in classic heavy traffic, Lemma \ref{lem: jsq_mgfexist_c} implies that the MGF of the scaled total queue length is bounded by a constant. Similarly, since $\nu_\gamma$ is of order $\sqrt{\gamma}$ in the critical heavy traffic regime, Lemma \ref{lem: jsq_mgfexist_a} implies that scaled total queue length is bounded by a constant in that regime. For the heavily overloaded regime, using the results in lemmas \ref{lem: jsq_mgfexist_a} and \ref{lem: jsq_mgfexist_b}, we can get that the MGF of the appropriately centered and scaled queue length is bounded. In the heavily overloaded regime, we observe that the total queue length concentrates around $\nu_\gamma/\gamma$ only for small values of $\gamma$, and thus we need $\gamma$ to be small enough in Lemma \ref{lem: jsq_mgfexist_b}. Further, from Lemma \ref{lem: jsq_mgfexist}, we can already conclude that, in heavily overloaded regime, all the queues in JSQ-A concentrate around $\nu_\gamma/n\gamma$.\\
\end{remark}

\begin{remark}
    In Lemma \ref{lem: jsq_mgfexist_a}, the exponential term in the RHS blows up only when the system is in the heavily overloaded regime, for which $\nu_\gamma>0$, and so we have $\nu_\gamma^+$ in Lemma \ref{lem: jsq_mgfexist_a}. Similarly, the bound in Lemma \ref{lem: jsq_mgfexist_b} is meaningful only for $\nu_\gamma>0$, as otherwise $ \sum_{i=1}^n M_{nq_i}^\gamma (-\phi)$ is clearly less than 1 while the term in the RHS is larger than 1.\\
\end{remark}

\subsection{Proofs for each case in Theorem \ref{THM: JSQ_LIMIT_DIS}} 
\label{sec: jsq_proofsketch}

\subsubsection{Classic Heavy Traffic: Proof of Theorem \ref{thm: jsq_fast}} \label{sec: jsq_proofsketch1}

In this section, we provide the proof of Theorem \ref{thm: jsq_fast}. Before presenting the proof, we present two crucial technical lemmas.\\

\begin{lemma}[Classic-HT: Second Order Approximation]
\label{lem: jsq_approximation_fast}
Under the same assumptions as in Theorem \ref{thm: jsq_fast}, we have the following results.
\begin{enumerate}[label=(\alph*), ref=\ref{lem: jsq_approximation_fast}.\alph*]
    \item \label{lem: jsq_approximation_fast_a} For any $\gamma \in (0,1)$ and for any $\phi \in \mathbb R$,
    \begin{align*}
      \EE\left[{e^{-\gamma^\alpha \phi  \vsum{u} }}\right]-\phi\gamma^\alpha \nu_\gamma -1 =: \overline{\mathcal{E}}_f^u(\gamma,\phi), &&  \big| \overline{\mathcal{E}}_f^u(\gamma,\phi)\big| \leq \overline K^u_f \phi^2 e^{n|\phi|A}  \gamma^{\min\{3\alpha,1\}}.
    \end{align*}
    \item \label{lem: jsq_approximation_fast_b} Suppose $c$ is a random variable distributed as $c(t) = a(t) - \vsum{s(t)}$, with $|c|\leq nA$, $\EE[c]= \nu_\gamma $ and $\EE[c^2] = \sigma^2_\gamma +\nu_\gamma^2$. Then, for any $\phi \in \mathbb R$,
    \begin{align*}
       \EE\left[{e^{\gamma^\alpha \phi  c }}\right]   - \frac{\gamma^{2\alpha}\phi^2 (\sigma^2_\gamma+\nu_\gamma^2)}{2} - \gamma^\alpha \phi \nu_\gamma -  1 =: \overline{\mathcal{E}}_f^c(\gamma,\phi), && \big| \overline{\mathcal{E}}_f^c(\gamma,\phi) \big| \leq \overline  K_{f}^c |\phi|^3e^{n|\phi|A} \gamma^{3\alpha}.
    \end{align*}
    \item \label{lem: jsq_approximation_fast_c} Let $\phi_1$ be a constant as  in Lemma \ref{lem: jsq_mgfexist_a}. Then, for any $\phi < \phi_1/2$,
    \begin{align*}
        \EE\left[{e^{\gamma^\alpha \phi  \langle {\bf 1},\q-{\bf d} \rangle }} \right] - \EE\left[{e^{\gamma^\alpha \phi  \vsum{q}}}\right] =: \overline{\mathcal{E}}_f^d(\gamma,\phi), && \big| \overline{\mathcal{E}}_f^d(\gamma,\phi) \big| \leq \overline  K_f^d |\phi| \gamma,
    \end{align*}
\end{enumerate}
where $\overline  K_f^u,\overline  K_f^c$ and $\overline  K_f^d$ are constant independent of $\gamma$ and $\phi$.\\
\end{lemma}

Lemma \ref{lem: jsq_approximation_fast} provides the second order approximation arguments required for the proof of Theorem \ref{thm: jsq_fast}. The proof is based on the second order Taylor expansion of the exponential function and on the bound for the MGF presented in Lemma \ref{lem: jsq_mgfexist}, and it is provided in Appendix \ref{app: jsq_fast}.\\

\begin{remark}
    Lemma \ref{lem: jsq_approximation_fast_c} shows that, in the classic heavy traffic regime, the term
    \[ \frac{1}{\gamma^{2\alpha}} \Big| \EE\Big[{e^{\gamma^\alpha \phi  \langle {\bf 1}, \q -{\bf d} \rangle }}\Big]-\EE\Big[{e^{\gamma^\alpha \phi \vsum{q} }}\Big] \Big| \]
    is of order $\gamma^{1-2\alpha}$, which is negligible for small values of $\gamma$ since $\alpha < 1/2$. This implies that we can replace the term with abandonments, $\EE\big[{e^{\gamma^\alpha \phi  \langle {\bf 1}, \q -{\bf d} \rangle }}\big]$, with a term without abandonments $\EE\big[{e^{\gamma^\alpha \phi  \langle {\bf 1}, \q \rangle }}\big]$, when using the transform method. This in turn implies that the abandonments do not affect the limiting distribution in the classic heavy traffic regime. Therefore, under the scaling factor $|1-\rho_\gamma| = \Omega (\gamma^\alpha)$, the system behaves like it is in classic heavy traffic without abandonments, and the corresponding limiting distribution is exponential.\\
\end{remark}

\begin{lemma}[Classic-HT: Implications of SSC]
\label{LEM: JSQ_SSC_fast}
Under the same assumptions as in Theorem \ref{thm: jsq_fast}, we have the following results.
\begin{enumerate}[label=(\alph*), ref=\ref{LEM: JSQ_SSC_fast}.\alph*]
\item \label{LEM: JSQ_SSC_fast_a} For any $\gamma \in (0,1)$ and for any $\phi \in \big(-\infty,\phi_1/4\big)$,
\begin{align*}
    \EE\left[\left(e^{\gamma^\alpha \phi \vsum{q^+}} -1\right)\left(e^{-\gamma^\alpha \phi \vsum{u} }-1\right)\right] =: \overline{\mathcal{E}}_f^{qu}(\gamma,\phi), &&|\overline{\mathcal{E}}_f^{qu}(\gamma,\phi)| \leq  \overline{K}_f^{qu}\phi^2 \gamma^{\frac{9\alpha}{4}}e^{n|\phi|A}.
\end{align*}
\item \label{LEM: JSQ_SSC_fast_b} Suppose $ \boldsymbol \Phi = \{ \mathbf x\in \mathbb R^n: 2\|{\bf x}\| \leq \phi_1 \}$, then for any $\boldsymbol \phi \in \boldsymbol \Phi$, with $\phi = \langle {\bf 1}, \boldsymbol \phi \rangle$,
\begin{equation*}
    \limg \EE\left[{e^{\gamma^\alpha \langle\boldsymbol \phi, \q \rangle }}\right] = \limg \EE\left[{e^{\frac{1}{n} \gamma^\alpha\phi \vsum{q} }}\right].
\end{equation*}
\end{enumerate}
\end{lemma}

Lemma \ref{LEM: JSQ_SSC_fast} provides the two key implications of the SSC which are required to prove Theorem \ref{thm: jsq_fast}. The proof is based on the second order Taylor expansion of the exponential function and the SSC result, and it is provided in Appendix \ref{app: jsq_fast}.\\ 

\begin{remark}
    \label{rem: ssc_implications_fast}
    For the case of SSQ-A, the term 
    \[ \EE\Big[\left(e^{\gamma^\alpha \phi q^+} -1 \right) \left(e^{-\gamma^\alpha \phi u }-1\right)\Big] \]
    is exactly equal to zero. And from Lemma \ref{LEM: JSQ_SSC_fast_a}, we get that for JSQ-A, the term 
    \[ \frac{1}{\gamma^{2\alpha}} \left| \EE\left[\left(e^{\gamma^\alpha \phi \vsum{q^+}} -1\right)\left(e^{-\gamma^\alpha \phi \vsum{u} }-1\right)\right]\right| \]
    is at most of order $\gamma^{\frac{\alpha}{4}}$, which is approximately zero for small values of $\gamma$. This allows us to use similar steps for JSQ-A as that for SSQ-A while implementing the transform method. \\
\end{remark}
Next, we present the proof of Theorem \ref{thm: jsq_fast}.

\begin{proof}[Proof of Theorem \ref{thm: jsq_fast}: ]

 Lemma \ref{lem: jsq_mgfexist_c} implies that the MGF of the steady state queue length is bounded, for any $\phi<\phi_1$. Then, Lemma \ref{LEM: JSQ_SSC_fast_a} implies that, for any $\phi<\phi_1/4$, we have
  \begin{align}
\label{eq: jsq_fast_starteq_proof}
 \EE\left[{e^{-\gamma^\alpha \phi \vsum{u}}}\right]-1 &= \EE\left[{e^{\gamma^\alpha \phi \langle {\bf 1}, \q^+-{\bf u} \rangle }}\right] - \EE\left[{e^{\gamma^\alpha \phi \langle {\bf 1}, \q^+ \rangle}}\right] - \overline{\mathcal{E}}_f^{qu}(\gamma,\phi) \nonumber \allowdisplaybreaks\\
 &\stackrel{(a)}{=} \EE\left[{e^{\gamma^\alpha \phi \langle {\bf 1}, \q + {\bf c}-{\bf d} \rangle }}\right]  - \EE\left[{e^{\gamma^\alpha \phi  \vsum{q} }}\right] -\overline{\mathcal{E}}_f^{qu}(\gamma,\phi)  \nonumber \allowdisplaybreaks\\
 &\stackrel{(b)}{=} \EE\left[{e^{\gamma^\alpha \phi \langle {\bf 1}, \q -{\bf d} \rangle }}\right] \EE\left[{e^{\gamma^\alpha \phi c }}\right] - \EE\left[{e^{\gamma^\alpha\phi \vsum{q} }}\right]  -\overline{\mathcal{E}}_f^{qu}(\gamma,\phi)\nonumber\allowdisplaybreaks\\
 &= \EE\left[{e^{\gamma^\alpha \phi  \vsum{q} }}\right] \left(\EE\left[{e^{\gamma^\alpha \phi c }}\right] -1 \right) \nonumber\\
 & \ \ \ \ + \left(\EE\left[{e^{\gamma^\alpha \phi \langle {\bf 1}, {\bf q} - {\bf d} \rangle }}\right] -\EE\left[{e^{\gamma^\alpha \phi \vsum{q} }} \right] \right) \EE\left[ {e^{\gamma^\alpha \phi  c }}\right] - \overline{\mathcal{E}}_f^{qu}(\gamma,\phi),
\end{align}
where (a) follows from $\q$ and $\q^+$ having the same distribution, and $\q^+-{\bf u} = \q+{\bf c}- {\bf d}$; and (b) follows from $c$ being independent of ${\bf q}$ and ${\bf d}$. 
Now, from Lemma \ref{lem: jsq_approximation_fast}, it follows that 
\begin{align}
\label{eq: jsq_prelimitmgfeq}
    \left( \frac{1}{\gamma^\alpha}\nu_\gamma + \frac{1}{2}\phi (\sigma^2_\gamma + \nu_\gamma^2) \right) \EE\left[{e^{\gamma^\alpha \phi  \vsum{q} }}\right] - \frac{1}{\gamma^\alpha} \nu_\gamma = \frac{1}{\phi \gamma^{2\alpha}}  \overline{\mathcal{E}}_f(\gamma,\phi),
\end{align}
where (using the terminology in Lemma \ref{lem: jsq_approximation_fast})
\begin{align*}
      \overline{\mathcal{E}}_f(\gamma,\phi) = -\EE\left[{e^{\gamma^\alpha \phi \vsum{q} }}\right] \overline{\mathcal{E}}_f^c(\gamma,\phi) +  \overline{\mathcal{E}}_f^u(\gamma,\phi)  -\EE\left[{e^{\gamma^\alpha \phi c }}\right] \overline{\mathcal{E}}_f^d(\gamma,\phi) + \overline{\mathcal{E}}_f^{qu}(\gamma,\phi).
\end{align*}
Using the bounds provided in Lemma \ref{lem: jsq_approximation_fast}, Lemma \ref{lem: jsq_mgfexist_c}, and the fact that $|\phi|\leq 1$, we get that
\begin{align*}
     \frac{1}{|\phi| \gamma^{2\alpha}} |\overline{\mathcal{E}}_f(\gamma,\phi)| &\leq e^{nA}\big(M_1 \overline K_{f}^c  +\overline  K^u_f   +\overline K_f^d \big)  \gamma^{\min\{\alpha,1-2\alpha\}} + \overline{K}_f^{qu} \gamma^{\frac{\alpha}{4}}e^{nA}\\
     &\leq e^{nA}\big(M_1 \overline K_{f}^c  +\overline  K^u_f   +\overline K_f^d + \overline{K}_f^{qu} \big) \gamma^{\min\{\alpha/4,1-2\alpha\}},
\end{align*}
and thus
\begin{align*}
    \limg \frac{1}{|\phi| \gamma^{2\alpha}} |\mathcal{E}_f(\gamma,\phi)| =0.
\end{align*}
 Moreover, we have $\limg \frac{1}{\gamma^\alpha} \nu_\gamma = C_f$ and $\limg \sigma^2_\gamma  + \nu_\gamma^2 = \sigma^2$. Then, by taking the limit as $\gamma \rightarrow 0$, for any $\phi < \phi_1/4$ in Eq. \eqref{eq: jsq_prelimitmgfeq}, we obtain
\begin{align*}
   \limg \EE\left[{e^{\gamma^\alpha \phi  \vsum{q} }}\right] =  \frac{1}{ 1 - \phi \frac{\sigma^2}{2C_f}  }.
\end{align*}
Note that RHS is the MGF of an exponential random variable with mean $\frac{\sigma^2}{2C_f} $. Further, using Lemma \ref{LEM: JSQ_SSC_fast_b}, for $\boldsymbol\phi \in \boldsymbol \Phi$ with $\langle {\bf 1},\boldsymbol\phi \rangle = \phi$ and $|\phi|<\phi_1/4$, we obtain
\begin{align*}
   \limg \EE\left[{e^{\gamma^\alpha \langle\boldsymbol \phi, \q \rangle }}\right] = \limg \EE\left[{e^{\frac{1}{n} \gamma^\alpha\phi \vsum{q} }}\right] =  \frac{1}{ 1 - \phi \frac{\sigma^2}{2nC_f}  }.
\end{align*}
Finally, the result follows from Lemma \ref{COR: MGF_CONVERGENCE}.
\end{proof}

\subsubsection{Critical Heavy Traffic: Proof of Theorem \ref{thm: jsq_crit}} \label{sec: jsq_proofsketch2}

In this section, we provide the proof of Theorem \ref{thm: jsq_crit}. Recall that, for any random variable $X$, we defined $\MM{X} = \EE[e^{\sqrt{\gamma}\phi X}]$. Before presenting the proof, we present two crucial technical lemmas.\\

\begin{lemma}[Critical-HT: Second order approximations]
\label{lem: jsq_approximation_crit}
Under the same assumptions as in Theorem \ref{thm: jsq_crit}, we have the following results.
\begin{enumerate}[label=(\alph*), ref=\ref{lem: jsq_approximation_crit}.\alph*]
    \item \label{lem: jsq_approximation_crit_a} For any $\gamma \in (0,1)$ and for any $\phi \in \mathbb R$,
    \begin{align*}
       \MM{-\vsum{u}} +\phi\sqrt{\gamma} \EE[\vsum{u}] -1 =: \overline{\mathcal{E}}_c^u(\gamma,\phi), && \big| \overline{\mathcal{E}}_c^u(\gamma,\phi)\big| \leq \overline K_c^u \phi^2 e^{n|\phi|A}  \gamma^{\frac{3}{2}}.
    \end{align*}
    \item \label{lem: jsq_approximation_crit_b} Suppose $c$ is a random variable distributed as $c(t) = a(t) - s(t)$, with $|c|\leq A$, $\EE[c]= \nu_\gamma $ and $\EE[c^2] = \sigma^2_\gamma +\nu_\gamma^2$, for any $\phi \in \mathbb R$,
    \begin{align*}
      \MM{c}  - \frac{\gamma\phi^2 (\sigma^2_\gamma+\nu_\gamma^2)}{2} - \sqrt{\gamma} \phi \nu_\gamma -  1 =: \overline{\mathcal{E}}_c^c(\gamma,\phi), && \big| \overline{\mathcal{E}}_c^c(\gamma,\phi) \big| \leq \overline  K_{c}^c |\phi|^3e^{n|\phi|A} \gamma^{\frac{3}{2}}.
    \end{align*}
    \item \label{lem: jsq_approximation_crit_c} Let $\phi_1$ be a constant as given in Lemma \ref{lem: ssq_mgfexist_a}. Then, for any $\phi < \phi_1/2$,
    \begin{align*}
        \left(\MM{\langle {\bf 1}, {\bf q} - {\bf d} \rangle} - \MM{\vsum{q}}\right) \MM{c} +\phi \gamma \frac{d}{d\phi} \MM{\vsum{q}}  =: \overline{\mathcal{E}}_c^d(\gamma,\phi), && \big|\overline{\mathcal{E}}_c^d(\gamma,\phi) \big| \leq \overline K_c^d \phi^2 e^{n|\phi| A} \gamma^\frac{3}{2},
    \end{align*}
\end{enumerate}
where $\overline K_c^u, \overline K_c^c$, and $ \overline K_c^d$ are constants independent of $\gamma$ and $\phi$.\\
\end{lemma}

Lemma \ref{lem: jsq_approximation_crit} provides the second order approximations required to prove Theorem \ref{thm: jsq_crit}. The proof is similar to that of Lemma \ref{lem: jsq_approximation_fast}, and it is provided in Appendix \ref{app: jsq_crit}.\\

\begin{remark}
In Lemma \ref{lem: jsq_approximation_crit_c}, the term  $\frac{1}{\phi\gamma}\big(\MM{\langle {\bf 1}, {\bf q} - {\bf d} \rangle} - \MM{\vsum{q}}\big)$ is approximated by the derivative of $\MM{\vsum{q}}$ while $\MM{c}$ is approximated by $1$. This differs from the classic heavy traffic regime as, in that case, the corresponding term due to abandonment was approximated by $0$ (see Lemma \ref{lem: jsq_approximation_fast_c}). Due to this derivative term, we get a one-dimensional differential equation in critical heavy traffic regime, instead of a closed form expression as in classic heavy traffic regime. This is the primary reason that the limiting distribution transitions from an exponential in classic heavy traffic regime to a truncated normal in the critical heavy traffic regime.\\
\end{remark}

\begin{lemma}[Critical-HT: Implications of SSC]
\label{LEM: JSQ_SSC_crit}
Under the same assumptions as in Theorem \ref{thm: jsq_crit}, we have the following results,
\begin{enumerate}[label=(\alph*), ref=\ref{LEM: JSQ_SSC_crit}.\alph*]
\item \label{LEM: JSQ_SSC_crit_a} For any $\gamma \in (0,1)$ and for any $\phi \in \big(-\infty,\frac{\phi_1}{4}\big)$,
\begin{align*}
    \EE\left[\left(e^{\sqrt{\gamma} \phi \vsum{q^+}} -1\right)\Big(e^{-\sqrt{\gamma} \phi \vsum{u} }-1\Big)\right] =: \overline{\mathcal{E}}_c^{qu}(\gamma,\phi), &&|\overline{\mathcal{E}}_c^{qu}(\gamma,\phi)| \leq \overline{K}_c^{qu}\phi^2 \gamma^{\frac{9}{8}}e^{n|\phi|A}.
\end{align*}
\item \label{LEM: JSQ_SSC_crit_b} Suppose $\boldsymbol\phi \in \boldsymbol \Phi = \{\mathbf x\in \mathbb R^n: 2\|{\bf x}\| \leq \phi_1 \}$, then for any $\boldsymbol \phi \in \boldsymbol \Phi$, with $\phi = \langle {\bf 1}, \boldsymbol \phi \rangle$,
\begin{equation*}
    \limg \EE\left[{e^{\sqrt{\gamma} \langle\boldsymbol \phi, \q \rangle }}\right] = \limg \EE\left[{e^{\frac{1}{n} \sqrt{\gamma}\phi \vsum{q} }}\right].
\end{equation*}
\end{enumerate}
\end{lemma}

Lemma \ref{LEM: JSQ_SSC_crit} provides the two key implications of the SSC required to prove of Theorem \ref{thm: jsq_crit}. The implication of Lemma \ref{LEM: JSQ_SSC_crit_a} is similar to the one mentioned in Remark \ref{rem: ssc_implications_fast}. The proof is similar to that of Lemma \ref{LEM: JSQ_SSC_fast}, and is given in Appendix \ref{app: jsq_crit}.\\

Next, we present the proof of Theorem \ref{thm: jsq_crit}.

\begin{proof}[Proof of Theorem \ref{thm: jsq_crit}: ]  
First note that Lemma \ref{lem: jsq_mgfexist_a} implies that the MGF of the queue length process in steady state, i.e., $\MM{\vsum{\q}}$, is bounded for any $\phi<\phi_1$. Then, by using Lemma \ref{LEM: JSQ_SSC_crit_a} and following the same argument that yielded Eq. \eqref{eq: ssq_starteq}, for any $\phi<\phi_1/4$, we obtain
\begin{align}
\label{eq: jsq_crit_starteq_proof}
 \MM{-\vsum{u}}-1 &= \MM{\vsum{q}}\big(\MM{c} -1\big) + \left(\MM{\langle {\bf 1}, {\bf q} - {\bf d}\rangle}-\MM{\vsum{q}}\right)\MM{c} - \overline{\mathcal{E}}_c^{qu}(\gamma,\phi).
\end{align}
It follows that
\begin{align*}
    -\MM{\vsum{q}} \left(\frac{\gamma\phi^2 (\sigma^2_\gamma+\nu_\gamma^2)}{2} + \sqrt{\gamma} \phi \nu_\gamma\right) +\phi \gamma \frac{d}{d\phi} \MM{\vsum{q}}	-\sqrt{\gamma} \phi \EE[\vsum{u}] =  \overline{\mathcal{E}}_c(\gamma,\phi),
\end{align*}
where (using the terminology in Lemma \ref{lem: jsq_approximation_crit})
\begin{align*}
    \overline{\mathcal{E}}_c(\gamma,\phi) = \MM{\vsum{q}}  \overline{\mathcal{E}}_c^c(\gamma,\phi) +  \overline{\mathcal{E}}_c^d(\gamma,\phi) -  \overline{\mathcal{E}}_c^u(\gamma,\phi) + \overline{\mathcal{E}}_c^{qu}(\gamma,\phi).
\end{align*}
Dividing by $\gamma \phi$ on both sides, we get
\begin{align}
\label{eq: jsq_crit_diffeq}
	-\MM{\vsum{q}} \left(\frac{\phi (\sigma^2_\gamma+\nu_\gamma^2)}{2} +  \frac{\nu_\gamma}{\sqrt{\gamma}}\right) +\frac{d}{d\phi} \MM{\vsum{q}} - \frac{1}{\sqrt{\gamma}}\EE[\vsum{u}] = \frac{1}{\phi\gamma}\overline{\mathcal{E}}_c(\gamma,\phi).
\end{align}
We now define
\[ G^{\gamma}(\phi) := \exp\left( -\frac{\phi^2(\sigma^2_\gamma+\nu_\gamma^2)}{4} -  \frac{\phi\nu_\gamma}{\sqrt{\gamma}} \right). \]
Lemma \ref{lem: jsq_mgfexist_a} implies that
\begin{align*}
    \MM{\vsum{q}} \leq M_1 e^{\frac{\phi\nu_\gamma}{\sqrt{\gamma}}} \leq M_1e^{[\phi]^+ ([C_c]^++1)} \leq  M_1e^{[C_c]^++1} = \overline M_c.
\end{align*}
By using the results in Lemma \ref{lem: jsq_approximation_crit} and Lemma \ref{LEM: JSQ_SSC_crit_a}, we obtain
\begin{align*}
    \frac{1}{|\phi|\gamma}|\overline{\mathcal{E}}_c(\gamma,\phi)| \leq \overline{K}_c\gamma^{\frac{1}{8}}(|\phi|+\phi^2) e^{n|\phi|A},
\end{align*}
where $\overline{K}_c \geq \overline{M}_c\overline{K}_c^c+ \overline{K}_c^u + \overline{K}_c^d +\overline{K}_c^{qu}$. Then, 
 for any $\phi<\phi_1/4$, 
 \begin{align*}
     \limg\int_{-\infty}^\phi \frac{1}{|s|\gamma}|\overline{\mathcal{E}}_c(\gamma,s)| G^{\gamma}(s) ds =0.
 \end{align*}
   Furthermore, we have
   \[ \lim_{\gamma \rightarrow 0} \sigma^2_\gamma + \nu_\gamma^2 = \sigma^2 \qquad \text{ and } \qquad \lim_{\gamma \rightarrow 0} \frac{\nu_\gamma}{\sqrt{\gamma}} = C_c, \]
   and therefore
   \[ \lim_{\gamma \rightarrow 0} G^\gamma(\phi) = G(\phi) = \exp\left( -\frac{\phi^2\sigma^2_\gamma}{4} -  \phi C_c \right). \]
   Thus, we can solve the first order differential equation in Eq. \eqref{eq: jsq_crit_diffeq} by using the initial condition
   $ M_{\vsum{q}}^\gamma(0) G^\gamma(0) =1, $
   to get 
\begin{align}
\label{eq: jsq_crit_uterm}
    \lim_{\gamma \rightarrow 0} \frac{1}{\sqrt{\gamma}} \EE[\vsum{u}] = \left( \lim_{\gamma \rightarrow 0} \int_{-\infty}^0 G^\gamma(s)ds \right)^{-1} =  \left( \int_{-\infty}^0 G(s)ds \right)^{-1}.
\end{align}
It follows that, for $\phi < \phi_1/4$, we have
\begin{align*}
	\lim_{\gamma \rightarrow 0} \MM{\vsum{q}}& = G^{-1}(\phi) \left( \int_{-\infty}^0 G(s)ds \right)^{-1} \int_{-\infty}^\phi G(s)ds \\
	&= \exp\left( \frac{\phi^2\sigma^2}{4} + C_c\phi \right) \frac{\int_{-\infty}^\phi \exp\left( -\frac{s^2\sigma^2}{4} -  C_c s \right) ds}{\int_{-\infty}^0 \exp\left( -\frac{s^2\sigma^2}{4} - C_c s \right)ds}.
\end{align*}
Note that the RHS is the MGF of a truncated Gaussian random variable. Finally, by using Lemma~\ref{LEM: JSQ_SSC_crit_b} for $\boldsymbol\phi \in \boldsymbol \Phi$ with $\langle {\bf 1},\boldsymbol\phi \rangle = \phi$ and $|\phi|<\phi_1/4$, we obtain
\begin{align*}
   \limg \EE\left[{e^{\sqrt{\gamma} \langle\boldsymbol \phi, \q \rangle }} \right] = \limg \MM{\frac{1}{n}\vsum{q}}  =  \exp\left( \frac{1}{4n^2}\phi^2\sigma^2 +\frac{1}{n} C_c\phi \right) \frac{\int_{-\infty}^\phi \exp\left( -\frac{1}{4n^2}s^2\sigma^2 - \frac{1}{n} C_c s \right)ds}{\int_{-\infty}^0 \exp\left( -\frac{1}{4n^2}s^2\sigma^2 -  \frac{1}{n} C_c s\right)ds}.
\end{align*}
The result now follows by using Lemma \ref{COR: MGF_CONVERGENCE}.
\end{proof}

\begin{remark}
In critical heavy traffic regime, the term $\frac{1}{\sqrt{\gamma}}\EE[\vsum{u}]$ in Eq. \eqref{eq: jsq_crit_diffeq} acts as an unknown constant. As shown in Eq. \eqref{eq: jsq_crit_uterm}, to solve for the limiting value of this unknown constant, it is necessary to show that the differential equation holds for all negative values of $\phi$, instead of just for an interval around $0$. In the classic heavy traffic regime, the corresponding term, i.e., $\lim_{\gamma\rightarrow 0}\frac{1}{\gamma^\alpha}\EE[\vsum{u}]$, can be easily solved for, and thus, it acts as a `known' constant (see Lemma \ref{lem: jsq_approximation_fast_a}). Meanwhile, in the heavily overloaded regime, the term $\frac{1}{\sqrt{\gamma}}\EE[\vsum{u}]$ is arbitrarily small, so the limiting value is zero (see Lemma \ref{lem: jsq_approximation_slow_a}). As a result, in the classic heavy traffic regime or in the heavily overloaded regime, we only consider $\phi$ to be in an interval around $0$.
\end{remark}

\subsubsection{Heavily Overloaded: Proof of Theorem \ref{thm: jsq_slow}} \label{sec: jsq_proofsketch3}
In this section, we provide the proof of Theorem \ref{thm: jsq_slow}. For simplicity, we use $\bar{{\bf q}}= \q - \frac{\nu_\gamma}{n\gamma}{\bf 1}$ and $\bar{{\bf q}}^+= \q^+ - \frac{\nu_\gamma}{n\gamma}{\bf 1}$ to denote the centered steady state queue lengths. Before presenting the proof, we present two crucial technical lemmas.\\

\begin{lemma}[Heavily overloaded: Second order approximations]
\label{lem: jsq_approximation_slow}
 Under the same assumptions as in Theorem \ref{thm: jsq_slow}, there exists $\gamma_s'$ such that for any $\phi \in \big(\frac{-\phi_1}{2},\frac{\phi_1}{2}\big)$ and $\gamma<\gamma_s'$,
\begin{enumerate}[label=(\alph*), ref=\ref{lem: jsq_approximation_slow}.\alph*]
    \item \label{lem: jsq_approximation_slow_a} We have
    \begin{align*}
      \MM{-\vsum{u}} -1 =: \overline{\mathcal{E}}_s^{u}(\gamma,\phi), && \big| \overline{\mathcal{E}}_s^{u}(\gamma,\phi)\big| \leq \overline K_s^u \sqrt{\gamma}|\phi|  e^{-\frac{\phi_1 \nu_\gamma} {\sqrt{\gamma}}}.
    \end{align*}
    \item \label{lem: jsq_approximation_slow_b} We have 
    \begin{align*} 
         \Big(\MM{\langle {\bf 1}, {\bf \bar q} - {\bf d} \rangle} \MM{c} - \MM{\langle {\bf 1}, {\bf \bar q}  \rangle}\Big)& -\frac{\phi^2 \gamma\bar \sigma^2_{\gamma} }{2}\MM{\langle {\bf 1}, {\bf \bar q}  \rangle}  +\phi \gamma \frac{d}{d\phi} \MM{\langle {\bf 1}, {\bf \bar q} - {\bf d} \rangle} =: \overline{\mathcal{E}}_s^{d}(\gamma,\phi),
    \end{align*}
    and
    \[ \big| \overline{\mathcal{E}}_s^{d}(\gamma,\phi)\big|\leq \overline K_s^d |\phi|^2 \gamma^\frac{3}{2}, \]
\end{enumerate}
where $\bar \sigma^2_{\gamma} =\sigma^2_\gamma + \nu_\gamma(1-\gamma) $, and $\overline K_s^u$ and $\overline K_s^d$ are constants independent of $\gamma$ and $\phi$.\\
\end{lemma}

Lemma \ref{lem: jsq_approximation_slow} provides the second order approximations required to prove Theorem \ref{thm: jsq_slow}. Its proof is similar to that of Lemma \ref{lem: jsq_approximation_fast}, and it is provided in Appendix \ref{app: jsq_slow}.\\

\begin{lemma}[Heavily overloaded: Implications of SSC]
\label{LEM: JSQ_SSC_slow}
Under the same assumptions as in Theorem \ref{thm: jsq_slow}, we have the following results.
\begin{enumerate}[label=(\alph*), ref=\ref{LEM: JSQ_SSC_slow}.\alph*]
\item \label{LEM: JSQ_SSC_slow_a} For any $\gamma \in (0,1)$ and for any $\phi \in \big(-\frac{\phi_1}{4},\frac{\phi_1}{4}\big)$,
\begin{align*}
    \EE\left[\left(e^{\sqrt{\gamma} \phi \langle {\bf 1},\bar{ {\bf q}}^+ \rangle} -1\right)\Big(e^{-\sqrt{\gamma} \phi \vsum{u} }-1\Big) \right] =: \overline{\mathcal{E}}_f^{qu}(\gamma,\phi), &&|\overline{\mathcal{E}}_s^{qu}(\gamma,\phi)| \leq  \overline{K}_s^{qu} \sqrt{\gamma}|\phi| e^{-\frac{\phi_1 \nu_\gamma}{2\sqrt{\gamma}}}.
\end{align*}
\item \label{LEM: JSQ_SSC_slow_b} Suppose $\boldsymbol\phi \in \boldsymbol \Phi = \{\mathbf x\in \mathbb R^n: 2\|{\bf x}\| \leq \phi_1 \}$, then for any $\boldsymbol \phi \in \boldsymbol \Phi$, with $\phi = \langle {\bf 1}, \boldsymbol \phi \rangle$,
\begin{equation*}
    \limg \EE\left[{e^{\sqrt{\gamma} \langle\boldsymbol \phi, \bar{ {\bf q}} \rangle }}\right] = \limg \EE\left[{e^{\frac{1}{n} \sqrt{\gamma}\phi\langle {\bf 1},\bar{ {\bf q}} \rangle }}\right].
\end{equation*}
\end{enumerate}
\end{lemma}
The proof is provided in Appendix \ref{app: jsq_slow}.

\begin{proof}[Proof of Theorem \ref{thm: jsq_slow}:] \
First note that Lemma \ref{lem: jsq_mgfexist} implies that the MGF of the centered scaled steady state total queue length is bounded, for any $|\phi|<\phi_1/4$. Therefore, we have
\begin{equation*}
  \EE\left[ \left(e^{\sqrt{\gamma} \phi \langle {\bf 1},\bar{ {\bf q}}^+ \rangle} -1\right)\Big(e^{-\sqrt{\gamma} \phi \vsum{u} }-1\Big) \right] = \overline{\mathcal{E}}_s^{qu}(\gamma,\phi).
\end{equation*}
From Lemma \ref{LEM: JSQ_SSC_slow_a}, we have that for any $\phi \in \big(-\frac{\phi_1}{4},\frac{\phi_1}{4}\big)$,
\begin{equation*}
    \MM{\langle {\bf 1},\bar{ {\bf q}} -{\bf d} \rangle} \MM{c} - \MM{\langle {\bf 1},\bar{ {\bf q}} \rangle} = \big(\MM{-\vsum{u}}-1 \big) + \overline{\mathcal{E}}_s^{qu}(\gamma,\phi).
\end{equation*}
Moreover, from Lemma \ref{lem: jsq_approximation_slow}, we have
\begin{align*}
     \MM{\langle {\bf 1},\bar{ {\bf q}} -{\bf d} \rangle} \MM{c} - \MM{\langle {\bf 1},\bar{ {\bf q}} \rangle} &= \frac{\phi^2 \gamma\bar \sigma^2_\gamma}{2}\MM{\langle {\bf 1},\bar{ {\bf q}} \rangle}  - \phi \gamma \frac{d}{d\phi} \MM{\langle {\bf 1},\bar{ {\bf q}} \rangle} + \overline{\mathcal{E}}_s^d(\gamma,\phi) \\
    \Big(\MM{-\vsum{u}}-1 \Big) & = \overline{\mathcal{E}}_s^u(\gamma,\phi)
\end{align*}
It follows that 
\begin{align*}
    \frac{1}{2} \phi \bar \sigma^2_\gamma\MM{\langle {\bf 1},\bar{ {\bf q}} \rangle}  - \frac{d}{d\phi} \MM{\langle {\bf 1},\bar{ {\bf q}} \rangle} =  \frac{1}{\phi \gamma} \overline{\mathcal{E}}_s(\gamma,\phi),
\end{align*}
where (using the terminology in Lemma \ref{lem: jsq_approximation_slow})
\begin{align*}
    \overline{\mathcal{E}}_s(\gamma,\phi) :=  \overline{\mathcal{E}}_s^u(\gamma,\phi) -\overline{\mathcal{E}}_s^d(\gamma,\phi)+\overline{\mathcal{E}}_s^{qu}(\gamma,\phi).
\end{align*}
From Lemma \ref{lem: jsq_approximation_slow} and Lemma \ref{LEM: JSQ_SSC_slow_a}, we have
\begin{align*}
    \frac{1}{|\phi| \gamma} |\overline{\mathcal{E}}_s(\gamma,\phi)| \leq \overline{K}_s^d |\phi| \gamma^\frac{1}{2} + (\overline{K}_s^{qu} + \overline{K}_s^{u}) \frac{1}{\sqrt{\gamma}} e^{-\frac{\phi_1 \nu_\gamma}{2\sqrt{\gamma}}}.
\end{align*}
Then, by redefining $G^{\gamma}(\phi) = \exp\left(-\frac{\phi^2 \bar \sigma^2_\gamma}{4}\right)$, for any $\phi \in \left( -\frac{\phi_1}{4}, \frac{\phi_1}{4} \right)$, we have
\begin{align*}
 \limg \int_{0}^\phi \frac{1}{|s| \gamma} |\overline{\mathcal{E}}_s(s,\phi)| G^{\gamma}(s)ds 
 &= 0.
\end{align*}
Thus, we get
\begin{equation*}
    \limg  \MM{\langle {\bf 1},\bar{ {\bf q}} \rangle} = \limg \big(G^{\gamma}(\phi)\big)^{-1} = \exp\left( \frac{\phi^2\bar \sigma^2}{4}\right),
\end{equation*}
where $\bar \sigma^2 = \limg \sigma^2_\gamma + \nu_\gamma$. Note that the RHS is the MGF of a zero mean Gaussian random variable. This completes the proof of Theorem \ref{thm: jsq_slow}. Finally, by using Lemma \ref{LEM: JSQ_SSC_slow_b} for $\boldsymbol\phi \in \boldsymbol \Phi$, with $\langle {\bf 1},\boldsymbol\phi \rangle = \phi$ and $|\phi|<\phi_1/4$, we get
\begin{align*}
   \limg \EE\left[{e^{\sqrt{\gamma} \langle\boldsymbol \phi, \bar{ {\bf q}} \rangle }} \right] = \limg \EE\left[{e^{\frac{1}{n} \sqrt{\gamma}\phi\langle {\bf 1},\bar{ {\bf q}} \rangle }}\right]  =  \exp\left( \frac{1}{4n^2}\phi^2\bar \sigma^2\right)
\end{align*}
The result follows from Lemma \ref{COR: MGF_CONVERGENCE}.
\end{proof}

\def\jsqssc{\ref{LEM: JSQ_SSC}}

\section{Proof of Theorem \jsqssc}
\label{sec: jsq_ssc_proof}

We present all the major steps of the proof here and, as for the proof of Theorem \ref{THM: JSQ_LIMIT_DIS}, we defer the tedious algebraic calculations (condensed into claims) to Appendix~\ref{app: jsq_ssc}.

The first step of the proof is to couple each of the individual queues of the JSQ-A system with a SSQ-A system so that all the moments of the steady state number of abandonments are bounded. Mathematically, we prove the following claim.

\begin{claim}
\label{clm: jsq_ssc_abandonment}
There exists a sequence of constants $\{E_m\}_{m\geq 1}$ such that 
\begin{equation*}
    \gamma^m\EE[\langle \q, {\bf 1}\rangle^m] \leq E_m, \ \ \ \ \forall m\geq 1.
\end{equation*}
\end{claim}
Note that this also implies that $\EE[\langle \mathbf d, {\bf 1}\rangle^m]$ are bounded by a constant, irrespective of the value of $\gamma$.
Next, we define the following drift quantities,
\begin{align*}
   \Delta\| \q  \|^2 := \|\q^+\|^2-\|\q \|^2, &&  \Delta \langle\q ,{\bf 1} \rangle^2 := \langle\q^+,{\bf 1} \rangle^2 - \langle\q ,{\bf 1} \rangle^2, && \Delta \|\q_{\perp} \|^2 := \|\q_{\perp}^+\|^2 -\|\q_{\perp} \|^2.
\end{align*}
We have the following bounds on the drift $\Delta\| \q  \|^2$ and $\Delta \langle\q ,{\bf 1} \rangle^2$.\\

\begin{claim}
\label{clm: jsq_ssc_drift}
Using the dynamics of JSQ-A and the abandonment distribution ${\bf d} \sim Bin(\q,\gamma)$, we show that
\begin{align*}
    \EE[\Delta \| \q  \|^2 |\q ]  \leq  nA^2 + \gamma\langle\q ,{\bf 1} \rangle + (\gamma^2-2\gamma) \|\q \|^2 + 2(1-\gamma) \big(\lambda_\gamma q_{\min}  - \langle \q ,\boldsymbol\mu_\gamma \rangle\big),
\end{align*}
where $q_{\min} = \min_i q_i$, 
and
\begin{align*}
    \EE\big[ \Delta \langle\q ,{\bf 1} \rangle^2 \,\big|\, \q \big] &  \ge   (\gamma^2-2\gamma) \langle\q ,{\bf 1} \rangle^2+2\nu_\gamma(1-\gamma)\langle\q ,{\bf 1} \rangle  - n^2A^2.
\end{align*}
\end{claim}

Next, note that
\[ \|\q_{\perp} \|^2 = \|\q \|^2 - \|\q_{\|} \|^2 =  \|\q \|^2 -\frac{1}{n}  \langle \q ,{\bf 1} \rangle^2. \]
Thus, we get
	\begin{align}
	\label{eq: jsq_lemssc_perp_drift_proofoutline}
		\EE\big[\Delta \|\q_{\perp} \|^2\,\big|\,\q \big] &= \EE\big[\Delta \|\q \|^2 | \q \big] - \frac{1}{n} \EE[\Delta \langle \q ,{\bf 1} \rangle^2|\q ] \nonumber\\
		&\leq 2nA^2 + \gamma \langle \q ,{\bf 1} \rangle + (\gamma^2-2\gamma) \|\q_\perp \|^2  +2(1-\gamma)\left( \lambda_\gamma  q_{\min}  -\langle {\bf q} ,\boldsymbol\mu_\gamma \rangle - \frac{\nu_\gamma}{n}\langle \q ,{\bf 1} \rangle \right)\nonumber\\
		&\leq 2nA^2 + \gamma \langle \q ,{\bf 1} \rangle + 2(1-\gamma)\left( \lambda_\gamma  q_{\min}  -\langle {\bf q} ,\boldsymbol \mu_\gamma \rangle - \frac{\nu_\gamma}{n}\langle \q ,{\bf 1} \rangle\right).
	\end{align}

\begin{claim}
\label{clm: jsq_ssc_driftterm}
    Using $q_{\min} \leq \frac{1}{n}\langle\q ,{\bf 1} \rangle - \frac{1}{n}\|\q_{\perp} \|$ and the assumption $\nu^-_\gamma \leq \frac{1}{2}n\mu_{\gamma,\min}$, where $\nu_\gamma^- = \max\{0,-\nu_\gamma\}$, we show that 
	\begin{align*}
	    \lambda_\gamma  q_{\min} - \langle {\bf q} ,\boldsymbol \mu_\gamma \rangle- \frac{\nu_\gamma}{n}\langle \q ,{\bf 1} \rangle  
		& \leq  -\frac{1}{2}\mu_{\gamma,\min }\|\q_\perp \|.
	\end{align*}
\end{claim}

Substituting the result given in Claim \ref{clm: jsq_ssc_driftterm} into Eq. \eqref{eq: jsq_lemssc_perp_drift_proofoutline}, we obtain
	\begin{align*}
	    \EE[\Delta \|\q_{\perp}   \|^2|\q ] &\leq 2nA^2 + \gamma \langle\q ,{\bf 1} \rangle - (1-\gamma) \mu_{\gamma,\min } \|\q_\perp \|.
	\end{align*}
Next, we equate  the drift of $\EE[\Delta \|\q_{\perp} \|^2]$ to zero in steady state to get that
	\begin{equation*}
	    \EE[\|\q_\perp\|] \leq \frac{1}{(1-\gamma)\mu_{\gamma,\min}} \big(2nA^2  + \gamma \EE [ \langle \q,{\bf 1} \rangle]\big) \leq \frac{1}{(1-\gamma)\mu_{\gamma,\min}} \big(2nA^2  + E_1\big),
	\end{equation*}
 where $E_1$ is as mentioned in Claim \ref{clm: jsq_ssc_abandonment}.
	Further, we also get
		\begin{align*}
		\EE[\Delta \|\q_{\perp} \| | \q ] &\leq \frac{1}{2\|\q_{\perp} \|} \EE[\Delta \|\q_{\perp} \|^2|\q ] \leq \frac{2n A^2 + \gamma \langle\q,{\bf 1} \rangle}{2\|\q_{\perp} \|} -\frac{1}{2}(1-\gamma)\mu_{\gamma,\min}\leq -\frac{1}{4}(1-\gamma)\mu_{\gamma,\min} =: -\mathcal{\epsilon}_0,
	\end{align*}
	where the last inequality holds when
 \[ \|\q_{\perp} \| \geq \frac{4n A^2 + 2\gamma \langle\q ,{\bf 1} \rangle}{(1-\gamma)\mu_{\gamma,\min}} =: L(\q ). \]

\begin{claim}
\label{clm: jsq_ssc_bounded}
    As the arrivals and service are bounded, we show that
\begin{align*}
	|\Delta \|\q_{\perp} \| | \leq 2nA +  \langle {\bf d}  ,{\bf 1} \rangle =: Z({\bf d} ),
\end{align*}
\end{claim}

Next, we have 
	\begin{align*}
	    \Delta \|\q_{\perp} \|^3 &= \|\q_{\perp}^+\|^3 - \|\q_{\perp} \|^3\\
	    & = \big(\|\q_{\perp} \|+  \Delta \|\q_{\perp} \|\big)^3 - \|\q_{\perp} \|^3\\
	    &=3\|\q_{\perp} \|^2 \Delta \|\q_{\perp} \| +  3\|\q_{\perp} \| \big(\Delta \|\q_{\perp} \|\big)^2 +  \big(\Delta \|\q_{\perp} \|\big)^3
	\end{align*}
We handle each of the term separately. Here, we use the definition of $L(\q)$ and $Z(\q)$ and Claim \ref{clm: jsq_ssc_abandonment} to prove the following.\\

\begin{claim}
\label{clm: jsq_ssc_thridorder}
There exists constants $K_1, K_2$ and $K_3$, independent of $\gamma$, such that 
 \begin{align*}
     \EE[3\|\q_{\perp}\|^2\Delta \|\q_{\perp}\|] & \leq -3\epsilon_0\EE\left[\|\q_{\perp}\|^2\right] + K_1,\\
      \EE\big[3\|\q_{\perp}\| \big(\Delta \|\q_{\perp}\|\big)^2\big]  &\leq 3K_2\EE[\|\q_{\perp}\|^2]^{\frac{1}{2}}, \\
      \EE\big[\big(\Delta \|\q_{\perp}\|\big)^3\big] &\leq   3K_3.
 \end{align*}
\end{claim}

Next, by using the fact that $\EE\big[\Delta \|\q_{\perp} \|^3 \big]=0$ in steady state, we get
    \begin{align*}
	    \epsilon_0\EE\big[\|\q_{\perp}\|^2\big] -  K_2\EE\big[\|\q_{\perp}\|^2\big]^{\frac{1}{2}} \leq K_1  + K_3.
    \end{align*}
    If $\EE\big[\|\q_{\perp}\|^2\big] \geq 4K_2^2/\epsilon_0^2$, we get
    \[ \epsilon_0\EE\big[\|\q_{\perp}\|^2\big] -  K_2\EE\big[\|\q_{\perp}\|^2\big]^{\frac{1}{2}} \geq K_2\EE\big[\|\q_{\perp}\|^2\big]^{\frac{1}{2}}. \]
    Then by above equation, we have
    \[ \EE\big[\|\q_{\perp}\|^2\big] \leq (K_1  + K_3)^2/ K_2^2. \]
    This finally gives us that
	\begin{equation*}
	    \EE[\|\q_{\perp}\|^2] \leq \min \left\{\frac{4K_2^2}{\epsilon_0^2} , \frac{(K_1  + K_3)^2}{K_2^2} \right\}.
	\end{equation*}
	We conclude that $\EE[\|\q_{\perp}\|^2]$ is bounded irrespective of the value of $\gamma$. This completes the proof.

\bibliographystyle{ACM-Reference-Format}
\bibliography{references}


\begin{thebibliography}{37}


\ifx \showCODEN    \undefined \def \showCODEN     #1{\unskip}     \fi
\ifx \showDOI      \undefined \def \showDOI       #1{#1}\fi
\ifx \showISBNx    \undefined \def \showISBNx     #1{\unskip}     \fi
\ifx \showISBNxiii \undefined \def \showISBNxiii  #1{\unskip}     \fi
\ifx \showISSN     \undefined \def \showISSN      #1{\unskip}     \fi
\ifx \showLCCN     \undefined \def \showLCCN      #1{\unskip}     \fi
\ifx \shownote     \undefined \def \shownote      #1{#1}          \fi
\ifx \showarticletitle \undefined \def \showarticletitle #1{#1}   \fi
\ifx \showURL      \undefined \def \showURL       {\relax}        \fi
\providecommand\bibfield[2]{#2}
\providecommand\bibinfo[2]{#2}
\providecommand\natexlab[1]{#1}
\providecommand\showeprint[2][]{arXiv:#2}

\bibitem[Billingsley(1986)]%
        {Bill86}
\bibfield{author}{\bibinfo{person}{Patrick Billingsley}.}
  \bibinfo{year}{1986}\natexlab{}.
\newblock \bibinfo{booktitle}{\emph{Probability and Measure}
  (\bibinfo{edition}{second} ed.)}.
\newblock \bibinfo{publisher}{John Wiley and Sons}.
\newblock


\bibitem[Bramson(1998)]%
        {Bramson_state_space}
\bibfield{author}{\bibinfo{person}{M. Bramson}.}
  \bibinfo{year}{1998}\natexlab{}.
\newblock \showarticletitle{State space collapse with application to
  heavy-traffic limits for multiclass queueing networks}.
\newblock \bibinfo{journal}{\emph{Queueing Systems Theory and Applications}}
  (\bibinfo{year}{1998}), \bibinfo{pages}{89 -- 148}.
\newblock


\bibitem[Bramson and Dai(2001)]%
        {bramson2001heavy}
\bibfield{author}{\bibinfo{person}{Maury Bramson} {and} \bibinfo{person}{J.G.
  Dai}.} \bibinfo{year}{2001}\natexlab{}.
\newblock \showarticletitle{Heavy traffic limits for some queueing networks}.
\newblock \bibinfo{journal}{\emph{Annals of Applied Probability}}
  (\bibinfo{year}{2001}), \bibinfo{pages}{49--90}.
\newblock


\bibitem[Braverman and Dai(2017)]%
        {braverman_steins_2017}
\bibfield{author}{\bibinfo{person}{Anton Braverman} {and} \bibinfo{person}{J.G.
  Dai}.} \bibinfo{year}{2017}\natexlab{}.
\newblock \showarticletitle{Stein’s method for steady-state diffusion
  approximations of {M/Ph/n+M} systems}.
\newblock \bibinfo{journal}{\emph{The Annals of Applied Probability}}
  \bibinfo{volume}{27}, \bibinfo{number}{1} (\bibinfo{date}{Feb.}
  \bibinfo{year}{2017}).
\newblock
\showISSN{1050-5164}
\urldef\tempurl%
\url{https://doi.org/10.1214/16-AAP1211}
\showDOI{\tempurl}


\bibitem[Braverman et~al\mbox{.}(2017)]%
        {braverman2017stein}
\bibfield{author}{\bibinfo{person}{Anton Braverman}, \bibinfo{person}{J.G.
  Dai}, {and} \bibinfo{person}{Jiekun Feng}.} \bibinfo{year}{2017}\natexlab{}.
\newblock \showarticletitle{Stein's method for steady-state diffusion
  approximations: {A}n introduction through the {E}rlang-{A} and {E}rlang-{C}
  models}.
\newblock \bibinfo{journal}{\emph{Stochastic Systems}} \bibinfo{volume}{6},
  \bibinfo{number}{2} (\bibinfo{year}{2017}), \bibinfo{pages}{301--366}.
\newblock


\bibitem[Dai and He(2010)]%
        {dai_customer_2010}
\bibfield{author}{\bibinfo{person}{J.G. Dai} {and} \bibinfo{person}{Shuangchi
  He}.} \bibinfo{year}{2010}\natexlab{}.
\newblock \showarticletitle{Customer Abandonment in Many-Server Queues}.
\newblock \bibinfo{journal}{\emph{Mathematics of Operations Research}}
  \bibinfo{volume}{35}, \bibinfo{number}{2} (\bibinfo{date}{May}
  \bibinfo{year}{2010}), \bibinfo{pages}{347--362}.
\newblock
\showISSN{0364-765X, 1526-5471}
\urldef\tempurl%
\url{https://doi.org/10.1287/moor.1100.0443}
\showDOI{\tempurl}


\bibitem[Dai and He(2012)]%
        {dai2012many}
\bibfield{author}{\bibinfo{person}{J.G. Dai} {and} \bibinfo{person}{Shuangchi
  He}.} \bibinfo{year}{2012}\natexlab{}.
\newblock \showarticletitle{Many-server queues with customer abandonment: A
  survey of diffusion and fluid approximations}.
\newblock \bibinfo{journal}{\emph{Journal of Systems Science and Systems
  Engineering}} \bibinfo{volume}{21}, \bibinfo{number}{1}
  (\bibinfo{year}{2012}), \bibinfo{pages}{1--36}.
\newblock


\bibitem[Dai et~al\mbox{.}(2010)]%
        {dai_many-server_2010}
\bibfield{author}{\bibinfo{person}{J.G. Dai}, \bibinfo{person}{Shuangchi He},
  {and} \bibinfo{person}{Tolga Tezcan}.} \bibinfo{year}{2010}\natexlab{}.
\newblock \showarticletitle{Many-server diffusion limits for {G}/{Ph}/n+{GI}
  queues}.
\newblock \bibinfo{journal}{\emph{The Annals of Applied Probability}}
  \bibinfo{volume}{20}, \bibinfo{number}{5} (\bibinfo{date}{Oct.}
  \bibinfo{year}{2010}).
\newblock
\showISSN{1050-5164}
\urldef\tempurl%
\url{https://doi.org/10.1214/09-AAP674}
\showDOI{\tempurl}


\bibitem[Einav(2019)]%
        {yoav2019amazon}
\bibfield{author}{\bibinfo{person}{Yoav Einav}.}
  \bibinfo{year}{2019}\natexlab{}.
\newblock \showarticletitle{Amazon found every 100ms of latency cost them 1\%
  in sales}.
\newblock \bibinfo{journal}{\emph{Gigaspaces}} (\bibinfo{year}{2019}).
\newblock


\bibitem[Eryilmaz and Srikant(2012)]%
        {atilla}
\bibfield{author}{\bibinfo{person}{A. Eryilmaz} {and} \bibinfo{person}{R.
  Srikant}.} \bibinfo{year}{2012}\natexlab{}.
\newblock \showarticletitle{Asymptotically tight steady-state queue length
  bounds implied by drift conditions}.
\newblock \bibinfo{journal}{\emph{Queueing Systems}} \bibinfo{volume}{72},
  \bibinfo{number}{3-4} (\bibinfo{year}{2012}), \bibinfo{pages}{311--359}.
\newblock
\showISSN{0257-0130}


\bibitem[Eschenfeldt and Gamarnik(2018)]%
        {Gamarnik_JSQ}
\bibfield{author}{\bibinfo{person}{Patrick Eschenfeldt} {and}
  \bibinfo{person}{David Gamarnik}.} \bibinfo{year}{2018}\natexlab{}.
\newblock \showarticletitle{Join the Shortest Queue with Many Servers. The
  Heavy-Traffic Asymptotics}.
\newblock \bibinfo{journal}{\emph{Mathematics of Operations Research}}
  \bibinfo{volume}{43}, \bibinfo{number}{3} (\bibinfo{year}{2018}),
  \bibinfo{pages}{867--886}.
\newblock


\bibitem[Ferragut and Paganini(2012)]%
        {P2P}
\bibfield{author}{\bibinfo{person}{Andr\'es Ferragut} {and}
  \bibinfo{person}{Fernando Paganini}.} \bibinfo{year}{2012}\natexlab{}.
\newblock \showarticletitle{Content dynamics in {P2P} networks from queueing
  and fluid perspectives}. In \bibinfo{booktitle}{\emph{2012 24th International
  Teletraffic Congress (ITC 24)}}. IEEE, \bibinfo{pages}{1--8}.
\newblock


\bibitem[Foschini and Salz(1978)]%
        {foschini1978basic}
\bibfield{author}{\bibinfo{person}{G Foschini} {and} \bibinfo{person}{JACK
  Salz}.} \bibinfo{year}{1978}\natexlab{}.
\newblock \showarticletitle{A basic dynamic routing problem and diffusion}.
\newblock \bibinfo{journal}{\emph{IEEE Transactions on Communications}}
  \bibinfo{volume}{26}, \bibinfo{number}{3} (\bibinfo{year}{1978}),
  \bibinfo{pages}{320--327}.
\newblock


\bibitem[Garnett et~al\mbox{.}(2002)]%
        {garnett_designing_2002}
\bibfield{author}{\bibinfo{person}{O. Garnett}, \bibinfo{person}{A.
  Mandelbaum}, {and} \bibinfo{person}{M. Reiman}.}
  \bibinfo{year}{2002}\natexlab{}.
\newblock \showarticletitle{Designing a {Call} {Center} with {Impatient}
  {Customers}}.
\newblock \bibinfo{journal}{\emph{Manufacturing \& Service Operations
  Management}} \bibinfo{volume}{4}, \bibinfo{number}{3} (\bibinfo{date}{July}
  \bibinfo{year}{2002}), \bibinfo{pages}{208--227}.
\newblock
\showISSN{1523-4614, 1526-5498}
\urldef\tempurl%
\url{https://doi.org/10.1287/msom.4.3.208.7753}
\showDOI{\tempurl}


\bibitem[Hajek(1982)]%
        {hajek1982hitting}
\bibfield{author}{\bibinfo{person}{Bruce Hajek}.}
  \bibinfo{year}{1982}\natexlab{}.
\newblock \showarticletitle{Hitting-time and occupation-time bounds implied by
  drift analysis with applications}.
\newblock \bibinfo{journal}{\emph{Advances in Applied probability}}
  \bibinfo{volume}{14}, \bibinfo{number}{3} (\bibinfo{year}{1982}),
  \bibinfo{pages}{502--525}.
\newblock


\bibitem[Harrison(1978)]%
        {harrison1978diffusion}
\bibfield{author}{\bibinfo{person}{J~Michael Harrison}.}
  \bibinfo{year}{1978}\natexlab{}.
\newblock \showarticletitle{The diffusion approximation for tandem queues in
  heavy traffic}.
\newblock \bibinfo{journal}{\emph{Advances in Applied Probability}}
  \bibinfo{volume}{10}, \bibinfo{number}{4} (\bibinfo{year}{1978}),
  \bibinfo{pages}{886--905}.
\newblock


\bibitem[Hurtado-Lange and Maguluri(2020)]%
        {hurtado2020transform}
\bibfield{author}{\bibinfo{person}{Daniela Hurtado-Lange} {and}
  \bibinfo{person}{Siva~Theja Maguluri}.} \bibinfo{year}{2020}\natexlab{}.
\newblock \showarticletitle{Transform methods for heavy-traffic analysis}.
\newblock \bibinfo{journal}{\emph{Stochastic Systems}} \bibinfo{volume}{10},
  \bibinfo{number}{4} (\bibinfo{year}{2020}), \bibinfo{pages}{275--309}.
\newblock


\bibitem[Jhunjhunwala and Maguluri(2022)]%
        {Jhun_heavy_traffic}
\bibfield{author}{\bibinfo{person}{Prakirt Jhunjhunwala} {and}
  \bibinfo{person}{Siva~Theja Maguluri}.} \bibinfo{year}{2022}\natexlab{}.
\newblock \showarticletitle{Heavy Traffic Distribution of Queueing Systems
  without Resource Pooling}.
\newblock \bibinfo{journal}{\emph{Preprint arXiv:2206.06504}}
  (\bibinfo{year}{2022}).
\newblock
\urldef\tempurl%
\url{https://doi.org/10.48550/ARXIV.2206.06504}
\showDOI{\tempurl}


\bibitem[Kang and Ramanan(2010)]%
        {kang2010fluid}
\bibfield{author}{\bibinfo{person}{Weining Kang} {and} \bibinfo{person}{Kavita
  Ramanan}.} \bibinfo{year}{2010}\natexlab{}.
\newblock \showarticletitle{Fluid limits of many-server queues with reneging}.
\newblock \bibinfo{journal}{\emph{The Annals of Applied Probability}}
  \bibinfo{volume}{20}, \bibinfo{number}{6} (\bibinfo{year}{2010}),
  \bibinfo{pages}{2204--2260}.
\newblock


\bibitem[Khiyaita et~al\mbox{.}(2012)]%
        {khiyaita2012load}
\bibfield{author}{\bibinfo{person}{A Khiyaita}, \bibinfo{person}{H El~Bakkali},
  \bibinfo{person}{M Zbakh}, {and} \bibinfo{person}{Dafir El~Kettani}.}
  \bibinfo{year}{2012}\natexlab{}.
\newblock \showarticletitle{Load balancing cloud computing: state of art}.
\newblock \bibinfo{journal}{\emph{2012 National Days of Network Security and
  Systems}} (\bibinfo{year}{2012}), \bibinfo{pages}{106--109}.
\newblock


\bibitem[Kingman(1961)]%
        {kingman1961_charfunction}
\bibfield{author}{\bibinfo{person}{J Kingman}.}
  \bibinfo{year}{1961}\natexlab{}.
\newblock \showarticletitle{The single server queue in heavy traffic}. In
  \bibinfo{booktitle}{\emph{Mathematical Proceedings of the Cambridge
  Philosophical Society}}, Vol.~\bibinfo{volume}{57}. Cambridge University
  Press, \bibinfo{pages}{902--904}.
\newblock


\bibitem[Kingman(1962)]%
        {kingman1962_brownian}
\bibfield{author}{\bibinfo{person}{J Kingman}.}
  \bibinfo{year}{1962}\natexlab{}.
\newblock \showarticletitle{On queues in heavy traffic}.
\newblock \bibinfo{journal}{\emph{Journal of the Royal Statistical Society.
  Series B (Methodological)}} (\bibinfo{year}{1962}),
  \bibinfo{pages}{383--392}.
\newblock


\bibitem[Koole and Mandelbaum(2002)]%
        {koole_queueing_2002}
\bibfield{author}{\bibinfo{person}{Ger Koole} {and} \bibinfo{person}{Avishai
  Mandelbaum}.} \bibinfo{year}{2002}\natexlab{}.
\newblock \showarticletitle{Queueing {Models} of {Call} {Centers}: {An}
  {Introduction}}.
\newblock \bibinfo{journal}{\emph{Annals of Operations Research}}
  \bibinfo{volume}{113}, \bibinfo{number}{1/4} (\bibinfo{year}{2002}),
  \bibinfo{pages}{41--59}.
\newblock
\showISSN{02545330}
\urldef\tempurl%
\url{https://doi.org/10.1023/A:1020949626017}
\showDOI{\tempurl}


\bibitem[M{\"o}rters and Peres(2010)]%
        {morters2010brownian}
\bibfield{author}{\bibinfo{person}{Peter M{\"o}rters} {and}
  \bibinfo{person}{Yuval Peres}.} \bibinfo{year}{2010}\natexlab{}.
\newblock \bibinfo{booktitle}{\emph{Brownian motion}}.
  Vol.~\bibinfo{volume}{30}.
\newblock \bibinfo{publisher}{Cambridge University Press}.
\newblock


\bibitem[Nahmias(1982)]%
        {nahmias1982perishable}
\bibfield{author}{\bibinfo{person}{Steven Nahmias}.}
  \bibinfo{year}{1982}\natexlab{}.
\newblock \showarticletitle{Perishable inventory theory: A review}.
\newblock \bibinfo{journal}{\emph{Operations research}} \bibinfo{volume}{30},
  \bibinfo{number}{4} (\bibinfo{year}{1982}), \bibinfo{pages}{680--708}.
\newblock


\bibitem[Reed and Ward(2008)]%
        {reed_approximating_2008}
\bibfield{author}{\bibinfo{person}{J.~E. Reed} {and} \bibinfo{person}{Amy~R.
  Ward}.} \bibinfo{year}{2008}\natexlab{}.
\newblock \showarticletitle{Approximating the {GI}/{GI}/1+{GI} {Queue} with a
  {Nonlinear} {Drift} {Diffusion}: {Hazard} {Rate} {Scaling} in {Heavy}
  {Traffic}}.
\newblock \bibinfo{journal}{\emph{Mathematics of Operations Research}}
  \bibinfo{volume}{33}, \bibinfo{number}{3} (\bibinfo{date}{Aug.}
  \bibinfo{year}{2008}), \bibinfo{pages}{606--644}.
\newblock
\showISSN{0364-765X, 1526-5471}
\urldef\tempurl%
\url{https://doi.org/10.1287/moor.1070.0303}
\showDOI{\tempurl}


\bibitem[Reiman(1983)]%
        {rei_state_space}
\bibfield{author}{\bibinfo{person}{M.~I. Reiman}.}
  \bibinfo{year}{1983}\natexlab{}.
\newblock \showarticletitle{Some diffusion approximations with state space
  collapse}. In \bibinfo{booktitle}{\emph{Proceedings of International Seminar
  on Modelling and Performance Evaluation Methodology, Lecture Notes in Control
  and Information Sciences}}. \bibinfo{publisher}{Springer},
  \bibinfo{address}{Berlin}, \bibinfo{pages}{209--240}.
\newblock


\bibitem[Schlosshauer(2019)]%
        {schlosshauer2019quantum}
\bibfield{author}{\bibinfo{person}{Maximilian Schlosshauer}.}
  \bibinfo{year}{2019}\natexlab{}.
\newblock \showarticletitle{Quantum decoherence}.
\newblock \bibinfo{journal}{\emph{Physics Reports}}  \bibinfo{volume}{831}
  (\bibinfo{year}{2019}), \bibinfo{pages}{1--57}.
\newblock


\bibitem[Uhlenbeck and Ornstein(1930)]%
        {uhlenbeck1930theory}
\bibfield{author}{\bibinfo{person}{George~E Uhlenbeck} {and}
  \bibinfo{person}{Leonard~S Ornstein}.} \bibinfo{year}{1930}\natexlab{}.
\newblock \showarticletitle{On the theory of the Brownian motion}.
\newblock \bibinfo{journal}{\emph{Physical review}} \bibinfo{volume}{36},
  \bibinfo{number}{5} (\bibinfo{year}{1930}), \bibinfo{pages}{823}.
\newblock


\bibitem[Vardoyan et~al\mbox{.}(2019)]%
        {quantum2-towsley}
\bibfield{author}{\bibinfo{person}{Gayane Vardoyan}, \bibinfo{person}{Saikat
  Guha}, \bibinfo{person}{Philippe Nain}, {and} \bibinfo{person}{Don Towsley}.}
  \bibinfo{year}{2019}\natexlab{}.
\newblock \showarticletitle{On the Capacity Region of Bipartite and Tripartite
  Entanglement Switching}.
\newblock \bibinfo{journal}{\emph{Preprint arXiv:1901.06786}}
  (\bibinfo{year}{2019}).
\newblock


\bibitem[Varma and Maguluri(2021)]%
        {twosided_ht}
\bibfield{author}{\bibinfo{person}{Sushil~Mahavir Varma} {and}
  \bibinfo{person}{Siva~Theja Maguluri}.} \bibinfo{year}{2021}\natexlab{}.
\newblock \bibinfo{title}{A Heavy Traffic Theory of Two-Sided Queues}.
  (\bibinfo{year}{2021}).
\newblock
\newblock
\shownote{Under Review for {ACM SIGMETRICS} 2021 Preprint:
  https://tinyurl.com/HT2sided}.


\bibitem[Ward and Glynn(2003)]%
        {ward_diffusion_2003}
\bibfield{author}{\bibinfo{person}{Amy~R. Ward} {and} \bibinfo{person}{Peter~W.
  Glynn}.} \bibinfo{year}{2003}\natexlab{}.
\newblock \showarticletitle{A {Diffusion} {Approximation} for a {Markovian}
  {Queue} with {Reneging}}.
\newblock \bibinfo{journal}{\emph{Queueing Systems}} \bibinfo{volume}{43},
  \bibinfo{number}{1/2} (\bibinfo{year}{2003}), \bibinfo{pages}{103--128}.
\newblock
\showISSN{02570130}
\urldef\tempurl%
\url{https://doi.org/10.1023/A:1021804515162}
\showDOI{\tempurl}


\bibitem[Ward and Glynn(2005)]%
        {ward_diffusion_2005}
\bibfield{author}{\bibinfo{person}{Amy~R. Ward} {and} \bibinfo{person}{Peter~W.
  Glynn}.} \bibinfo{year}{2005}\natexlab{}.
\newblock \showarticletitle{A {Diffusion} {Approximation} for a {GI}/{GI}/1
  {Queue} with {Balking} or {Reneging}}.
\newblock \bibinfo{journal}{\emph{Queueing Systems}} \bibinfo{volume}{50},
  \bibinfo{number}{4} (\bibinfo{date}{Aug.} \bibinfo{year}{2005}),
  \bibinfo{pages}{371--400}.
\newblock
\showISSN{0257-0130, 1572-9443}
\urldef\tempurl%
\url{https://doi.org/10.1007/s11134-005-3282-3}
\showDOI{\tempurl}


\bibitem[Williams(1998)]%
        {Williams_state_space}
\bibfield{author}{\bibinfo{person}{R Williams}.}
  \bibinfo{year}{1998}\natexlab{}.
\newblock \showarticletitle{Diffusion approximations for open multiclass
  queueing networks: Sufficient conditions involving state space collapse}.
\newblock \bibinfo{journal}{\emph{Queueing Systems Theory and Applications}}
  (\bibinfo{year}{1998}), \bibinfo{pages}{27 -- 88}.
\newblock


\bibitem[Williams(2016)]%
        {williams_survey_SPN}
\bibfield{author}{\bibinfo{person}{Ruth Williams}.}
  \bibinfo{year}{2016}\natexlab{}.
\newblock \showarticletitle{Stochastic Processing Networks}.
\newblock \bibinfo{journal}{\emph{Annual Review of Statistics and Its
  Application}}  \bibinfo{volume}{3} (\bibinfo{year}{2016}),
  \bibinfo{pages}{323--345}.
\newblock


\bibitem[Yurke and Denker(1984)]%
        {yurke1984quantum}
\bibfield{author}{\bibinfo{person}{Bernard Yurke} {and} \bibinfo{person}{John~S
  Denker}.} \bibinfo{year}{1984}\natexlab{}.
\newblock \showarticletitle{Quantum network theory}.
\newblock \bibinfo{journal}{\emph{Physical Review A}} \bibinfo{volume}{29},
  \bibinfo{number}{3} (\bibinfo{year}{1984}), \bibinfo{pages}{1419}.
\newblock


\bibitem[Zeltyn and Mandelbaum(2005)]%
        {zeltyn_call_2005}
\bibfield{author}{\bibinfo{person}{Sergey Zeltyn} {and}
  \bibinfo{person}{Avishai Mandelbaum}.} \bibinfo{year}{2005}\natexlab{}.
\newblock \showarticletitle{Call {Centers} with {Impatient} {Customers}:
  {Many}-{Server} {Asymptotics} of the {M}/{M}/n + {G} {Queue}}.
\newblock \bibinfo{journal}{\emph{Queueing Systems}} \bibinfo{volume}{51},
  \bibinfo{number}{3-4} (\bibinfo{date}{Dec.} \bibinfo{year}{2005}),
  \bibinfo{pages}{361--402}.
\newblock
\showISSN{0257-0130, 1572-9443}
\urldef\tempurl%
\url{https://doi.org/10.1007/s11134-005-3699-8}
\showDOI{\tempurl}


\end{thebibliography}

\pagebreak

\begin{appendix}

\def\mgfconvergence{\ref{COR: MGF_CONVERGENCE}}
\section{Proof of Lemma \mgfconvergence}
\label{app: mgf_convergence}

Let $\mu_n$ be the distribution of the random variable $X_n$. First, we prove that the sequence of probability measures $\{\mu_n\}_{n=1}^\infty$ is tight. Since $M_n(s_0) + M_n(-s_0)$ converges, this sequence is also bounded. Then,
\begin{align*}
    \mathbb P(|X_n| > R) \leq e^{-s_0 R} \EE\left[ e^{s_0 |X_n|} \right] \leq e^{-s_0 R} \big( M_n(s_0) + M_n(-s_0) \big).
\end{align*}
It follows that $\{\mu_n\}_{n=1}^\infty$ is tight, and thus \cite[Theorem 25.10]{Bill86} implies that there exists a subsequence $\{\mu_{n_k}\}$ and a probability measure $\Tilde{\mu}$ such that $X_{n_k} \stackrel{d}{\rightarrow} \tilde X$, where $\tilde X \sim \tilde \mu$. By continuity, the MGF of $\tilde X$ is same as that of $X$. In particular, if $\tilde M(\cdot)$ is the MGF of the random variable $\tilde X$,  then, $\tilde M(s) = M(s)$ for $s\in [-s_0,s_0]$. Moreover, by using analytic continuation, we get that $\EE[e^{z \tilde X}] = \EE[e^{z X}]$ for any $z\in \mathbb C$ such that $Re(z) \in (-s_0,s_0)$. This implies that the characteristic function of $\tilde X$ matches with that of $X$, and then by \cite[Theorem 26.2]{Bill86}, $\tilde X$ and $X$ have the same distribution. This implies that every convergent subsequence of $\{\mu_n\}$ converges to $\mu$, where $\mu$ is the distribution of $X$, then, using \cite[Theorem 25.10]{Bill86} again, we get that $ X_n \stackrel{d}{\rightarrow} X$. 

To prove the convergence of moments, we use a similar argument, noting that the boundedness of the MGF in an interval around zero implies the boundedness of the moments, and the moments are uniquely determined by the MGF when the MGF exists in an interval around zero.

\section{Single Server Queue with Abandonment} \label{app: ssq_results}

\subsection{Model and main result}

\label{sec: ssq_model}

The SSQ-A system is just the JSQ-A system as given in Section \ref{sec: model} for the case $n=1$. For simplicity and distinction from JSQ-A, we do not use bold face fonts in the notation for SSQ-A. We use $q(t), a(t), s(t)$ to denote the queue length, arrivals, and potential services for any time $t$. The abandonments are denoted by $d(t)\sim Bin(q(t),\gamma)$. To reiterate, $\EE [a(t)] = \lambda_\gamma$ and Var$(a(t)) = \sigma_{\gamma,a}^2$, $\EE[s(t)] = \mu_\gamma$ and Var$(s(t)) = \sigma_{\gamma,s}^2$. Further, we use $c(t) = a(t)-s(t)$, then, $|c(t)|\leq A$, $\EE [c(t)] = \lambda_\gamma - \mu_\gamma = \nu_\gamma$ and Var$(c(t)) = \sigma_{\gamma,a}^2 + \sigma_{\gamma,s}^2 = \sigma^2_\gamma$. And, the queue length evolution equation is given by
\begin{align}
\label{eq: ssq_lindley}
    q(t+1) &= [q(t) + c\ti -d\ti]^+ = q(t) + c\ti -d\ti + u\ti,
\end{align}
where $u(t)$ denotes the unused services, with the condition $q(t+1)u(t) =0$ for any $t>0$, and $0\leq u(t)\leq s(t) \leq A$. Finally, we drop the symbol $t$ to denote the variables in steady state, i.e., $(q,d,u)$ follows the steady state distribution of the process $(q(t),d(t),u(t))$, and $c$ is an independent random variable distributed as $c(t)$ and $q^+ = q+c -d+u$. For this system, we have the following result (equivalent to Theorem \ref{THM: JSQ_LIMIT_DIS}).\\

\begin{theorem}
\label{thm: ssq_limit_dis}
For the SSQ-A setting, we have the following results.
\begin{enumerate}[label=(\alph*), ref=\ref{thm: ssq_limit_dis}.\alph*]
    \item \label{thm: ssq_fast}  Classic Heavy Traffic: Suppose $ |\nu_\gamma +C_f\gamma^{ \alpha }| \leq \gamma^{\alpha + \beta}$, where $ \alpha\in \lfsmall{0,\frac{1}{2}}$, $\beta >0$, and $C_f >0$ are constants. Let $\Upsilon_f$ be an exponential random variable with mean $\frac{\sigma^2}{2C_f}$, where $\sigma^2 = \limg \sigma^2_\gamma$. Then, 
     \begin{align*}
         \gamma^\alpha q \stackrel{d}{\rightarrow} \Upsilon_f, && \limg \gamma^{\alpha m} \EE[ q^m] = \EE[\Upsilon_f^m]. 
     \end{align*}
    \item \label{thm: ssq_crit} Critical Abandonment: Suppose $|\nu_\gamma-C_c\sqrt{\gamma}| \leq \gamma^{\frac{1}{2} + \beta}$, where $\beta >0$, and $C_c \in \mathbb R$ are constants. Let let $\Upsilon_c$ is a Gaussian random variable with mean $C_c$ and variance $\frac{\sigma^2}{2}$. Then,
    \begin{align*}
        \sqrt{\gamma} q \stackrel{d}{\rightarrow} \Upsilon_c | \Upsilon_c>0, && \limg \gamma^{\frac{m}{2}}\EE[ q^m] = \EE[\Upsilon_c^m | \Upsilon_c>0],
    \end{align*}
    where $\Upsilon_c | \Upsilon_c >0$ denotes the random variable $\Upsilon_c$ conditioned on the event $\{\Upsilon_c>0\}$.
    \item \label{thm: ssq_slow} Slow Abandonment: Suppose $|\nu_\gamma - C_s \gamma^{\alpha}| \leq \gamma^{\alpha +\beta}$, where $\alpha \in \big[0,\frac{1}{2}\big)$, $\beta >0$, and $C_s >0$ are constants. Let $\Upsilon_s$ be a Gaussian random variable with zero mean and variance $\frac{\bar\sigma^2}{2}$, where $\bar\sigma^2 = \lim_{\gamma\rightarrow 0} \sigma^2_\gamma +\nu_\gamma (1-\gamma)$. Then,
\begin{align*}
     \sqrt{\gamma} \left( q - \frac{\nu_\gamma}{\gamma} \right) \stackrel{d}{\rightarrow} \Upsilon_s, &&  \limg \gamma^{\frac{m}{2}}\EE \left[ \left( q - \frac{\nu_\gamma}{\gamma}\right)^m \right] = \EE[\Upsilon_s^m].
\end{align*}
\end{enumerate}
\end{theorem}
The proof of this theorem is provided in the rest of this Appendix. The first step is to obtain bounds on the MGF of the queue length (cf. Lemma \ref{lem: ssq_mgfexist}).

\subsection{Proof of Lemma \ref{lem: ssq_mgfexist}}
\label{app: ssq_mgf}

First, if $\EE[c(t)] = \nu_\gamma<0$, then we allow for extra arrivals and define the new sequence $c'(t) = c(t) + a'(t)$, where $a'(t)\geq 0$ are i.i.d. and $\EE[a'(t)] = -\nu_\gamma>0$. This yields $\EE[c'(t)] = 0$. Also, the queue length process with arrival-service sequence $\{c'(t)\}_{t=0}^\infty$ dominates the original queue length process. Thus any upper bound on this new queue length process will be valid for the original queue length process. So, without loss of generality, we assume that $\nu_\gamma \geq 0$.\\

\textbf{Coupling:} We consider two different stochastic processes $\{\tilde q_1 (t) \}_{t=0}^\infty$ and $\{\tilde q_2 (t) \}_{t=0}^\infty$ such that the arrival sequence and service sequence is same as that for the original, i.e., $\{c(t)\}_{t=0}^\infty$ is same for  $q(t)$, $\tilde q_1 (t)$ and $\tilde q_2(t)$, but the abandonment processes $\tilde d_1(t) $ and $\tilde d_2(t) $ fore $\tilde q_1(t)$ and  $\tilde q_2(t)$ are given by
\begin{align*}
    \tilde d_1 (t) \sim Bin \big(\min\{ \lfloor\tilde C_1/ \gamma \rfloor, \tilde q_1(t) \}, \gamma\big), && \tilde d_2 (t) \sim Bin \big(\max\{ \lceil\tilde C_2/ \gamma \rceil, \tilde q_2(t) \}, \gamma\big),
\end{align*}
where $\tilde C_1 = \nu_\gamma + A\sqrt{\gamma}$ (with $\nu_\gamma >0$) and $\tilde C_2 = [\nu_\gamma - A\sqrt{\gamma}]^+ = \max\{\nu_\gamma - A\sqrt{\gamma},0\}$. The queue evolution is given by
\begin{align*}
    \tilde q_i (t+1) = [\tilde q_i(t) + c(t) - \tilde d_i(t)]^+ = \tilde q_i(t) + c(t) - \tilde d_i(t) + \tilde u_i(t),
\end{align*}
 for $i\in\{1,2\}$, with the initial condition $\tilde q_1(0) =\tilde q_2(0) = q(0)$. 
 
  We assume that the three binomial random variables $\tilde    d_1\ti$, $\tilde    d_2\ti$ and $    d\ti$ are coupled in a way that they are generated using same sequence of i.i.d. Bernoulli random variables, i.e., given a sequence of i.i.d. Bernoulli random variables $\{X_i(t)\}_{i=1}^\infty$ with success probability $\gamma$, then
\begin{align*}
    d\ti = \sum_{i=1}^{q\ti} X_i(t), && \tilde d_1\ti = \sum_{i=1}^{\min\{\lfloor \tilde C_1 / \gamma \rfloor, \tilde q_1(t)\}} X_i(t), && \tilde d_2\ti = \sum_{i=1}^{\max\{\lceil \tilde C_2 / \gamma \rceil, \tilde q_2(t)\}} X_i(t).
\end{align*} 

Note that, even though we are using the $\lfloor \cdot \rfloor$ and $\lceil \cdot \rceil$ functions to properly define $\tilde d_1(t)$ and $\tilde d_2(t)$, in order to simplify some steps, we are going to substitute $\lfloor x \rfloor$ and $\lceil x \rceil$ by $x$ while providing any upper bound. Using this coupling, since $X_i(t)\in \{0,1\}$, we have 
\begin{align}
\label{eq: ssq_mgfbounf_abandiffbound}
    \tilde d_1\ti -d\ti &\leq \big[\min\{\tilde C_1 / \gamma, \tilde q_1(t)\} - q\ti\big]^+ \leq \big[ \tilde q_1(t) - q\ti\big]^+, \nonumber\\ d\ti - \tilde d_2\ti &\leq \big[q\ti - \max\{\tilde C_2 / \gamma, \tilde q_2(t)\} \big]^+ \leq \big[ q(t) - \tilde  q_2\ti\big]^+.
\end{align}
We now claim that $q \ti \leq \tilde q_1\ti$, for all $t\geq 0$. We prove this claim by using induction. At $t= 0$, we know that $q(0) = \tilde q_1(0) = 0$ by assumption. Suppose $\tilde q_1(t) \geq q(t)$ for some $t>0$. Then, 
\begin{align}
\label{eq: ssq_mgfexist_coupling}
    \tilde q_1(t+1) &= [\tilde q_1(t) + c\ti -\tilde d_1\ti]^+ \nonumber\allowdisplaybreaks\\
    &= [q(t)+ c\ti -d_1\ti + \tilde q_1(t) - q(t) - (\tilde d_1\ti - d\ti)]^+\nonumber\allowdisplaybreaks\\
    & \stackrel{(a)}{\geq} [q(t)+ c\ti -d\ti]^+\nonumber\\
    & = q(t+1),
\end{align}
where (a) follows because $\tilde q_1(t) - q(t) = \big[ \tilde q_1(t) - q\ti\big]^+ \geq \tilde d_1\ti - d\ti  $ as shown previously in Eq. \eqref{eq: ssq_mgfbounf_abandiffbound}. This completes the induction argument, and thus  $q \ti \leq \tilde q_1\ti$ for all $t\geq 0$. Using a similar argument, we obtain $\tilde q_2(t) \leq q \ti$, for all $t\geq 0$.\\

\textbf{Stability of coupled processes:} For the process $\{\tilde q_1(t)\}_{t=0}^\infty$, using that $\tilde q_1(t+1)\tilde u_1(t) =0$, we get
\begin{align}
\label{eq: ssq_tildeq1_drift}
    \EE[ \tilde q_1^2(t+1) - \tilde q_1^2(t) |\tilde q_1(t) ]    &= \EE[(\tilde 
    q_1(t+1) - \tilde u_1(t))^2 |\tilde q_1(t)] - \EE[\tilde u_1^2(t)|\tilde q_1(t)] - \tilde q_1^2(t) \allowdisplaybreaks\nonumber\\
    & = \EE[(\tilde 
    q_1(t) +c(t) - \tilde d_1(t))^2 |\tilde q_1(t)] - \EE[\tilde u_1^2(t)|\tilde q_1(t)] - \tilde q_1^2(t) \allowdisplaybreaks\nonumber\\
    &\leq \EE[(c(t) - \tilde d_1(t))^2 | \tilde q_1(t)] +2\tilde q_1(t)\EE[c(t) - \tilde d_1(t)|\tilde q_1(t)]\allowdisplaybreaks\nonumber\\
    & = \EE[c^2(t)] + \EE[ \tilde d_1^2(t) | \tilde q_1(t)] -2\EE[c(t)]\EE[\tilde d_1(t) | \tilde q_1(t)]+ 2\tilde q_1(t)\EE[c(t) - \tilde d_1(t)|\tilde q_1(t)]\allowdisplaybreaks\nonumber\\
    & \stackrel{(a)}{\leq} A^2 +(1-\gamma)\tilde C_1+ \tilde C_1^2 +2A \tilde C_1 +2 \tilde q_1(t)(\nu_\gamma - \tilde C_1 + [\tilde C_1  -\gamma \tilde q_1(t)]^+)\allowdisplaybreaks\nonumber\\
    &\stackrel{(b)}{\leq} 11A^2-2A\sqrt{\gamma}  \tilde q_1(t) + 2 \tilde q_1(t)[\tilde C_1  -\gamma \tilde q_1(t)]^+ \allowdisplaybreaks\nonumber\\
    &\stackrel{(c)}{\leq} -A\sqrt{\gamma} \tilde q_1(t)
\end{align}
where (a) holds because $|c(t)|\leq A$, and $\tilde d_1(t) \sim Bin(\min\{\lfloor \tilde C_1/\gamma\rfloor, \tilde q_2\},\gamma) \leq Bin( \lfloor\tilde C_1/\gamma\rfloor , \gamma),$ and thus, $\EE[ \tilde d_1(t) | \tilde q_1(t)] = \tilde C_1 - [\tilde C_1 - \gamma \tilde q_1]^+ \leq \tilde C_1$ and $\EE[ \tilde d_1(t) | \tilde q_1(t)] \leq (1-\gamma)\tilde C_1+ \tilde C_1^2$; (b) holds because $\tilde C_1 = \nu_\gamma + A\sqrt{\gamma} \leq 2A$; and (c) holds when $\tilde q_1(t) \geq \max\big\{\frac{ 11A}{\sqrt{\gamma}}, \frac{\tilde C_1}{\gamma}\big\}$. Now, by Foster-Lyapunov Theorem, the Markov chain $\{\tilde q_1(t)\}$ is positive recurrent. Since $\tilde q_2(t)\leq \tilde q_1(t)$, this implies that $\tilde q_2(t)$ is also positive recurrent.\\

\textbf{Existence of the MGF of ${\bf \tilde q_1}$:} 
Recall that $\tilde q_1 (t+1) = \tilde q_1(t) + c(t) - \tilde d_1(t) + \tilde u_1(t)$. Then, 
\begin{align}
\label{eq: ssq_mgfexist_starteq1}
    \EE[ & \exp\big(\phi \sqrt{\gamma}(\tilde q_1(t+1)-\tilde u_1(t))\big) | \tilde q_1(t)]\nonumber\\
    &= \exp\big(\phi \sqrt{\gamma}\tilde q_1(t)\big)  \EE[ \exp\big(\phi \sqrt{\gamma} (c(t) -\tilde d_1(t))\big) | \tilde q_1 ]\nonumber\\
    &\stackrel{(a)}{\leq} \exp\big(\phi \sqrt{\gamma}\tilde q_1(t)\big) \EE \Big[ 1 +  \phi \sqrt{\gamma}(c(t)-\tilde d_1(t) ) + \phi^2 \gamma(c(t)-\tilde d_1(t))^2 \exp\big(\phi \sqrt{\gamma} [c(t)]^+\big) \Big| \tilde q_1(t) \Big ] \nonumber\\
    &\stackrel{(b)}{\leq}  \exp\big(\phi \sqrt{\gamma}\tilde q_1(t)\big)  \Big( 1 +  \phi \sqrt{\gamma}\EE[c(t)-\tilde d_1(t)| \tilde q_1(t)]\Big)   +  \phi^2 e^{\phi \sqrt{\gamma}A} \gamma\EE \Big[(c(t)-\tilde d_1(t) )^2  \Big| \tilde q_1(t) \Big ]
\end{align}
where (a) follows from using $e^x \leq 1+x+ x^2e^{[x]^+}$ and $[c(t)-\tilde d_1(t)]^+\leq [c(t)]^+$; (b) follows from $c(t) \leq A$. Since $e^{x}\geq 1+x$, we have
\begin{align}
\label{eq: ssq_mgfexist_term1}
	\EE[ \exp\big(\phi \sqrt{\gamma}(\tilde q_1(t+1)-\tilde u_1(t))\big) | \tilde q_1(t)]
	&\geq \EE[ (1-\phi\sqrt{\gamma}\tilde u_1(t)) \exp\big(\phi \sqrt{\gamma}\tilde q_1(t+1)\big) | \tilde q_1(t)]\nonumber\\
	& \stackrel{(a)}{=} \EE[  \exp\big(\phi \sqrt{\gamma}\tilde q_1(t+1)\big) | \tilde q_1(t)]  -\phi \sqrt{\gamma}\EE[\tilde u_1(t) | \tilde q_1(t)],
\end{align}
where (a) follows by using  $\tilde q_1(t+1) \tilde u_1(t) =0$ and $\tilde u_1(t) e^{\phi \sqrt{\gamma}\tilde q_1(t+1)} = \tilde u_1(t) $.
Moreover, by similar calculations as in Eq. \eqref{eq: ssq_tildeq1_drift}, we obtain 
\[ \EE \big[(c(t)-\tilde d_1(t) )^2  \big| \tilde q_1(t) \big ] \leq 11A^2 \]
and
\begin{align}
\label{eq: ssq_mgfexist_term3}
    \EE\big[c(t)-\tilde d_1(t) | \tilde q_1(t) \big] = \nu_\gamma - \tilde C_1 + [\tilde C_1-\gamma \tilde q_1(t)]^+= -\sqrt{\gamma} A + [\tilde C_1-\gamma \tilde q_1(t)]^+
\end{align}
Thus, combining Eq. \eqref{eq: ssq_mgfexist_starteq1},\eqref{eq: ssq_mgfexist_term1}, and \eqref{eq: ssq_mgfexist_term3}, we get
\begin{align}
\label{eq: ssq_tildeq1_expdrift}
    \EE[  & \exp\big(\phi \sqrt{\gamma}\tilde q_1(t+1)\big) | \tilde q_1(t)] -\exp\big(\phi \sqrt{\gamma}\tilde q_1(t)\big)\nonumber \allowdisplaybreaks\\
    & \leq \Big[-\phi \gamma A + \phi \sqrt{\gamma}[\tilde C_1-\gamma \tilde q_1(t) ]^++ 11A^2\phi^2 \gamma e^A \Big]\exp\big(\phi \sqrt{\gamma}\tilde q_1(t)\big) + \phi \sqrt{\gamma}\EE[\tilde u_1(t) | \tilde q_1(t)]\nonumber\allowdisplaybreaks\\
    & \stackrel{(a)}{\leq} \big[-\phi \gamma A +11A^2 \phi^2 \gamma e^A \big]\exp\big(\phi \sqrt{\gamma}\tilde q_1(t)\big) + \phi \sqrt{\gamma}\EE[\tilde u_1(t) | \tilde q_1(t)]\nonumber\allowdisplaybreaks\\
    & \stackrel{(b)}{\leq}  - \frac{1}{2} \phi \gamma A\exp\big(\phi \sqrt{\gamma}\tilde q_1(t)\big) + \phi \sqrt{\gamma}\EE[\tilde u_1(t) | \tilde q_1(t)]\nonumber\allowdisplaybreaks\\
    & \stackrel{(c)}{\leq}  - \frac{1}{2} \phi \gamma A\exp\big(\phi \sqrt{\gamma}\tilde q_1(t)\big),
\end{align}
where (a) holds whenever $\tilde q_1(t) \geq \frac{\tilde C_1}{\gamma}$; (b) holds when $\phi < \frac{1}{22Ae^A}$ and (c) holds by further assuming that $\tilde q_1(t) \geq \frac{\tilde C_1}{\gamma} + A$, in which case $\tilde q_1(t) + c(t) - \tilde d_1(t) \geq 0$ as $\tilde d_1(t) \leq \frac{\tilde C_1}{\gamma}$ and $c(t)\leq A$, and so $\tilde u_1(t) =0$. Thus, for $\tilde q_1(t) \geq \frac{\tilde C_1}{\gamma} + A$ and $\phi < \frac{1}{22Ae^A}$, the drift of the exponential Lyapunov function is negative and so by Foster-Lyapunov theorem, we get that $\EE[\exp\big(\phi \sqrt{\gamma}\tilde q_1(t)\big)]<\infty$ in steady-state.\\ 

\textbf{Bound on MGF of ${\bf \tilde q_1}$:} Suppose $(\tilde q_i,\tilde d_i,\tilde u_i)$ is distributed as the steady state of the Markov chain $\{(\tilde q_i(t),\tilde d_i(t),\tilde u_i(t))\}_{t= 0}^\infty$, and let \[\tilde q_i^+ = \tilde q_i + c - \tilde d_i + \tilde u_i.\] Then, $\tilde q_i^+$ is also distributed as the steady state  of the Markov chain $\{\tilde q_i(t)\}_{t= 0}^\infty$. We first look at the process $\{\tilde q_1(t)\}_{t= 0}^\infty$, and provide a bound on its MGF.
Since $\tilde q_1$ and $\tilde q_1^+$ have the same distribution, we have $\EE[\tilde q_1^+ - \tilde q_1] =0$ and $\EE[c - \tilde d_1 + \tilde u_1] =0$, and thus
\begin{align}
\label{eq: ssq_usmall}
    \EE[\tilde u_1] + \EE \big[\tilde C_1 -\gamma \tilde q_1\big]^+ &= \EE[\tilde u_1] + \tilde C_1 - \EE[\min\{\tilde C_1, \gamma \tilde q_1\}] \nonumber\\
    &= \EE[\tilde u_1] + \tilde C_1 - \EE[\tilde d_1 ] \nonumber\\
    &= \tilde C_1 - \EE[c] \nonumber\\
    &= A\sqrt{\gamma},
\end{align}
where last equality holds since $\tilde C_1 = \nu_\gamma + A \sqrt{\gamma}$. By similar calculations as in Eq. \eqref{eq: ssq_tildeq1_expdrift} for the steady-state variables, we have
\begin{align*}
    \EE[  \exp\big(\phi \sqrt{\gamma}\tilde q_1^+\big) | \tilde q_1] & - \exp\big(\phi \sqrt{\gamma}\tilde q_1\big)\nonumber\\
    & \leq \Big[-\phi \gamma A + \phi \sqrt{\gamma}[\tilde C_1-\gamma \tilde q_1 ]^++ 11A^2\phi^2 \gamma e^A \Big]\exp\big(\phi \sqrt{\gamma}\tilde q_1\big) + \phi \sqrt{\gamma}\EE[\tilde u_1 | \tilde q_1].\nonumber
\end{align*}
By taking expectation on both sides, and equating the drift to zero in steady state, for $0<\phi < \frac{1}{22Ae^A} <1$, we have 
\begin{align*}
    0 &= \EE[  \exp\big(\phi \sqrt{\gamma}\tilde q_1^+\big) ] -\EE[\exp\big(\phi \sqrt{\gamma}\tilde q_1\big)]\nonumber\\
    &\leq -\frac{1}{2}\phi \gamma A \MM{\tilde q_1} + \phi \sqrt{\gamma} \EE\big[[\tilde C_1-\gamma \tilde q_1 ]^+\exp\big(\phi \sqrt{\gamma}\tilde q_1\big)  \big] + \phi \sqrt{\gamma} \EE[\tilde u_1]\\
    &\stackrel{(a)}{\leq}-\frac{1}{2}\phi \gamma A \MM{\tilde q_1} + \phi \sqrt{\gamma} \exp\left(\frac{\phi \tilde C_1}{\sqrt{\gamma}}\right) \EE[\tilde C_1-\gamma \tilde q_1 ]^+ + \phi \sqrt{\gamma}\EE[\tilde u_1]\\
    &\leq -\frac{1}{2}\phi \gamma A \MM{\tilde q_1} + \phi \sqrt{\gamma} \exp\left(\frac{\phi \tilde C_1}{\sqrt{\gamma}}\right) \EE\big[[\tilde C_1-\gamma \tilde q_1 ]^++ \tilde u_1\big]\\
    &\stackrel{(b)}{\leq} -\frac{1}{2}\phi \gamma A \MM{\tilde q_1} + \phi \gamma A \exp\left(\frac{\phi \tilde C_1}{\sqrt{\gamma}}\right),
\end{align*}
where (a) follows by using $[\tilde C_1-\gamma \tilde q_1 ]^+\exp\big(\phi \sqrt{\gamma}\tilde q_1\big) \leq \exp\big(\frac{\phi \tilde C_1}{\sqrt{\gamma}}\big) [\tilde C_1-\gamma \tilde q_1 ]^+$; and (b) follows from Eq. \eqref{eq: ssq_usmall}. Finally, since $\tilde C_1 = \nu_\gamma + A\sqrt{\gamma}$, for any $0<\phi < \frac{1}{22Ae^A}$, we obtain 
\begin{align*}
    \MM{\tilde q_1} \leq 2e^A \exp\left(\frac{\phi \nu_\gamma}{\sqrt{\gamma}}\right).
\end{align*}
 By coupling, we know that $\MM{q} \leq \MM{\tilde q_1}$ for any $\phi >0$. This provides a bound on the MGF $\MM{q} $ for any $0<\phi<\frac{1}{22Ae^A}$ and $\gamma\in (0,1)$. Recall that when, $\nu_\gamma <0$, we can create a coupled queue length process with extra arrivals that dominates the original one, and for the bigger process (with $\EE[c(t)] = \nu_\gamma^+$), we can use the above argument. Therefore, for any $0<\phi < \frac{1}{22Ae^A}$, we have
\begin{align}
\label{eq: ssq_tildeq1_mgfbound}
    \MM{q} \leq 2e^A \exp\left(\frac{\phi \nu_\gamma^+}{\sqrt{\gamma}}\right).
\end{align}

 \textbf{Bound on the MGF of ${\bf \tilde q_2}$:} For Lemma \ref{lem: ssq_mgfexist_b}, we provide a bound on the MGF of $\tilde q_2$, i.e. $\EE[e^{\sqrt{\gamma}\phi\tilde q_2}]$, for $\phi<0$. 
 For simplicity, we replace $\phi$ by $-\phi$. Note that $e^{x}\leq 1$ for any $x<0$. So the MGF exists for $\phi >0$. Using the notation $\tilde q_2^+= \tilde q_2 + c - \tilde d_2 + \tilde u_2$ for the steady state variables, for $\phi >0$, we have 
\begin{align}
    \label{eq: ssq_mgfexist_starteq2}
    \EE[e^{-\phi \sqrt{\gamma}\tilde q_2^+} | \tilde q_2] &=  e^{-\phi \sqrt{\gamma}\tilde q_2} \EE[ e^{-\phi \sqrt{\gamma}(c-\tilde d_2+\tilde u_2)} | \tilde q_2 ] \allowdisplaybreaks \nonumber\\
    &\stackrel{(a)}{\leq} e^{-\phi \sqrt{\gamma}\tilde q_2} \EE[ e^{\phi \sqrt{\gamma}(\tilde d_2-c)} | \tilde q_2 ] \allowdisplaybreaks \nonumber\\
    &\stackrel{(b)}{\leq} e^{-\phi \sqrt{\gamma}\tilde q_2}\Big(1+ \phi \sqrt{\gamma} \EE[\tilde d_2-c | \tilde q_2]  + \phi^2 \gamma \EE\Big[(\tilde d_2-c)^2e^{\phi \sqrt{\gamma}[\tilde d_2-c]^+} \Big| \tilde q_2\Big] \Big) \allowdisplaybreaks \nonumber\\
    & \stackrel{(c)}{\leq} e^{-\phi \sqrt{\gamma}\tilde q_2}\Big(1+ \phi \sqrt{\gamma} (\max\{\tilde C_2, \gamma \tilde q_2\} -\nu_\gamma)  + \phi^2 \gamma e^{A} \EE\Big[(A+\tilde d_2)^2e^{\phi \sqrt{\gamma}\tilde d_2} \Big| \tilde q_2\Big] \Big),
\end{align}
where (a) follows from $\tilde u_2 \geq 0$; (b) follows from $e^{x} \leq 1+x+x^2e^{[x]^+}$; and (c) follows from $\EE[\tilde d_2 | \tilde q_2] = \max\{\tilde C_2, \gamma \tilde q_2\}$, $\EE[c|\tilde q_2] = \EE[c] = \nu_\gamma$, $(\tilde d_2-c)^2\leq  (A+\tilde d_2)^2$ and $[\tilde d_2-c]^+ \leq A + \tilde d_2$. Denoting $\phi_2 = \frac{1}{22A e^A}$,  from Eq. \eqref{eq: ssq_tildeq1_mgfbound}, we know that
\begin{align*}
    1+\sqrt{\gamma} \phi_2 \EE \left[ q - \frac{\nu_\gamma}{\gamma} \right]^+ \leq \EE\left[ e^{\sqrt{\gamma}\phi_2  \big[q- \frac{\nu_\gamma}{\gamma} \big]^+} \right] \leq 1 + \EE\left[ e^{\phi_2 \sqrt{\gamma} ( q - \nu_\gamma/\gamma)} \right] \leq 1 +2e^A.
\end{align*}
This implies that 
\begin{align*}
    \EE[\gamma \tilde q_2 - \tilde C_2]^+ \stackrel{(a)}{\leq} \EE[\gamma \tilde q_2 - \nu_\gamma]^+ + A\sqrt{\gamma} \stackrel{(b)}{\leq} \EE[\gamma q - \nu_\gamma]^+ + A\sqrt{\gamma} \leq A\sqrt{\gamma} + \frac{2\sqrt{\gamma}e^A}{\phi_2} \leq  \sqrt{\gamma} K_1,
\end{align*}
where (a) follows using $\tilde C_2 = \nu_\gamma - A\sqrt{\gamma}$; and (b) follows by using $\tilde q_2 \leq q$, given by the coupling argument, and finally, we pick $K_1 = A+2\phi_2^{-1}e^A$. For the first term, since $\tilde C_2 = \nu_\gamma - A\sqrt{\gamma}$,
\begin{align}
\label{eq: ssq_mgfexists2_term1}
    \EE[(\max\{\tilde C_2, \gamma \tilde q_2\} -\nu_\gamma) e^{-\phi \sqrt{\gamma}\tilde q_2}] &
    = \EE[(-A\sqrt{\gamma} +(\gamma \tilde q_2 - \tilde C_2)\mathds{1}_{\{\gamma \tilde q_2 \geq \tilde C_2 \}}  ) e^{-\phi \sqrt{\gamma}\tilde q_2}] \nonumber \allowdisplaybreaks\\
    & \leq -A\sqrt{\gamma} M_{\tilde q_2}^\gamma(-\phi) + e^{-\phi\frac{\tilde C_2}{\sqrt{\gamma}}}\EE[\gamma \tilde q_2 - \tilde C_2]^+ \nonumber \allowdisplaybreaks\\
    & \leq -A\sqrt{\gamma} M_{\tilde q_2}^\gamma(-\phi) + \sqrt{\gamma} K_1 e^{-\phi\frac{\tilde C_2}{\sqrt{\gamma}}}.
\end{align}
For the second term, for any Binomial random variable $d \sim Bin(x,\gamma)$ and $0<\theta<1$, using $x^2 \leq 2e^x$ for $x\geq 0$, we obtain,
\begin{align*}
    \EE\big[(A+d)^2 e^{\theta d}\big] &\leq 2e^{A}\EE\big[ e^{(1+\theta)d}\big] = 2e^{A}\big(1+\gamma(e^{\theta+1} -1)\big)^{x} \leq 2e^{A} e^{8\gamma x},
\end{align*}
where the last inequality follows by using $e^{1+\theta} \leq 9$ for $\theta< 1$. 
Thus, since $\tilde C_2 \leq \nu_\gamma \leq A$, we get
\begin{align*}
    \EE\left[\left. (A+\tilde d_2)^2 e^{\phi\sqrt{\gamma} \tilde d_2} \,\right|\, \tilde q_2 \right] &\leq 2e^A \exp\big(8 \max\{\tilde C_2 , \gamma \tilde q_2\}\big) 
    \leq 2e^{9A} + 2e^A e^{8\gamma \tilde q_2} \mathds{1}_{\{\gamma \tilde  q_2 \geq \tilde C_2  \}}.
\end{align*}
This gives us 
\begin{align}
\label{eq: ssq_mgfexists2_term2}
    \EE\Big[ (A+\tilde d_2)^2 e^{-\phi \sqrt{\gamma}\tilde q_2} e^{\phi\sqrt{\gamma} \tilde d_2} \Big] &
    \leq 2e^{9A} \EE\Big[  e^{-\phi \sqrt{\gamma}\tilde q_2} \Big] + 2e^{A}\EE\Big[  e^{8\gamma \tilde q_2} e^{-\phi \sqrt{\gamma}\tilde q_2} \mathds{1}_{\{\gamma  \tilde q_2 \geq \tilde C_2  \}} \Big] \nonumber \allowdisplaybreaks\\
    &  \leq 2e^{9A}M_{\tilde q_2}^\gamma(-\phi) + 2e^{A} e^{-\phi\frac{\tilde C_2}{\sqrt{\gamma}}} \EE\Big[  e^{8\gamma \tilde q_2}  \Big]\nonumber \allowdisplaybreaks\\
    & \stackrel{(a)}{\leq}  2e^{9A}M_{\tilde q_2}^\gamma(-\phi) + 2e^{A} e^{-\phi\frac{\tilde C_2}{\sqrt{\gamma}}} \EE\Big[  e^{8\gamma q}  \Big]\nonumber \allowdisplaybreaks\\
    & \stackrel{(b)}{\leq} 2e^{9A}M_{\tilde q_2}^\gamma(-\phi) + 4e^{2A} e^{-\phi\frac{\nu_\gamma}{\sqrt{\gamma}}} e^{\nu_\gamma} \nonumber\\
    & \stackrel{(c)}{\leq} 2e^{9A}M_{\tilde q_2}^\gamma(-\phi) + 4e^{3A} e^{-\phi\frac{\nu_\gamma}{\sqrt{\gamma}}}
\end{align}
where (a) follows by using $\tilde q_2 \leq \tilde q_1$, given by the coupling  argument stated previously; (b) follows by using the bound on the MGF of $q$ in Eq. \eqref{eq: ssq_tildeq1_mgfbound} assuming $8\sqrt{ \gamma }\leq \phi_2$; and (c) follows by using $\nu_\gamma < A$, and $-\tilde C_2\leq -\nu_\gamma + A\sqrt{\gamma}$ and $\phi<1$.
Finally, using that the zero drift condition, i.e., $\EE[e^{-\phi \sqrt{\gamma}\tilde q_2}] = \EE[e^{-\phi \sqrt{\gamma}\tilde q_2^+}]$, and combining Eq. \eqref{eq: ssq_mgfexist_starteq2}, \eqref{eq: ssq_mgfexists2_term1} and \eqref{eq: ssq_mgfexists2_term2}, we get
\begin{align*}
    0 
    &= M_{\tilde q_2^+}^\gamma(-\phi) -M_{\tilde q_2}^\gamma(-\phi) \\
    &\leq M_{\tilde q_2}^\gamma(-\phi)\big(-A\gamma\phi +2e^{9A}\phi^2\gamma\big) + \big(\phi\gamma K_1 + 4e^{3A} \phi^2 \gamma \big)e^{-\phi\frac{\nu_\gamma}{\sqrt{\gamma}}}\\
    &\stackrel{(a)}{\leq}-\frac{1}{2} A\gamma\phi M_{\tilde q_2}^\gamma(-\phi) + \phi\gamma\big( K_1 + 4e^{3A} \big)e^{-\phi\frac{\nu_\gamma}{\sqrt{\gamma}}},
\end{align*}
where (a) follows by choosing $\phi$ such that
$0<\phi < \frac{A}{4e^{9A}}< 1$. 
Thus, for $\phi \leq \frac{A}{4e^{9A}}$ and for $8\sqrt{\gamma}< 8\sqrt{\gamma_0} = \frac{A}{4e^{9A}} $, we have that 
\begin{align*}
    M_{ q}^\gamma(-\phi) &\leq M_{\tilde q_2}^\gamma(-\phi) \leq  K_2 e^{-\phi\frac{\nu_\gamma}{\sqrt{\gamma}}},
\end{align*}
where $K_2 = 2(K_1 + 4e^{3A})e^{A}/A$ is a constant independent of $\gamma$. \\

\textbf{Upper bound for Lemma \ref{lem: ssq_mgfexist_c}:} For simplicity, we assume that $\nu_\gamma = -C_f\gamma^\alpha$. The argument works even if $\nu_\gamma $ deviates from $-C_f\gamma^\alpha$ according to the condition $ |\nu_\gamma +C_f\gamma^{ \alpha }| \leq \gamma^{\alpha + \beta}$, where $ \alpha\in \lfsmall{0,1/2}$, $\beta >0$, and $C_f >0$ are constants.

The argument for the upper bound on the MGF of the steady state queue length in this case is very similar to the argument for the upper bound for the MGF of $\tilde q_1$ given previously. Here, we create a coupled process $\{\tilde q_3(t)\}_{t=0}^\infty$ for which the sequence $\{c(t)\}_{t=0}^\infty$ is the same as for the original process, but there are no abandonment. Therefore,
\begin{equation*}
    \tilde q_3(t+1) = [\tilde q_3(t) + c(t) ]^+ = \tilde q_3(t) + c(t) + \tilde u_3(t). 
\end{equation*}
Using the similar argument as in Eq. \eqref{eq: ssq_mgfexist_coupling}, we get that $\tilde q_3(t) \geq q(t)$ for all $t\geq 0$. Then, using the similar argument as in Eq. \eqref{eq: ssq_mgfexist_starteq1}, since $|c(t)|\leq A$ for any $\phi>0$, we obtain
\begin{align}
\label{eq: ssq_mgfexist_starteq3}
    \EE[ \exp\big(\phi \gamma^\alpha(\tilde q_3(t+1)-\tilde u_3(t))\big) | \tilde q_3(t)]   &= \exp\big(\phi \gamma^\alpha\tilde q_3(t)\big)  \EE[ \exp\big(\phi \gamma^\alpha c(t)\big)  ]\nonumber\allowdisplaybreaks\\
    &\leq \exp\big(\phi \gamma^\alpha\tilde q_3(t)\big) \EE \Big[ 1 +  \phi \gamma^\alpha c(t) + \phi^2 \gamma^{2\alpha} c(t)^2 \exp\big(\phi \gamma^\alpha [c(t)]^+\big) \Big ] \nonumber\allowdisplaybreaks\\
    &\leq  \exp\big(\phi \gamma^\alpha\tilde q_3(t)\big)  \Big( 1 -  C_f\phi \gamma^{2\alpha}+\phi^2\gamma^{2\alpha} A^2 e^{A} \Big)\nonumber\allowdisplaybreaks\\
    &\stackrel{(a)}{\leq} \exp\big(\phi \gamma^\alpha\tilde q_3(t)\big)  \left( 1 - \frac{1}{2} C_f \phi \gamma^{2\alpha} \right),
\end{align}
where $(a)$ holds for $0<\phi < \frac{C_f}{2A^2e^{A}}$. Next, by the similar argument as in Eq. \eqref{eq: ssq_mgfexist_term1}, we get
\begin{align*}
	\EE[ & \exp\big(\phi \gamma^\alpha(\tilde q_3(t+1)-\tilde u_3(t))\big) | \tilde q_3(t)] \geq \EE[  \exp\big(\phi \gamma^\alpha\tilde q_3(t+1)\big) | \tilde q_3(t)]  -\phi \gamma^\alpha\EE[\tilde u_3(t) | \tilde q_3(t)].
\end{align*}
Thus, 
\begin{align*}
    \EE[ \exp\big(\phi \gamma^\alpha\tilde q_3(t+1)\big) | \tilde q_3(t)]  - \exp\big(\phi \gamma^\alpha\tilde q_3(t)\big)    & \leq - \frac{1}{2}C_f \phi \gamma^{2\alpha}\exp\big(\phi \gamma^\alpha\tilde q_3(t)\big) + \phi \gamma^\alpha\EE[\tilde u_3(t) | \tilde q_3(t)]\\
    &\stackrel{(a)}{\leq} - \frac{1}{2} C_f\phi \gamma^{2\alpha}\exp\big(\phi \gamma^\alpha\tilde q_3(t)\big),
\end{align*}
where $(a)$ holds whenever $\tilde q_3(t) >A$, in which case $\tilde q_3(t) + c(t) >0$ as $|c(t)|<A$ and so $\tilde u_3(t) =0$. Then, the Foster-Lyapunov Theorem implies
\[ \EE\big[\exp\big(\phi \gamma^\alpha\tilde q_3(t)\big)\big]<\infty \]
in steady state. Further, using $(\tilde q_3^+,\tilde q_3,\tilde u_3)$ to denote the steady state variables with $\tilde q_3^+ = \tilde q_3 + c+ \tilde u_3$, and since $\EE[\tilde q_3^+- \tilde q_3] =0$, it follows that $\EE[\tilde u_3] = -\EE[c] = C_f\gamma^\alpha$. Finally, using that
\[ \EE\big[ \exp\big(\phi \gamma^\alpha\tilde q_3^+\big) \big] = \EE\big[ \exp\big(\phi \gamma^\alpha\tilde q_3\big) \big], \]
we get that for $0<\phi < \frac{C_f}{2A^2e^{A}}$,
\begin{align*}
    \EE\big[  \exp\big(\phi \gamma^\alpha q\big) \big]\leq \EE\big[ \exp\big(\phi \gamma^\alpha\tilde q_3\big) \big] \leq \frac{2}{C_f \gamma^\alpha}\EE[\tilde u_3] = 2.
\end{align*}
The result in Lemma \ref{lem: ssq_mgfexist} follows by choosing $M_0 = \max\{2e^A, K_2,2\}$ and 
\[ \phi_0 = \min\left\{\frac{1}{22Ae^A},\frac{A}{4e^{9A}},\frac{C_f}{2A^2e^{A}}\right\}.\]

\subsection{Proof of Theorem \ref{thm: ssq_fast}}
\label{app: ssq_fast}

Recall that $(q^+, q, d,u)$ is distributed as the steady state of $(q(t+1), q(t), d(t),u(t))$, where $q^+ = q+ c-d+u$, and $c$ has the same distribution as $c(t)$. Thus, $q^+$ and $q$ have the same distribution. Also, we have $|c|\leq A$ and $0\leq u \leq A$.\\

\begin{lemma}
\label{lem: ssq_approximation_fast}
Under the same assumptions as in Theorem \ref{thm: ssq_fast}, we have the following upper bounds.
\begin{enumerate}[label=(\alph*), ref=\ref{lem: ssq_approximation_fast}.\alph*]
    \item \label{lem: ssq_approximation_fast_a} For any $\gamma \in (0,1)$ and for any $\phi \in \mathbb R$,
    \begin{align*}
        \Big| \EE\big[{e^{-\gamma^\alpha \phi  u }}\big]-\phi\gamma^\alpha \nu_\gamma -1\Big| \leq K^u_f \phi^2 e^{|\phi|A}  \gamma^{\min\{3\alpha,1\}}.
    \end{align*}
    \item \label{lem: ssq_approximation_fast_b} Suppose $c$ is a random variable distributed as $c(t) = a(t) - s(t)$, with $|c|\leq A$, $\EE[c]= \nu_\gamma $, and $\EE[c^2] = \sigma^2_\gamma +\nu_\gamma^2$. Then, for any $\phi \in \mathbb R$,
    \begin{equation*}
        \left| \EE\big[{e^{\gamma^\alpha \phi  c }}\big]   - \frac{\gamma^{2\alpha}\phi^2 (\sigma^2_\gamma+\nu_\gamma^2)}{2} - \gamma^\alpha \phi \nu_\gamma -  1 \right| \leq K_{f}^c |\phi|^3e^{|\phi|A} \gamma^{3\alpha}.
    \end{equation*}
    \item \label{lem: ssq_approximation_fast_c} Let $\phi_0$ be a constant as given in Lemma \ref{lem: ssq_mgfexist_a}. Then, for any $\phi < \phi_0/2$,
    \begin{equation*}
        \Big| \EE\big[{e^{\gamma^\alpha \phi  (q-d) }}\big] - \EE\big[{e^{\gamma^\alpha \phi  q }}\big] \Big| \leq K_f^d |\phi| \gamma,
    \end{equation*}
\end{enumerate}
where $K_f^u,K_f^c$ and $K_f^d$ are constants independent of $\gamma$ and $\phi$.
\end{lemma}

\begin{proof}[Proof: ]
By equating the drift of the Lyapunov function $f(q) = q$ to zero in steady state, we have
\begin{align*}
    0 = \EE[q^+ - q] = \EE[c-d+u] = \nu_\gamma -\gamma \EE[q] + \EE[u].
\end{align*}
Moreover, Lemma \ref{lem: ssq_mgfexist_c} implies $\gamma^\alpha\EE[q] \leq \phi_0^{-1}M_0$ and thus,
\begin{align*}
    |\EE[u] +\nu_\gamma| = \gamma \EE[q]  \leq \gamma^{1-\alpha}\phi_0^{-1}M_0,
\end{align*}
and the facts that $\nu_\gamma \leq (C_f+1)\gamma^\alpha$ and $\alpha \in \Big( 0,\frac{1}{2}\Big)$ imply
\begin{align*}
    \EE[u] = -\nu_\gamma + \gamma \EE[q] \leq (1+C_f) \gamma^\alpha + \gamma^{1-\alpha}\phi_0^{-1}M_0 \leq (1+C_f + \phi_0^{-1}M_0)\gamma^\alpha.
\end{align*}
Furthermore, since $u\leq A$, we have ${e^{-\gamma^\alpha \phi  u }} \leq {e^{|\phi|  A }}$. Thus, $\EE\big[{e^{-\gamma^\alpha \phi  u }}\big]$ is finite. Then, 
\begin{align}
\label{eq: ssq_approximation_fast_uterm}
    \Big| \EE\big[{e^{-\gamma^\alpha \phi  u }}\big]-\phi\gamma^\alpha \nu_\gamma -1\Big| &\leq \Big| \EE\big[{e^{-\gamma^\alpha \phi  u }}+\phi\gamma^\alpha u -1\big]\Big| + |\phi | \gamma^\alpha  |\EE[u] +\nu_\gamma| \nonumber\\
    &\stackrel{(a)}{\leq} \phi^2 \gamma^{2\alpha} \EE\big[ u^2 e^{\gamma^\alpha |\phi|  u } \big]+|\phi | \gamma\phi_0^{-1}M_0\nonumber\\
    &\stackrel{(b)}{\leq}\phi^2 \gamma^{2\alpha} A e^{|\phi|A} \EE[u]+|\phi | \gamma\phi_0^{-1}M_0\nonumber\\
    &\leq A(\phi_0^{-1}M_0 + C_f+1) \phi^2 \gamma^{3\alpha} e^{|\phi|A} +|\phi | \gamma\phi_0^{-1}M_0\nonumber\\
    &\leq A(2\phi_0^{-1}M_0 + C_f+1) \phi^2 \gamma^{\min\{3\alpha,1\}} e^{|\phi|A}
\end{align}
where (a) follows by using $|e^{x} - x-1| \leq x^2e^{|x|}$ for any $x \in \mathbb R$; and (b) follows by using $0\leq u\leq A$ and $\gamma \in (0,1)$. This proves Lemma \ref{lem: ssq_approximation_fast_a} with $K_f^u = A(2\phi_0^{-1}M_0 + C_f+1)$. For Lemma \ref{lem: ssq_approximation_fast_b},
\begin{align}
\label{eq: ssq_approximation_fast_cterm}
    \left| \EE\left[ {e^{\gamma^\alpha \phi  c }} \right] - \frac{\gamma^{2\alpha}\phi^2 (\sigma^2_\gamma+\nu_\gamma^2)}{2} - \gamma^\alpha \phi \nu_\gamma - 1 \right| &= \left| \EE\left[ e^{\gamma^\alpha \phi  c} - \frac{\phi^2 \gamma^{2\alpha} c^2}{2} - \phi \gamma^\alpha c - 1 \right]  \right| \nonumber\\ 
&\stackrel{(a)}{\leq} |\phi|^3 \gamma^{3\alpha} \left| \EE\left[ c^3 e^{\sqrt{\gamma}|\phi c|} \right] \right| \nonumber\\
&\stackrel{(b)}{\leq} A^3|\phi|^3 \gamma^{3\alpha} e^{|\phi| A},
\end{align}
where (a) follows from $|e^x-\frac{x^2}{2} -x-1 | \leq |x|^3 e^{|x|}$ and (b) follows from $|c|\leq A$. This proves Lemma \ref{lem: ssq_approximation_fast_b} with $K_f^c = A^3$. For Lemma \ref{lem: ssq_approximation_fast_c}, we have $\EE[  e^{\gamma^\alpha \phi  q } ]\leq M_0$ for any $\phi <\phi_0$ (by Lemma \ref{lem: ssq_mgfexist_c}), then, for any $k \geq 1$, and for $\phi < \phi_0/2$, we have
\begin{align}
\label{eq: ssq_approximation_fast_dmoment}
    \EE\left[\gamma^{\alpha k} q^k e^{\gamma^{\alpha }\phi q} \right] \leq \EE\left[\gamma^{\alpha k} q^k e^{\frac{1}{2}\gamma^{\alpha } \phi_0 q} \right] \leq \frac{2^k k!}{\phi_0^k} \EE\left[ e^{\gamma^{\alpha }  \phi_0 q} \right] \leq \frac{2^k k!}{\phi_0^k} M_0.
\end{align}
Then,
\begin{align}
\label{eq: ssq_approximation_fast_dterm}
    \Big| \EE\big[e^{\gamma^\alpha \phi  (q-d) }\big] - \EE\big[e^{\gamma^\alpha \phi  q }\big] \Big| &= \Big| \EE\Big[e^{\gamma^\alpha \phi  q } \big(e^{-\gamma^\alpha \phi  d } -1\big)\Big] \Big| \nonumber\\
    &\stackrel{(a)}{\leq } |\phi| \gamma^\alpha \EE\Big[d e^{\gamma^\alpha \phi  q }e^{\gamma^\alpha [-\phi  d]^+ } \Big]\nonumber\\
    &\stackrel{(b)}{\leq } |\phi| \gamma^\alpha \EE\Big[d e^{\gamma^\alpha [\phi]^+  q } \Big]\nonumber\\
    &\stackrel{(c)}{\leq } |\phi| \gamma \EE\left[ \gamma^\alpha q e^{\gamma^\alpha [\phi]^+ q} \right] \nonumber\\ 
    &\stackrel{(d)}{\leq } \frac{2M_0}{\phi_0} |\phi| \gamma.
\end{align}
where $(a)$ follows from $|e^x-1|\leq |x| e^{[x]^+}$; (b) follows by using $\phi  q + [-\phi d]^+ \leq \max\{0, \phi  q \}$ for any $\phi \in \mathbb R$ as $0\leq d \leq q$; (c) follows from using $\EE[d| q] = \gamma q$; and finally, (d) follows from the Eq. \eqref{eq: ssq_approximation_fast_dmoment}. Thus, Lemma \ref{lem: ssq_approximation_fast} is satisfied with $K_f^u = A(2\phi_0^{-1}M_0 + C_f+1)$, $K_f^c = A^3$ and $K_f^d = \frac{2M_0}{\phi_0}$.
\end{proof}

\begin{proof}[Proof of Theorem \ref{thm: ssq_fast}: ] 
Recall Lemma \ref{lem: ssq_mgfexist_c} implies  that, for any $0<\phi<\phi_0$, we have $\EE[e^{\gamma^\alpha \phi q}] \leq M_0$. Also, for $\phi<0$, we have $\EE[e^{\gamma^\alpha \phi q}]<1\leq M_0$. Thus, the MGF is bounded for any $\phi<\phi_0$. As mentioned in Section \ref{sec: ssq_model}, we have $q(t+1) u(t)=0$, and so for any $\phi \in \mathbb R$, and at any time $t$, we also have
\begin{equation*}
    \left(e^{\gamma^\alpha \phi q(t+1)} -1\right)\left(e^{-\gamma^\alpha \phi u(t) }-1\right) =0.
\end{equation*}
Taking expectation with respect to the steady state distribution for $\phi<\phi_0$, we get
\begin{align}
\label{eq: ssq_fast_starteq}
 \EE\left[{e^{-\gamma^\alpha \phi  u }}\right]-1 &=  \EE\left[ {e^{\gamma^\alpha \phi (q^+-u) }}\right] - \EE\left[ {e^{\gamma^\alpha \phi  q^+ }} \right] \nonumber\\
 &\stackrel{(a)}{=} \EE\left[ {e^{\gamma^\alpha \phi  (q+c-d) }} \right] - \EE\left[ {e^{\gamma^\alpha \phi q}} \right]  \nonumber\\
 &\stackrel{(b)}{=} \EE\left[{e^{\gamma^\alpha \phi (q-d) }} \right] \EE\left[ {e^{\gamma^\alpha \phi c }} \right] - \EE\left[{e^{\gamma^\alpha \phi q }}\right]  \nonumber\\
 &=  \EE\left[{e^{\gamma^\alpha \phi  q }}\right] \left(\EE\left[{e^{\gamma^\alpha \phi c }}\right] -1 \right) + \left( \EE\left[ {e^{\gamma^\alpha \phi (q-d) }}\right] - \EE\left[ {e^{\gamma^\alpha \phi q }} \right] \right) \EE\left[ {e^{\gamma^\alpha \phi  c }} \right],
\end{align}
where (a) follows from the fact that $q$ and $q^+$ have same distribution, and that $q^+-u = q+c-d$; and (b) follows since $c$ is independent of $q$ and $d$. Next, we use that
\begin{align*}
   \EE\left[{e^{-\gamma^\alpha \phi  u }}\right]-1 & =\phi \gamma^\alpha\nu_\gamma + \mathcal{E}_f^u(\gamma,\phi) \\
   \EE\left[{e^{\gamma^\alpha \phi  c }}\right] -1 &= \frac{1}{2} \phi^2 \gamma^{2\alpha} (\sigma^2_\gamma + \nu_\gamma^2) + \phi \gamma^\alpha \nu_\gamma + \mathcal{E}_f^c(\gamma,\phi) \\
   \EE\left[{e^{\gamma^\alpha \phi (q-d) }}\right]-\EE\left[{e^{\gamma^\alpha \phi q }}\right] &= \mathcal{E}_f^d(\gamma,\phi)
\end{align*}
to obtain
\begin{align*}
    \left( \frac{1}{\gamma^\alpha}\nu_\gamma + \frac{1}{2}\phi (\sigma^2_\gamma + \nu_\gamma^2) \right) \EE\big[{e^{\gamma^\alpha \phi  q }}\big] - \frac{1}{\gamma^\alpha} \nu_\gamma = \frac{1}{\phi \gamma^{2\alpha}} \mathcal{E}_f(\gamma,\phi),
\end{align*}
where 
\begin{align*}
     \mathcal{E}_f(\gamma,\phi) = -\EE\left[{e^{\gamma^\alpha \phi q }}\right]\mathcal{E}_f^c(\gamma,\phi) + \mathcal{E}_f^u(\gamma,\phi)  -\EE\left[{e^{\gamma^\alpha \phi c }}\right]\mathcal{E}_f^d(\gamma,\phi).
\end{align*}
Lemma \ref{lem: ssq_approximation_fast} implies that for every $\phi<\phi_0/2$,
\begin{align}
\label{eq: ssq_error_fast}
     \frac{1}{|\phi| \gamma^{2\alpha}} |\mathcal{E}_f(\gamma,\phi)| &\leq \frac{1}{|\phi| \gamma^{2\alpha}} \Big( \EE\left[{e^{\gamma^\alpha \phi q }}\right] |\mathcal{E}_f^c(\gamma,\phi)| + |\mathcal{E}_f^u(\gamma,\phi)| +  \EE\left[{e^{\gamma^\alpha \phi c }}\right] |\mathcal{E}_f^d(\gamma,\phi) |\Big)\nonumber\\
     &\stackrel{(a)}{\leq} \frac{1}{|\phi| \gamma^{2\alpha}} \Big( M_0|\mathcal{E}_f^c(\gamma,\phi)| + |\mathcal{E}_f^u(\gamma,\phi)|+  e^{|\phi|A}|\mathcal{E}_f^d(\gamma,\phi) |\Big)\nonumber\\
     &\stackrel{(b)}{\leq } M_0K_{f}^c \phi^2e^{|\phi|A} \gamma^{\alpha} + K^u_f |\phi| e^{|\phi|A}  \gamma^{\min\{\alpha,1-2\alpha\}} + e^{|\phi|A} K_f^d \gamma^{1-2\alpha}\nonumber\\
     &\stackrel{(c)}{\leq }  e^{A}\big(M_0K_{f}^c  + K^u_f   + K_f^d \big)  \gamma^{\min\{\alpha,1-2\alpha\}},
\end{align}
where (a) follows by using Lemma \ref{lem: ssq_mgfexist_c} and $|c|\leq A$; (b) follows by using the bound in Lemma \ref{lem: ssq_approximation_fast}; and (c) follows by taking $|\phi|< \phi_0/2<1$. Finally, since $\alpha \in\Big(0,\frac{1}{2}\Big)$, we have that, for $|\phi|< \phi_0/2<1$,
\begin{align*}
    \limg \frac{1}{|\phi| \gamma^{2\alpha}} |\mathcal{E}_f(\gamma,\phi)| =0.
\end{align*}
 Moreover, we have $\limg \frac{1}{\gamma^\alpha} \nu_\gamma = C_f$ and $\limg \sigma^2_\gamma  + \nu_\gamma^2 = \sigma^2$. Then, by taking the limit as $\gamma \rightarrow 0$, for any $\phi < \phi_0/2$, we have
\begin{align*}
   \limg \EE\left[{e^{\gamma^\alpha \phi  q }}\right] =  \frac{1}{ 1 - \phi \frac{\sigma^2}{2C_f}  }.
\end{align*}
It is easy to observe that the RHS is the MGF of an exponential random variable with mean $\frac{\sigma^2}{2C_f} $. Finally, the result follows from combining this with Lemma \ref{COR: MGF_CONVERGENCE}.
\end{proof}

\subsection{Proof of Theorem \ref{thm: ssq_crit}}
\label{app: ssq_crit}
Recall that $(q^+, q, d,u)$ is distributed as the steady state of $(q(t+1), q(t), d(t),u(t))$, where $q^+ = q+ c-d+u$ and $c$ has the same distribution as $c(t)$. Thus, $q^+$ and $q$ have the same distribution. Also, we have $|c|\leq A$ and $0\leq u \leq A$.\\

\begin{lemma}
\label{lem: ssq_approximation}
Under the same assumptions as in Theorem \ref{thm: ssq_crit}, we have the following upper bounds
\begin{enumerate}[label=(\alph*), ref=\ref{lem: ssq_approximation}.\alph*]
    \item \label{lem: ssq_approximation_a} For any $\gamma \in (0,1)$ and, for any $\phi \in \mathbb R$,
    \begin{equation*}
        \Big| \MM{-u} +\phi\sqrt{\gamma} \EE[u] -1\Big| \leq K_c^u \phi^2 e^{|\phi|A}  \gamma^{\frac{3}{2}}.
    \end{equation*}
    \item \label{lem: ssq_approximation_b} Suppose $c$ is a random variable distributed as $c(t) = a(t) - s(t)$, with $|c|\leq A$, $\EE[c]= \nu_\gamma $, and $\EE[c^2] = \sigma^2_\gamma +\nu_\gamma^2$. Then, for any $\phi \in \mathbb R$, we have
    \begin{equation*}
        \left| \MM{c}   - \frac{\gamma\phi^2 (\sigma^2_\gamma+\nu_\gamma^2)}{2} - \sqrt{\gamma} \phi \nu_\gamma - 1 \right| \leq K_{c}^c |\phi|^3e^{|\phi|A} \gamma^{\frac{3}{2}}.
    \end{equation*}
    \item \label{lem: ssq_approximation_c} Let $\phi_0$ be a constant as in Lemma \ref{lem: ssq_mgfexist_a}. Then, for any $\phi < \phi_0/2$, we have
    \begin{equation*}
        \left| (\MM{q-d} - \MM{q})\MM{c} +\phi \gamma \frac{d}{d\phi} \MM{q} \right| \leq K_c^d \phi^2 e^{|\phi| A} \gamma^\frac{3}{2},
    \end{equation*}
\end{enumerate}
where $K_c^u,K_c^c$ and $K_c^d$ are constants, independent of $\gamma$ and $\phi$.
\end{lemma}

\begin{proof}[Proof: ]
Recall that Lemma \ref{lem: ssq_mgfexist_a} implies $\sqrt{\gamma}\EE[q] \leq \phi_0^{-1}M_0$. Combining this with $\nu_\gamma \leq (C_c+1)\sqrt{\gamma}$, we have
\begin{align*}
    \gamma \EE[q] -\nu_\gamma \leq \sqrt{\gamma}\phi_0^{-1}M_0 + (C_c+1) \sqrt{\gamma}\leq \sqrt{\gamma}(\phi_0^{-1}M_0 + C_c+1),
\end{align*}
where we used $\nu_\gamma \leq (C_c+1)\sqrt{\gamma}$ according to the condition given in Theorem \ref{thm: ssq_crit}. Also, since $u\leq A$, we have ${e^{-\sqrt{\gamma} \phi  u }} \leq {e^{|\phi|  A }}$ and  thus $\EE\big[{e^{-\sqrt{\gamma} \phi  u }}\big]$ is finite.
Moreover, in steady state, we have $\EE[q^+ - q] = 0$, then $\EE[u] = \EE[d - c] \leq \gamma\EE[q] -\nu_\gamma\leq \sqrt{\gamma}K_c^q$, where $K_c^q = \phi_0^{-1}M_0 + C_c+1 $. It follows that
\begin{align}
\label{eq: ssq_lhs}
   \Big| \MM{-u} +\phi\sqrt{\gamma} \EE[u] -1\Big| &= \Big| \EE\big [e^{-\phi\sqrt{\gamma} u } + \sqrt{\gamma}\phi u -1\big]\Big| \nonumber\\
    & \stackrel{(a)}{\leq}  \big|  \phi^2 \gamma  \EE\big [ u^2e^{|\phi|\sqrt{\gamma} u} \big] \big|\nonumber\\
    & \stackrel{(b)}{\leq} |\phi|^2 \gamma A e^{|\phi|\sqrt{\gamma} A} \times \EE[u]\nonumber\\
    &\leq A K_q e^{|\phi|A} |\phi|^2 \gamma^{\frac{3}{2}} ,
\end{align}
where (a) follows by using $|e^{x} - x-1| \leq x^2e^{|x|}$ for any $x \in \mathbb R$; and (b) follows by using $0\leq u\leq A$. For the next part, we have
\begin{align}
\label{eq: ssq_c_term2}
\left| \MM{c} - \frac{\gamma\phi^2 (\sigma^2_\gamma+\nu_\gamma^2)}{2} - \sqrt{\gamma} \phi \nu_\gamma - 1 \right| &= \left| \EE\left[ \left( e^{\phi \sqrt{\gamma} c} - \frac{\phi^2 \gamma c^2}{2} - \phi \sqrt{\gamma}c -  1\right) \right] \right|\nonumber\\ 
&\stackrel{(a)}{\leq}  \Big| \EE\big [ (|\phi| \sqrt{\gamma} c)^3 e^{ \sqrt{\gamma}|\phi c|}  \big]  \Big| \nonumber\\
&\stackrel{(b)}{\leq} A^3|\phi|^3 \gamma^{\frac{3}{2}} e^{|\phi| A},
\end{align}
where (a) follows by using $|e^x-\frac{x^2}{2} -x-1 | \leq |x|^3 e^{|x|}$ and (b) follows by using $|c|\leq A$. For the final part, first note that, for any $0<\phi<\phi_0<1$ Lemma \ref{lem: ssq_mgfexist_a} implies that 
\begin{align*}
    \MM{q} \leq M_0 e^{\frac{\phi\nu_\gamma}{\sqrt{\gamma}}} \leq M_0e^{\phi ([C_c]^++1)} \leq  M_0e^{[C_c]^++1} = M_c.
\end{align*}
Further, for $\phi<0$, we have $\MM{q}\leq 1\leq M_c$. Then, for $0<\phi < \phi_0/2$,
\begin{align}
\label{eq: ssq_highermoment}
    \EE\left[(\sqrt{\gamma}q)^k e^{\sqrt{\gamma} \phi q}\right] \leq \EE\left[\big(\sqrt{\gamma}q\big)^k e^{\frac{1}{2}\sqrt{\gamma} \phi_0 q}\right] \leq \frac{2^k k!}{\phi_0^k} \EE\big[ e^{\sqrt{\gamma} \phi_0 q}\big] \leq \frac{2^k k!}{\phi_0^k} M_c. 
\end{align}
Using this, we get that
\begin{align}
\label{eq: ssq_rhs2}
\left| \MM{q-d} - \MM{q} +\phi \gamma \frac{d}{d\phi} \MM{q} \right| &=
    \Big| \EE \big[ e^{\phi \sqrt{\gamma} (q-d) } \big] - \EE \big[ e^{\phi \sqrt{\gamma} q } \big]  +\phi\gamma\EE[\sqrt{\gamma}q e^{\phi \sqrt{\gamma}  q}]  \Big| \nonumber\\
    &\stackrel{(a)}{=} \Big| \EE \big[ e^{\phi \sqrt{\gamma} (q-d) } \big] - \EE \big[ e^{\phi \sqrt{\gamma} q } \big]  +\phi\sqrt{\gamma}\EE[d e^{\phi \sqrt{\gamma}  q}]  \Big|\nonumber \\
    &\leq  \EE \big[ e^{\phi \sqrt{\gamma} q } \big|e^{-\phi \sqrt{\gamma} d} - 1+ \phi\sqrt{\gamma} d\big|   \big]\allowdisplaybreaks \nonumber\\
    &\stackrel{(b)}{\leq} \phi^2 \gamma \EE \Big[ d^2 e^{\phi \sqrt{\gamma} q + [-\phi \sqrt{\gamma} d]^+ } \Big]\allowdisplaybreaks \nonumber\\
    &\stackrel{(c)}{\leq} \phi^2 \gamma \EE \Big[ d^2  e^{[\phi]^+ \sqrt{\gamma} q }  \Big]\allowdisplaybreaks \nonumber\\
    & \stackrel{(d)}{\leq} \phi^2 \gamma^{\frac{3}{2}} \EE \Big[ \big(\sqrt{\gamma} q + \gamma q^2 \big)  e^{[\phi]^+\sqrt{\gamma} q }  \Big]\allowdisplaybreaks \nonumber\\
    & \stackrel{(e)}{\leq} \phi^2 \gamma^{\frac{3}{2}} \left( \frac{2}{\phi_0}+\frac{8}{\phi_0^2} \right)M_c,
\end{align}
where (a) follows by using $\EE[d| q] = \gamma q$; (b) follows by using $e^x - x -1 \leq x^2 e^{[x]^+}$ for all $x \in \mathbb R$; (c) follows by using $\phi  q + [-\phi d]^+ \leq \max\{0, \phi  q \}$ for any $\phi \in \mathbb R$ as $0\leq d \leq q$; (d) follows from $\EE[d^2 | q] = \gamma(1-\gamma)q + \gamma^2 q^2 \leq \sqrt{\gamma}(\sqrt{\gamma} q + \gamma q^2)$; and (e) follows from Eq. \eqref{eq: ssq_highermoment}. Further, we have that,
\begin{align}
\label{eq: ssq_approximation_crit_cd}
	\Big|(\MM{c} -1) (\MM{q-d} - \MM{q})\Big| &= \Big| \EE\Big [\big( e^{\phi \sqrt{\gamma} c} -  1\big) \Big]  \EE \Big[ e^{\phi \sqrt{\gamma} q}\big( e^{-\phi \sqrt{\gamma} d}-1\big) \Big]\Big|\allowdisplaybreaks \nonumber\\
    &\stackrel{(a)}{\leq}\Big| \EE\Big [\phi \sqrt{\gamma} c e^{\sqrt{\gamma} [\phi c]^+ } \Big]  \EE \Big[ -\phi \sqrt{\gamma}d e^{\phi \sqrt{\gamma} q + \sqrt{\gamma}[-\phi d]^+} \Big]\Big|\allowdisplaybreaks \nonumber\\
    &\stackrel{(b)}{\leq} \phi^2 \gamma A e^{|\phi| A}  \EE\Big [ d  e^{[\phi]^+  \sqrt{\gamma} q}  \Big]\allowdisplaybreaks \nonumber\\
    &\stackrel{(c)}{=}  \phi^2 \gamma^{\frac{3}{2}} A e^{|\phi|  A}  \EE\Big[ \sqrt{\gamma}q  e^{[\phi]^+ \sqrt{\gamma} q} \Big] \allowdisplaybreaks \nonumber\\
    &\stackrel{(d)}{=} \frac{2AM_c}{\phi_0} \phi^2 \gamma^{\frac{3}{2}} e^{|\phi|  A} ,
\end{align}
where (a) follows from $e^x -1 \leq x^2 e^{[x]^+}$; (b) follows by using $|c|\leq A$, $\sqrt{\gamma}[\phi c]^+ \leq |\phi| A$ and $\phi  q + [-\phi d]^+ \leq [\phi]^+  q $ for any $\phi \in \mathbb R$ and $0\leq d \leq q$; (c) follows from $\EE[d| q] = \gamma q$; and (d) follows from Eq. \eqref{eq: ssq_highermoment}. By combining this with Eq. \eqref{eq: ssq_rhs2}, we get the result.
\end{proof}

\begin{proof}[Proof of Theorem \ref{thm: ssq_crit}: ]
First note that, for any $0<\phi<\phi_0<1$, Lemma \ref{lem: ssq_mgfexist_a} implies that 
\begin{align*}
    \MM{q} \leq M_0 e^{\frac{\phi\nu_\gamma}{\sqrt{\gamma}}} \leq M_0e^{\phi ([C_c]^++1)} \leq  M_0e^{[C_c]^++1} = M_c.
\end{align*}
Further, for $\phi<0$, we have $\MM{q}\leq 1\leq M_c$.
As mentioned before, $q^+ u=0$, and so, for any $\phi < \phi_0$,
\begin{equation*}
  \EE\Big[ \Big(e^{\sqrt{\gamma} \phi q^+} -1 \Big)\Big(e^{-\sqrt{\gamma} \phi u }-1\Big) \Big] =0.
\end{equation*}
Thus,
\begin{align}
\label{eq: ssq_starteq}
 \MM{-u}-1 &=  \MM{q^+-u} - \MM{q^+} \nonumber \allowdisplaybreaks\\
 &\stackrel{(a)}{=} \MM{q+c-d} - \MM{q} \nonumber\allowdisplaybreaks\\
 &\stackrel{(b)}{=}  \MM{q-d}\MM{c} - \MM{q} \nonumber\allowdisplaybreaks\\
 &=  \MM{q}(\MM{c} -1) + (\MM{q-d}-\MM{q})\MM{c},
\end{align}
where (a) follows from $q$ and $q^+$ having the same distribution; and (b) follows from $c$ being independent of $q$ and $d$. 
Now we use that
\begin{align*}
    \MM{-u} +\phi\sqrt{\gamma} \EE[u] -1 & = \mathcal{E}_c^u(\gamma,\phi),\\
    \MM{c}   - \frac{\gamma\phi^2 (\sigma^2_\gamma+\nu_\gamma^2)}{2} - \sqrt{\gamma} \phi \nu_\gamma -  1 & = \mathcal{E}_c^c(\gamma,\phi),\\
    \big(\MM{q-d} - \MM{q}\big)\MM{c} +\phi \gamma \frac{d}{d\phi} \MM{q} & = \mathcal{E}_c^d(\gamma,\phi),
\end{align*}
to get
\begin{align*}
    -\MM{q} \left(\frac{\gamma\phi^2 (\sigma^2_\gamma+\nu_\gamma^2)}{2} + \sqrt{\gamma} \phi \nu_\gamma\right) +\phi \gamma \frac{d}{d\phi} \MM{q}	-\sqrt{\gamma} \phi \EE[u] = \mathcal{E}_c(\gamma,\phi),
\end{align*}
where
\begin{align*}
    \mathcal{E}_c(\gamma,\phi) = \MM{q}  \mathcal{E}_c^c(\gamma,\phi) +  \mathcal{E}_c^d(\gamma,\phi) -  \mathcal{E}_c^u(\gamma,\phi).
\end{align*}
After dividing by $\gamma \phi$ on both sides, the above equation can be written as
\begin{align}
\label{eq: ssq_crit_diffeq}
	-\MM{q} \left(\frac{\phi (\sigma^2_\gamma+\nu_\gamma^2)}{2} +  \frac{\nu_\gamma}{\sqrt{\gamma}}\right) +\frac{d}{d\phi} \MM{q} - \frac{1}{\sqrt{\gamma}}\EE[u] = \frac{1}{\phi\gamma}\mathcal{E}_c(\gamma,\phi).
\end{align}
Furthermore, for any $\phi<\phi_0/2$, Lemma \ref{lem: ssq_approximation} implies that
\begin{align*}
    \frac{1}{|\phi|\gamma}|\mathcal{E}_c(\gamma,\phi)| 
    &\leq K_c^c M_c \phi^2 \gamma^{\frac{1}{2}} e^{|\phi| A} +  K_c^d |\phi| \gamma^{\frac{1}{2}} e^{|\phi| A} + K_c^u |\phi| \gamma^{\frac{1}{2}} e^{|\phi| A} \\
    &\leq K_c \big(|\phi| + \phi^2\big)e^{|\phi|A}\sqrt{\gamma},
\end{align*}
where $K_c = K_c^c M_c + K_c^d+ K_c^u$. By denoting
\[ G^{\gamma}(\phi) = \exp\left( -\frac{\phi^2(\sigma^2_\gamma+\nu_\gamma^2)}{4} -  \frac{\phi\nu_\gamma}{\sqrt{\gamma}} \right), \]
we get
 \begin{align*}
 	\frac{d}{d\phi} \big(\MM{q}G^{\gamma}(\phi)\big) - \frac{1}{\sqrt{\gamma}}\EE[u]G^{\gamma}(\phi) = \frac{1}{\phi\gamma}\mathcal{E}_c(\gamma,\phi) G^{\gamma}(\phi). 
 \end{align*}
 Since the above equation holds for any $\phi \leq \phi_0/2$, we can integrate both sides over $\phi$ to get that, for any $\phi \leq \phi_0/2$,
 \begin{align*}
 	\MM{q}G^{\gamma}(\phi)- \frac{1}{\sqrt{\gamma}} \EE[u] \int_{-\infty}^\phi G^\gamma(s)ds = \int_{-\infty}^\phi \frac{1}{s\gamma}\mathcal{E}_c(\gamma,s) G^{\gamma}(s) ds
 \end{align*}
 Note that, $G^\gamma(\cdot)$ matches with the (scaled) distribution function of a Gaussian random variable with finite variance and so, we have that
 \[ \int_{-\infty}^\infty (|s| + s^2) e^{|s|A}  G^\gamma(s)ds < K' \]
 for some $K'<\infty$. Thus,  for any $\phi<\phi_0/2$, we have
 \begin{align*}
     \limg\int_{-\infty}^\phi \frac{1}{|s|\gamma}|\mathcal{E}_c(\gamma,s)| G^{\gamma}(s) ds \leq \limg K_c\sqrt{\gamma} \int_{-\infty}^\infty (|s| + s^2) e^{|s|A}  G^\gamma(s)ds =0
 \end{align*}
   Further, $\lim_{\gamma \rightarrow 0} \sigma^2_\gamma + \nu_\gamma^2 = \sigma^2$ and $\lim_{\gamma \rightarrow 0} \nu_\gamma/\sqrt{\gamma} = C_c$, and so
   \[ \lim_{\gamma \rightarrow 0} G^\gamma(\phi) = G(\phi) = \exp\left( -\frac{\phi^2\sigma^2}{4} -  \phi C_c \right). \]
   Setting $\phi = 0$ in the above equation and using the condition $M_q^\gamma(0) G^\gamma(0) =1 $,  we get 
\begin{align*}
    \lim_{\gamma \rightarrow 0} \frac{1}{\sqrt{\gamma}} \EE[u] = \left( \lim_{\gamma \rightarrow 0} \int_{-\infty}^0 G^\gamma(s)ds  \right)^{-1} =  \left( \int_{-\infty}^0 G(s)ds \right)^{-1}.
\end{align*}
Therefore, for $\phi < \phi_0/2$, we have
\begin{align*}
\lim_{\gamma \rightarrow 0} \MM{q} &= G^{-1}(\phi) \left( \int_{-\infty}^0 G(s)ds \right)^{-1} \int_{-\infty}^\phi G(s)ds = \exp\left( \frac{\phi^2\sigma^2}{4} +C\phi \right) \frac{\int_{-\infty}^\phi \exp\left( -\frac{s^2\sigma^2}{4} - C_cs \right)ds}{\int_{-\infty}^0 \exp\left( -\frac{s^2\sigma^2}{4} - C_c s \right)ds}.
\end{align*}
This defines the MGF of a truncated Gaussian random variable as in Theorem \ref{thm: ssq_crit}. The result follows from combining this with Lemma \ref{COR: MGF_CONVERGENCE}.
\end{proof}

\subsection{Proof of Theorem \ref{thm: ssq_slow}}
\label{app: ssq_slow}

Recall that $(q^+, q, d,u)$ is distributed as the steady state of $(q(t+1), q(t), d(t),u(t))$, where $q^+ = q+ c-d+u$ and $c$ has the same distribution as $c(t)$. Thus, $q^+$ and $q$ have the same distribution. Also, we have $|c|\leq A$ and $0\leq u \leq A$.\\

\begin{lemma}
\label{lem:ssq_slow_mgfapprox}
Suppose $\bar q = q - \frac{\nu_\gamma}{\gamma}$. Under the same assumptions as in Theorem \ref{thm: ssq_slow}, there exists $\gamma_s \in (0,1)$ such that for any $\phi \in \big(\frac{-\phi_0}{2},\frac{\phi_0}{2}\big)$ and for any $\gamma <\gamma_s$, we have the following results. 
\begin{enumerate}[label=(\alph*), ref=\ref{lem:ssq_slow_mgfapprox}.\alph*]
    \item We have\label{lem:ssq_slow_mgfapprox_a}
    \begin{equation*}
        \Big| \MM{-u} -1\Big|e^{-\frac{\phi \nu_\gamma}{\sqrt{\gamma}}} \leq K_s^u \sqrt{\gamma}|\phi|  e^{-\frac{\phi_0 \nu_\gamma} {\sqrt{\gamma}}}e^{-\frac{\phi \nu_\gamma}{\sqrt{\gamma}}}.
    \end{equation*}
    \item \label{lem:ssq_slow_mgfapprox_b} We have
    \begin{equation*} 
        \left| (\MM{\bar q-d} \MM{c} - \MM{\bar q}) -\frac{\phi^2 \gamma\bar \sigma^2_{\gamma} }{2}\MM{\bar q}  +\phi \gamma \frac{d}{d\phi} \MM{\bar q} \right| 
 \leq K_s^d |\phi|^2 \gamma^\frac{3}{2},
    \end{equation*}
\end{enumerate}
where  $\bar \sigma^2_{\gamma} =\sigma^2_\gamma + \nu_\gamma(1-\gamma) $, and $K_s^u$ and $K_s^d$ are constants independent of $\gamma$ and $\phi$.
\end{lemma}

The proof of Lemma \ref{lem:ssq_slow_mgfapprox} is similar to that of Lemma \ref{lem: ssq_approximation}, which uses first and second order approximations of the exponential function. 

\begin{proof}[Proof: ]
Recall that $q^+ u =0$ as mentioned in Section \ref{sec: ssq_model}. Combining this with the fact  that $q$ and $q^+$ have the same distribution and also the result from Lemma \ref{lem: ssq_mgfexist}, we have
\begin{align}
\label{eq: ssq_u_verysmall}
    \EE[u] = \EE \left[u e^{-\phi \sqrt{\gamma} q^+}\right] \leq A \EE \left[ e^{-\phi \sqrt{\gamma} q^+}\right]  = A M_{q}^\gamma(-\phi) \leq AM_0 e^{-\frac{\phi \nu_\gamma}{\sqrt{\gamma}}}.
\end{align}
Since this holds for any $0<\phi<\phi_0$, we also get that, $\EE[u] \leq AM_0 e^{-\frac{\phi_0 \nu_\gamma}{\sqrt{\gamma}}}$. 
Combining this with the fact that $|u|\leq A$, and $e^{x} -1 \leq |x|e^{|x|}$, we obtain
\begin{align*}
     \big| \MM{ -u} -  1\Big|
     = \big| \EE\big[ e^{-\sqrt{\gamma}\phi u} -1 \big]\Big| \leq \sqrt{\gamma}|\phi| e^{A} \EE [u]   \leq \sqrt{\gamma}|\phi| e^{A} AM_0 e^{-\frac{\phi_0 \nu_\gamma} {\sqrt{\gamma}}}
\end{align*}
This implies that, for any $\phi>-\phi_0/2$,
\begin{equation*}
    \limg \frac{1}{|\phi|\gamma} \big| \MM{ -u} -  1\big|e^{-\frac{\phi \nu_\gamma}{\sqrt{\gamma}}} = 0.
\end{equation*}
For Lemma \ref{lem:ssq_slow_mgfapprox_b}, we first consider
\begin{align*} 
\EE\Big  [ |c-d|^3 e^{\sqrt{\gamma} [\phi(c - d)]^+ } \Big| q \Big] &\stackrel{(a)}{\leq}  \EE\Big  [ (A+d)^3 e^{(A +d) } \Big| q \Big] \stackrel{(b)}{\leq}  6e^{2A}  \EE\big  [ e^{2d } \big| q \big],
\end{align*}
where (a) follows from using $|c-d|\leq A+d$ and $|\phi|\leq 1$; (b) follows from $x^3 \leq 6e^x$. Thus,
\begin{align}
\label{eq: ssq_lemslow_cd}
    \left| \EE\left[  e^{\sqrt{\gamma} \phi\bar q}\left(  e^{\sqrt{\gamma} \phi(c - d)} - \frac{\phi^2 \gamma}{2}(c-d)^2 - \phi \sqrt{\gamma} (c-d) -1 \right) \right] \right| 
    &\stackrel{(a)}{\leq}|\phi|^3 \gamma^{\frac{3}{2}}  \EE\Big  [ |c-d|^3 e^{\sqrt{\gamma} [\phi(c - d)]^+ + \sqrt{\gamma} \phi\bar q} \Big] \nonumber\allowdisplaybreaks\\
    &\stackrel{(b)}{\leq} 6|\phi|^3 \gamma^{\frac{3}{2}}  e^{2A}\EE\Big  [ e^{2d + \sqrt{\gamma} \phi q } \Big]e^{-\frac{\phi \nu_\gamma}{\sqrt{\gamma}}} \nonumber\allowdisplaybreaks\\
     &\leq 6|\phi|^3 \gamma^{\frac{3}{2}}  e^{2A}\EE\big  [e^{4d}\big]^{\frac{1}{2}} \EE\big  [ e^{\sqrt{\gamma} \phi q } \big]^{\frac{1}{2}}e^{-\frac{\phi \nu_\gamma}{\sqrt{\gamma}}} \nonumber\allowdisplaybreaks\\
    & \stackrel{(c)}{\leq} 6M_0|\phi|^3 \gamma^{\frac{3}{2}}  e^{2A} e^{8\nu_\gamma} e^{\frac{\phi \nu_\gamma}{\sqrt{\gamma}}} e^{-\frac{\phi \nu_\gamma}{\sqrt{\gamma}}} \nonumber\allowdisplaybreaks\\
    & = 6M_0e^{10A} |\phi|^3 \gamma^{\frac{3}{2}} 
\end{align}
where (a) follows from using $|e^x - \frac{x^2}{2} - x -1| \leq |x|^3e^{[x]^+}$; (b) follows from the previous equation; (c) follows from, for $16\sqrt{\gamma} < \phi_0$,
\begin{equation}
\label{eq: ssq_slow_2dexpectation}
    \EE \big[e^{2d}\big] = \EE[\big(1+\gamma (e^2-1)\big)^q] \leq \EE[e^{8\gamma q}] \leq M_0 e^{8\nu_\gamma}.
\end{equation}
Further,
\begin{align*}
    \EE\big  [ (c-d)^2 \big|q \big] - \sigma^2_\gamma - \nu_\gamma(1-\gamma)  &=  \gamma(1-\gamma) q + (\nu_\gamma-\gamma q)^2 - \nu_\gamma(1-\gamma)  =  \gamma(1-\gamma) \bar q +  \gamma^2 \bar q^2,
\end{align*}
where the inequality follows from using $\gamma q = \gamma \bar q + \nu_\gamma$. Also, Lemma \ref{lem: ssq_mgfexist} implies that $\MM{\bar q} = e^{-\frac{\phi \nu_\gamma}{\sqrt{\gamma}}} \MM{q} \leq M_0$ for any $\phi \in (-\phi_0,\phi_0)$ and $\gamma<\gamma_0$. It follows that
\begin{align*}
    \EE\left[(\sqrt{\gamma}\bar q)^k e^{\sqrt{\gamma } \phi \bar q}\right] \leq \frac{2^k k!}{\phi_0^k} \EE\left[ e^{ \frac{1}{2}\sqrt{\gamma } \phi_0 \bar q} e^{\sqrt{\gamma } \phi \bar q} \right] \leq \frac{2^k k!}{\phi_0^k} M_0,
\end{align*}
where the last inequality holds for any $\phi \in \left(-\frac{\phi_0}{2},\frac{\phi_0}{2}\right)$ and $\gamma < \gamma_0$. Thus, for $|\phi|\leq \phi_0/2 $, we have
\begin{align}
\label{eq: ssq_lemslow_cd1}
    \Big| \EE\Big  [ e^{\sqrt{\gamma} \phi\bar q} \Big( (c-d)^2 - (\sigma^2_\gamma + \nu_\gamma(1-\gamma))\Big) \Big] \Big| &\leq \Big|  \EE\Big  [ (\gamma(1-\gamma) \bar q +  \gamma^2 \bar q^2) e^{\sqrt{\gamma} \phi\bar q}  \Big] \Big|\nonumber\\
    &\leq  \sqrt{\gamma}\frac{2 }{\phi_0} M_0 + \gamma\frac{8}{\phi_0^2} M_0.
\end{align}
Moreover,
\begin{align}
\label{eq: ssq_lemslow_cd2}
    \EE[(c-d)e^{\sqrt{\gamma} \phi\bar q}] = -\EE[(\gamma q -\nu_\gamma) e^{\sqrt{\gamma} \phi\bar q}] = -\sqrt{\gamma} \EE[\sqrt{\gamma} \bar q e^{\sqrt{\gamma} \phi\bar q}] = -\sqrt{\gamma} \frac{d}{d\phi} \MM{\bar q}.
\end{align}
Now, by combining Eq. \eqref{eq: ssq_lemslow_cd}, \eqref{eq: ssq_lemslow_cd1} and \eqref{eq: ssq_lemslow_cd2}, for $\phi \in \left(-\frac{\phi_0}{2},\frac{\phi_0}{2}\right)$ and $\gamma < \gamma_s= \min\{\gamma_0, \phi_0^2/256\}$, we get
\begin{align*}
     \Big| (\MM{\bar q-d} \MM{c} - \MM{\bar q}) -\frac{\phi^2 \gamma\bar \sigma^2_{\gamma} }{2}\MM{\bar q} & +\phi \gamma \frac{d}{d\phi} \MM{\bar q} \Big| \\
     & = \Big| \EE\Big  [  e^{\sqrt{\gamma} \phi\bar q}\Big( e^{\sqrt{\gamma} \phi(c - d)} - \frac{\phi^2 \gamma\bar \sigma^2_{\gamma}}{2}+ \phi\gamma  (\sqrt{\gamma}\bar q) -1 \Big) \Big] \Big|\allowdisplaybreaks\\
     &\leq 6M_0e^{10A} |\phi|^3 \gamma^{\frac{3}{2}} + \frac{\phi^2 \gamma}{2} \left(\sqrt{\gamma}\frac{2 }{\phi_0} M_0 + \gamma\frac{8}{\phi_0^2} M_0\right) \allowdisplaybreaks\\
     &\leq \left( 6M_0e^{10A} + \frac{2 }{\phi_0} M_0 + \frac{4}{\phi_0^2} M_0 \right) |\phi|^2 \gamma^{\frac{3}{2}}.
\end{align*}
This completes the proof of Lemma \ref{lem:ssq_slow_mgfapprox}.
\end{proof}

\begin{proof}[Proof of Theorem \ref{thm: ssq_slow}: ]
Similar to what we have in Eq. \eqref{eq: ssq_starteq} in the proof of Theorem \ref{thm: ssq_crit}, we have that, for any $\phi \in (-\phi_0,\phi_0)$ and $\gamma < \gamma_0$, Lemma \ref{lem: ssq_mgfexist}, $\MM{\bar q}$ is bounded by $M_0<\infty$. Thus, 
\begin{align}
 \MM{-u}-1  = \MM{q+c-d} - \MM{q}.
 \end{align}
Multiplying by $e^{-\frac{\phi \nu_\gamma}{\sqrt{\gamma}}}$ on both sides, we get
\begin{equation*}
    \MM{\bar q-d} \MM{c} - \MM{\bar q} = \big(\MM{-u}-1 \big)e^{-\frac{\phi \nu_\gamma}{\sqrt{\gamma}}}.
\end{equation*}
Next, we use that
\begin{align*}
    \MM{\bar q-d} \MM{c} - \MM{\bar q} &= \frac{\phi^2 \gamma\bar \sigma^2_\gamma}{2}\MM{\bar q}  - \phi \gamma \frac{d}{d\phi} \MM{\bar q} + \mathcal{E}_s^d(\gamma,\phi) \\
    \big(\MM{-u}-1 \big)e^{-\frac{\phi \nu_\gamma}{\sqrt{\gamma}}} & = \mathcal{E}_s^u(\gamma,\phi)
\end{align*}
This yields
\begin{align*}
    \frac{1}{2} \phi \bar \sigma^2_\gamma\MM{\bar q}  - \frac{d}{d\phi} \MM{\bar q} = \frac{1}{\phi \gamma} \big( \mathcal{E}_s^u(\gamma,\phi) -\mathcal{E}_s^d(\gamma,\phi) \big)= \frac{1}{\phi \gamma} \mathcal{E}_s(\gamma,\phi).
\end{align*}
 Lemma \ref{lem:ssq_slow_mgfapprox} implies that for $\phi \in \left(-\frac{\phi_0}{2},\frac{\phi_0}{2}\right)$ and $\gamma<\gamma_s$, we have
\begin{align*}
    \frac{1}{|\phi| \gamma} |\mathcal{E}_s(\gamma,\phi)| \leq K_s^d|\phi| \sqrt{\gamma} + K_s^u \frac{1}{\sqrt{\gamma}} e^{-\frac{\phi_0 \nu_\gamma} {\sqrt{\gamma}}}e^{-\frac{\phi \nu_\gamma}{\sqrt{\gamma}}}.
\end{align*}
By redefining $G^{\gamma}(\phi) = \exp\big(-\frac{\phi^2 \bar \sigma^2_\gamma}{4}\big)$, we get
\begin{align*}
   \frac{d}{d\phi} \MM{\bar q}  G^{\gamma}(\phi) = \frac{1}{\phi \gamma} \mathcal{E}_s(\gamma,\phi) G^{\gamma}(\phi).
\end{align*}
Since the above equation is valid for any $\phi \in \big( -\frac{\phi_0}{2}, \frac{\phi_0}{2} \big)$, integrating it over $\phi$ and using the condition $M_{\bar q}^\gamma(0) G^\gamma(0) =1$, we have that, for any $\phi \in \big( -\frac{\phi_0}{2}, \frac{\phi_0}{2} \big)$,
\begin{align*}
     \MM{\bar q}  G^{\gamma}(\phi) -1 = \int_{0}^\phi \frac{1}{s \gamma} \mathcal{E}_s(\gamma,s) G^{\gamma}(s)ds.
\end{align*}
Now, since $G^{\gamma}(\phi)$ matches with the (scaled) distribution of a Gaussian random variable with finite variance, the quantities $\int_{0}^{\phi} |s|  G^{\gamma}(s) ds$ and $\int_{0}^{\phi}e^{-\frac{s \nu_\gamma}{\sqrt{\gamma}}} G^{\gamma}(s)ds$ are bounded. By taking $\gamma \rightarrow 0$, for any $\phi \in \big( -\frac{\phi_0}{2}, \frac{\phi_0}{2} \big)$, we have
\begin{align*}
 \limg \int_{0}^\phi \frac{1}{|s| \gamma} |\mathcal{E}_s(s,\phi)| G^{\gamma}(s)ds &\leq  \limg  K_s^d \gamma^\frac{1}{2} \int_{0}^{\phi} |s|  G^{\gamma}(s) ds + \limg K_s^u  \frac{1}{\sqrt{\gamma}}  e^{-\frac{\phi_0 \nu_\gamma} {\sqrt{\gamma}}} \int_{0}^{\phi}e^{-\frac{s \nu_\gamma}{\sqrt{\gamma}}} G^{\gamma}(s)ds \\
 &= 0.
\end{align*}
Finally, we get
\begin{equation*}
    \limg  \MM{\bar q} = \limg \big(G^{\gamma}(\phi)\big)^{-1} = \exp\left( \frac{\phi^2\bar \sigma^2}{4}\right),
\end{equation*}
where $\bar \sigma^2 = \limg \sigma^2_\gamma + \nu_\gamma$. Note that the RHS is the MGF of a zero mean Gaussian random variable. The result follows from combining this with Lemma \ref{COR: MGF_CONVERGENCE}.
\end{proof}

 \section{Proof of lemmas for Theorem \ref{THM: JSQ_LIMIT_DIS}}
\label{app: jsq}

\def\jsqmgfexist{\ref{lem: jsq_mgfexist}}
\subsection{Proof of Lemma \jsqmgfexist}
\label{app: jsq_mgfexist}

$ $\\

\textbf{Stability and existence of the MGF: } 
We repeat the argument presented in the proof of Theorem~\ref{LEM: JSQ_SSC}. Consider a given queue in JSQ-A, i.e., fix $i$ and consider the process $\{q_i(t)\}_{t=0}^\infty$. The evolution of the queue length process is given by 
\begin{align*}
    q_i(t+1) &= q_i(t) + c_i(t) -d_i(t) + u_i(t)\\
    &=  q_i(t) + a(t)Y_i(t) -s_i(t) -d_i(t) + u_i(t)
\end{align*}
We now create a coupled process $\tilde q_i(t)$ with the initial condition $\tilde q_i(0) = q_i(0)$, where we replace $Y_i(t)$ with $1$. Then, 
\begin{equation*}
    \tilde q_i(t+1) = \tilde q_i(t)+  a(t) -s_i(t) - \tilde d_i(t) + \tilde u_i(t),
\end{equation*}
Thus, $\{\tilde q_i(t)\}_{t=0}^\infty $ behaves like a SSQ with abandonment. It can be easily checked that $\tilde q_i(t) \geq q_i(t)$ for all $t\geq 0$, by using a  similar argument as in Eq. \eqref{eq: ssq_mgfexist_coupling}. Then, the result for SSQ with abandonment implies that $\{q_i\}$ is positive recurrent, Moreover, by Lemma \ref{lem: ssq_mgfexist}, in steady state, there exists $\phi_0$ and $M_0$ such that, for any $\phi \in (-\phi_0,\phi_0)$, we have
\begin{equation*}
    \MM{q_i} \leq \MM{\tilde q_i} \leq M_0 \exp \left( \frac{|\phi| [\lambda_\gamma - \mu_{\gamma,i}]^+}{\sqrt{\gamma}}\right).
\end{equation*}
It follows that the moments of the steady state queue lengths $q_i$ are finite for any $i$. Then, for $\phi \in (-\phi_0,\phi_0)$, we have
\begin{align*}
    \EE\left[e^{\frac{1}{n} \sqrt{\gamma} \phi \langle \q, {\bf 1}\rangle } \right] \leq \prod_{i=1}^n  \EE\left[e^{\sqrt{\gamma} \phi q_i} \right]^\frac{1}{n} \leq M_0 \exp \left( \frac{|\phi| [\lambda_\gamma - \mu_{\gamma,\min}]^+}{\sqrt{\gamma}}\right).
\end{align*}

Since the system is stable, we can consider the queue length process in steady-state. Suppose that $\q$ is distributed as the steady state queue length, ${\bf d} \sim Bin(\q,\gamma)$, and ${\bf c} = a{\bf Y} -{\bf s}$, where $a$ and ${\bf s}$ are random variables with same distribution as $a(t)$ and ${\bf s}(t)$ and $Y_i =1$ if arrivals join the $i^{th}$ queue according to the JSQ policy, otherwise $Y_j =0$. Then,
\[\q^+ =[\q +{\bf c} - {\bf d} ]^+ =  \q +{\bf c} - {\bf d} + {\bf u}.\]

Next, we use a similar argument as in the proof of Lemma \ref{lem: ssq_mgfexist} to claim that whenever $\nu_\gamma<0$, we can add extra arrivals to get a queue length process which dominates the original one. So, without loss of generality, we assume $\nu_\gamma>0$.\\

\textbf{One step coupling on abandonments: } Suppose ${\bf \tilde d} \sim Bin(\min\{ {\bf \tilde C}/\gamma, {\bf q}\},\gamma)$, where we couple ${\bf \tilde d}$ with ${\bf d}$ in a similar manner as in the proof of Lemma \ref{lem: ssq_mgfexist} so that ${\bf \tilde d} = Bin(\min\{ {\bf \tilde C}/\gamma, {\bf q}\},\gamma) \leq Bin(\q,\gamma) = {\bf d}$. Also, we use ${\bf \tilde C} = \tilde C {\bf 1}_n$, where $\tilde C = \frac{1}{n}\nu_\gamma + A\sqrt{\gamma}$ (with $\nu_\gamma>0$). Then, for any $i \in \{1,2,\dots,n\}$, we have
\begin{align}
\label{eq: jsq_mgfexist_expo_rel1}
    \EE\left[e^{\sqrt{\gamma}\phi (q^+_i - u_i)}\right] = \EE\left[ e^{\sqrt{\gamma}\phi (q_i +c_i -d_i)}\right]\leq \EE\left[ e^{\sqrt{\gamma}\phi (q_i +c_i -\tilde d_i)}\right].
\end{align}

\textbf{Drift of exponential Lyapunov function: }For $0<\phi < \phi_0$, we have
\begin{align*}
    \EE\left[\left. e^{\sqrt{\gamma}\phi (c_i -\tilde d_i)} \,\right|\, \q \right] 
    & \stackrel{(a)}{\leq}1 + \sqrt{\gamma}\phi \EE[c_i -\tilde d_i | \q] + \gamma \phi^2 \EE\big[(c_i -\tilde d_i)^2 e^{\sqrt{\gamma}\phi [c_i -\tilde d_i]^+} \big| \q\big]\\
    &\stackrel{(b)}{\leq} 1 + \sqrt{\gamma}\phi \big(\lambda_\gamma Y_i -\mu_{\gamma,i} - \min\{\tilde C,\gamma q_i\}\big) + 2\gamma \phi^2 e^{\sqrt{\gamma}\phi A}  \EE\big[A^2 +\tilde d_i^2 \big| \q\big]\\
    &\stackrel{(c)}{\leq}1 + \sqrt{\gamma}\phi \big(\lambda_\gamma Y_i -\mu_{\gamma,i} - \tilde C\big) +\sqrt{\gamma}\phi [ \tilde C - \gamma q_i]^+ + \gamma \phi^2 K_1,
\end{align*}
where (a) follows by using $e^{x} \leq 1+ x+ x^2 e^{[x]^+}$; (b) follows by using $(c_i -\tilde d_i)^2 \leq 2c_i^2 +2\tilde d_i^2 \leq 2A^2+ 2\tilde d_i^2 $ and $[c_i -\tilde d_i]^+ \leq A$; and (c) follows by using $\EE[\tilde d_i^2 | \q]  \leq \tilde C + \tilde C^2 \leq 6A^2 $ as $\tilde C \leq 2A$ and then picking $K_1 = 7A^2e^A $.
We then have that, for $0<\phi < \phi_0$,
\begin{align}
\label{eq: jsq_mgfexist_expo_rel2}
   \sum_{i=1}^n &\EE\left[ e^{\sqrt{\gamma}\phi (q_i +c_i -\tilde d_i)} \right] \nonumber\\
    &\leq\sum_{i=1}^n (1+\gamma \phi^2 K_1)\MM{q_i} + \sqrt{\gamma}\phi \EE\big[\big(\lambda_\gamma Y_i -\mu_{\gamma,i} - \tilde C\big) e^{\sqrt{\gamma}\phi q_i}\big] + \sqrt{\gamma}\phi \EE\Big[[ \tilde C - \gamma q_i]^+e^{\sqrt{\gamma}\phi q_i}\Big]\nonumber\\
    & \stackrel{(a)}{\leq} \big( 1 - \sqrt{\gamma}\phi \tilde C +\gamma \phi^2 K_1 \big) \sum_{i=1}^n \MM{q_i} + \sqrt{\gamma}\phi \left( \lambda_\gamma \MM{q_{\min}} - \sum_{i=1}^n \mu_{\gamma,i} \MM{q_i} \right) \nonumber\\
    & \ \ \ \ + \sqrt{\gamma}\phi e^{\phi \tilde C /\sqrt{\gamma} } \sum_{i=1}^n \EE\big[ \tilde C - \gamma q_i\big]^+ \allowdisplaybreaks\nonumber\\
    &\stackrel{(b)}{\leq} \big( 1 - \sqrt{\gamma}\phi \tilde C +\gamma \phi^2 K_1 \big) \sum_{i=1}^n \MM{q_i} + \sqrt{\gamma}\phi \left( \lambda_\gamma  - \sum_{i=1}^n \mu_{\gamma,i} \right) \MM{q_{\min}} \nonumber\\
    & \ \ \ \ + \sqrt{\gamma}\phi e^{\phi \tilde C /\sqrt{\gamma}  } \sum_{i=1}^n \EE[ \tilde C - \gamma q_i]^+\nonumber \allowdisplaybreaks \\
    &\stackrel{(c)}{\leq}(1 - \sqrt{\gamma}\phi \tilde C +\gamma \phi^2 K_1)\sum_{i=1}^n \MM{q_i} + \sqrt{\gamma}\phi \frac{\nu_\gamma}{n}\sum_{i=1}^n \MM{q_i} + \sqrt{\gamma}\phi e^{\phi \tilde C /\sqrt{\gamma}  }\sum_{i=1}^n  \EE[ \tilde C - \gamma q_i]^+\nonumber \allowdisplaybreaks\\
    & \stackrel{(d)}{\leq} (1 - A \gamma \phi  +\gamma \phi^2 K_1)\sum_{i=1}^n \MM{q_i} + \sqrt{\gamma}\phi e^{\phi \tilde C /\sqrt{\gamma}  } \sum_{i=1}^n  \EE[ \tilde C - \gamma q_i]^+\nonumber \allowdisplaybreaks\\
    & \stackrel{(e)}{\leq} \left(1-\frac{1}{2}A\gamma\phi\right) \sum_{i=1}^n \MM{q_i} + \sqrt{\gamma}\phi e^{\phi \tilde C /\sqrt{\gamma}  } \sum_{i=1}^n  \EE[ \tilde C - \gamma q_i]^+,
\end{align} 
where (a) follows by using $\sum_{i=1}^n \EE[Y_i e^{\sqrt{\gamma}\phi q_i}] = \MM{q_{\min}}$; (b) follows by using $\MM{q_{\min}} \leq \MM{q_{i}}$ for all $i$ when $\phi>0$; (c) follows by using $\lambda_\gamma - \sum_{i=1}^n \mu_{\gamma,i} = \nu_\gamma $ and $\MM{q_{\min}} \leq \frac{1}{n}\sum_{i=1}^n \MM{q_i}$ for $\phi>0$; (d) follows by using $\tilde C = \frac{1}{n} \nu_\gamma + A\sqrt{\gamma} $; and (e) follows by choosing $\phi<  \frac{A}{2K_1} = \frac{1}{7Ae^{A}}$. We also have
\begin{align}
\label{eq: jsq_mgfexist_expo_rel3}
   \sum_{i=1}^n  \EE\left[e^{\sqrt{\gamma}\phi (q^+_i - u_i)}\right] &\geq \sum_{i=1}^n \MM{q^+_i} -\sqrt{\gamma}\phi \sum_{i=1}^n\EE\left[u_ie^{\sqrt{\gamma}\phi q^+_i }\right] \nonumber \allowdisplaybreaks\\
   &= \sum_{i=1}^n \MM{q^+_i} -\sqrt{\gamma}\phi \sum_{i=1}^n\EE[u_i] \nonumber \allowdisplaybreaks\\
   &= \sum_{i=1}^n\MM{q_i} -\sqrt{\gamma}\phi \sum_{i=1}^n\EE[u_i],
\end{align}
where the last equality follows by equating the drift of the exponential Lyapunov function to zero in steady state. Combining Eq. \eqref{eq: jsq_mgfexist_expo_rel1}, \eqref{eq: jsq_mgfexist_expo_rel2} and \eqref{eq: jsq_mgfexist_expo_rel3}, we obtain
\begin{align*}
    \frac{1}{2}A\gamma\phi \sum_{i=1}^n\MM{q_i} &\leq \sqrt{\gamma}\phi \sum_{i=1}^n\EE[u_i] + \sqrt{\gamma}\phi e^{\phi \tilde C /\sqrt{\gamma}  } \sum_{i=1}^n \EE[ \tilde C - \gamma q_i]^+,
\end{align*}
and thus
\begin{align}
\label{eq: jsq_mgfexist_sumbound_intermediateterm}
    \sum_{i=1}^n\MM{q_i} &\leq \frac{2}{\sqrt{\gamma} A} e^{\phi A} \sum_{i=1}^n\Big( \EE[u_i] +  \EE[ \tilde C - \gamma q_i]^+\Big) e^{\frac{\phi \nu_\gamma}{n\sqrt{\gamma}}  }
\end{align}
 
 \textbf{Bound on RHS using second order drift: } Equating the drift of $\langle {\bf 1} , \q\rangle$ to zero in steady state, we get
\begin{align*}
   0 & = \EE[\langle {\bf 1} , {\bf q}^+\rangle] -\EE[\langle {\bf 1} , {\bf q}\rangle]\\
   & = \EE[\langle {\bf 1} , {\bf c} - {\bf d} + {\bf u}\rangle]\\
   &= \lambda_\gamma -\langle {\bf 1} , \boldsymbol \mu_\gamma\rangle - \gamma \EE[  \langle {\bf 1} , {\bf q}\rangle] + \EE[ \langle {\bf 1} , {\bf u}\rangle ]\\
   & = \nu_\gamma - \gamma \EE[\langle {\bf 1} , \q\rangle] + \EE[\langle {\bf 1} , {\bf u}\rangle]\\
   & =   \EE[\langle {\bf 1} , {\bf u}\rangle]  + \EE \left[\sum_{i=1}^n \left( \left[\frac{\nu_\gamma}{n} - \gamma q_i\right]^+ - \left[\gamma q_i- \frac{\nu_\gamma}{n} \right]^+\right) \right].
\end{align*}
This implies that $\EE[ \langle {\bf 1},\bar \q\rangle] = \EE[ \langle {\bf 1}, \q\rangle] - \frac{\nu_\gamma}{\gamma} = \frac{1}{\gamma} \EE[\langle {\bf 1},{\bf u}\rangle] \geq 0$, $\gamma \EE[ \langle {\bf 1},\bar \q\rangle] = \EE[\langle {\bf 1},{\bf u}\rangle] \leq nA$, and
\begin{align*}
      \sum_{i=1}^n \left( \EE[u_i] + \EE\left[\frac{\nu_\gamma}{n} - \gamma q_i\right]^+ \right) 
     &=  \sum_{i=1}^n \EE\left[\gamma q_i- \frac{\nu_\gamma}{n} \right]^+= \gamma \sum_{i=1}^n \EE[ \bar q_i ]^+ =  \gamma \EE[\|\bar \q\|],
\end{align*}
where $\bar q_i = q_i - \frac{\nu_\gamma}{n\gamma}$. 
From Eq. \eqref{eq: jsq_lemssc_l2drift}, we know that 
\begin{align*}
    	\Delta \| \q \|^2 &\leq nA^2+\|{\bf d}\|^2 + 2\langle \q - {\bf d},{\bf c} \rangle - 2 \langle \q, {\bf d}\rangle.
\end{align*}
By taking expectation on both side and equating the drift of to zero in steady state, (i.e., $\EE[\Delta \| \q \|^2 ] =0$), we get
\begin{align}
\label{eq: jsq_mgfexist_l2driftzero}
    0&\leq  \EE \Big[nA^2 + \gamma(1-\gamma)\langle {\bf 1}, \q\rangle+ (\gamma^2-2\gamma)\|\q\|^2 + 2(1-\gamma)\big( \lambda_\gamma q_{\min} -  \langle \boldsymbol \mu , \q \rangle \big)\Big] \nonumber\\
    &\leq  nA^2 + \gamma\EE[ \langle {\bf 1}, \q\rangle] + (\gamma^2-2\gamma) \EE[\|\q\|^2] + 2(1-\gamma) \nu_\gamma \EE[q_{\min}] \nonumber\\
    & \leq nA^2 +\gamma\EE[ \langle {\bf 1}, \q\rangle]+ (\gamma^2-2\gamma)\EE[\|\q\|^2] + \frac{2(1-\gamma)\nu_\gamma}{n} \EE[\langle {\bf 1}, \q\rangle].
\end{align}
Further,
\begin{align*}
    \gamma\langle {\bf 1}, \q\rangle= \gamma\langle {\bf 1}, \bar \q\rangle + \nu_\gamma, &&  \|\q\|^2 = \| \bar \q\|^2 + \frac{2\nu_\gamma }{n\gamma} \langle {\bf 1}, \bar \q\rangle + \frac{\nu_\gamma^2}{n\gamma^2}.
\end{align*}
By substituting this in Eq. \eqref{eq: jsq_mgfexist_l2driftzero}, we get
\begin{align*}
    0& \leq  nA^2 + \left(\gamma + (1-\gamma)\frac{2\nu_\gamma}{n} \right) \left( \EE[ \langle {\bf 1}, \bar \q\rangle] + \frac{\nu_\gamma}{\gamma} \right) + \big(\gamma^2-2\gamma\big) \left(\EE\big[\| \bar \q\|^2\big] + \frac{2\nu_\gamma }{n\gamma} \EE[ \langle {\bf 1}, \bar \q\rangle] + \frac{\nu_\gamma^2}{n\gamma^2} \right)\\
    & = nA^2 + \nu_\gamma - \frac{\nu_\gamma^2}{n} + \left( \gamma-\frac{2\nu_\gamma}{n} \right) \EE[ \langle {\bf 1},\bar \q\rangle] + \big(\gamma^2 - 2\gamma\big) \EE\big[ \| \bar \q\|^2 \big] \\
    &\stackrel{(a)}{\leq } 2nA^2 + \gamma\EE[ \langle {\bf 1},\bar \q\rangle] -\gamma\EE\big[ \| \bar \q\|^2 \big] \\
    &\stackrel{(b)}{\leq } 3nA^2-\gamma\EE\big[ \| \bar \q\|^2 \big],
\end{align*}
where (a) follows by using $0\leq \nu_\gamma \leq A$; and (b) follows by using $\gamma \EE[ \langle {\bf 1},\bar \q\rangle] = \EE[\langle {\bf 1},{\bf u}\rangle] \leq nA$ as shown earlier.
Thus, we get $\gamma\EE[ \| \bar \q\|^2] \leq  3nA^2$ and so
\begin{align}
\label{eq: jsq_mgfexist_center1storder}
       \gamma \EE[ \| \bar \q\|] \leq A \sqrt{3n\gamma}
\end{align}
Thus, we conclude that 
\begin{align*}
      \sum_{i=1}^n \left( \EE[u_i] + \EE\left[\tilde C - \gamma q_i\right]^+ \right) &\le \sum_{i=1}^n \left( \EE[u_i] + \EE\left[\frac{\nu_\gamma}{n} - \gamma q_i\right]^+ +A\sqrt{\gamma} \right) \\
      &\leq  \gamma \EE[\|\bar \q\|] + A n\sqrt{\gamma}\\
      &\leq A (n+\sqrt{3n}) \sqrt{\gamma} \leq 4nA\sqrt{\gamma}.
\end{align*}

\textbf{Bound for Lemma \ref{lem: jsq_mgfexist_a}: } 
Combing the above inequality with Eq. \eqref{eq: jsq_mgfexist_sumbound_intermediateterm}, we get
\begin{align}
\label{eq: jsq_mgfexist_sumexpobound}
  \sum_{i=1}^n\MM{q_i} \leq \frac{2e^{A}}{\sqrt{\gamma}A} \times 4nA\sqrt{\gamma} \times e^{\frac{\phi \nu_\gamma}{n\sqrt{\gamma}}} = 8n e^A e^{\frac{\phi \nu_\gamma}{n\sqrt{\gamma}}}.
\end{align}
Further, by Jensen's inequality, we have
\begin{align*}
  \sum_{i=1}^n\MM{q_i} =    \sum_{i=1}^n \EE\big[e^{\sqrt{\gamma}\phi q_i}\big] \geq n \EE\left[\exp\left( \frac{\sqrt{\gamma} \phi}{n} \sum_{i=1}^n q_i \right) \right].
\end{align*}
Using $\phi$ instead of $\phi/n$, for $0<\phi<\min\big \{\frac{\phi_0}{n}, \frac{1}{7nAe^{A}}\big \}$, we have
\begin{align*}
    \MM{\langle {\bf 1}, \q \rangle} =  \EE\left[\exp\left( \sqrt{\gamma} \phi \sum_{i=1}^n q_i \right)\right] \leq 8e^{A} e^{\frac{\phi \nu_\gamma}{\sqrt{\gamma}}}.
\end{align*}
This completes the proof of Lemma \ref{lem: jsq_mgfexist_a} by choosing $M_1 > 8e^A$.\\

\textbf{Bound for Lemma \ref{lem: jsq_mgfexist_b}: } In this part, we abuse notation and redefine \[\tilde{{\bf d}} \sim Bin(\max\{\tilde{{\bf C}}/\gamma,{\bf q}\},\gamma) \geq Bin(\q,\gamma) \sim {\bf d},\] where we also redefine $\tilde{{\bf C}}$ to be $\tilde{{\bf C}} = \tilde{C}{\bf 1}_n$, with $\tilde{C} = \frac{1}{n} \nu_\gamma - A\sqrt{\gamma}$. Here we assume $\frac{1}{n} \nu_\gamma > A\sqrt{\gamma}$, as  we can increase the arrival rate to make sure it holds and create a coupled process that dominates the original one.
For any $\phi>0$, $M_{q_i}^\gamma(-\phi) \leq 1$, thus in steady state $M_{q_i^+}^\gamma(-\phi)=M_{q_i}^\gamma(-\phi)$. Then,
\begin{align}
\label{eq: jsq_mgfbound_mgfzerodrift}
    0&= M_{q_i^+}^\gamma(-\phi)-M_{q_i}^\gamma(-\phi)\nonumber\\
    &= \EE[e^{-\sqrt{\gamma}\phi (q_i +c_i - d_i  +u_i)}] - M_{q_i}^\gamma(-\phi) \nonumber\\
    &\stackrel{(a)}{\leq} \EE[e^{-\sqrt{\gamma}\phi (q_i+c_i - \tilde d_i )}] -M_{q_i}^\gamma(-\phi) \nonumber\\
    &\stackrel{(b)}{\leq} \phi\sqrt{\gamma}\EE\big[( \max\{\tilde{ C},\gamma q_i\} - \lambda_\gamma Y_i +\mu_{\gamma,i} )e^{-\sqrt{\gamma }\phi q_i} \big] + \phi^2\gamma e^A \EE\big[(A+\tilde d_i)^2 e^{\sqrt{\gamma }\phi (\tilde d_i - q_i)}\big],
\end{align}
where (a) follows by using $\tilde{{\bf d}} \ge {\bf d}$ element wise and ${\bf u}\geq 0$; and (b) follows by same second order approximation as in Eq. \eqref{eq: ssq_mgfexist_starteq2}. Moreover, by same argument as in Eq. \eqref{eq: ssq_mgfexists2_term2}, we obtain
\begin{align}
\label{eq: jsq_mgfbound_mgfsecondorder}
    \EE\left[ (A+\tilde d_i)^2 e^{\sqrt{\gamma }\phi (\tilde d_i - q_i)} \right] \leq 2e^{9A}M_{q_i}^\gamma(-\phi) + 4e^{3A} e^{-\frac{\phi \nu_\gamma}{n\sqrt{\gamma}}} 
\end{align}
Further, using $\sum_{i=1}^n \EE\big[ Y_i e^{-\sqrt{\gamma }\phi q_i} \big] = M_{q_{\min}}^\gamma(-\phi) $, we get
\begin{align*}
    \sum_{i=1}^n \EE\big[( & \max\{\tilde{ C},\gamma q_i\} - \lambda_\gamma Y_i +\mu_i )e^{-\sqrt{\gamma }\phi q_i} \big] \\
    &= \tilde C \sum_{i=1}^n M_{q_i}^\gamma(-\phi) +\sum_{i=1}^n \EE\big[[\gamma q_i - \tilde C]^+e^{-\sqrt{\gamma }\phi q_i} \big] -\lambda_\gamma M_{q_{\min}}^\gamma(-\phi) + \sum_{i=1}^n \mu_{\gamma,i} M_{q_i}^\gamma(-\phi)\allowdisplaybreaks\\
    & \stackrel{(a)}{\leq} \Big(\frac{1}{n} \nu_\gamma -A\sqrt{\gamma}\Big) \sum_{i=1}^n M_{q_i}^\gamma(-\phi) + e^{-\sqrt{\gamma }\phi \tilde C/\gamma} \sum_{i=1}^n \EE[\gamma q_i - \tilde C]^+  - \nu_\gamma M_{q_{\min}}^\gamma(-\phi)\allowdisplaybreaks\\
    &\stackrel{(b)}{\leq} -A\sqrt{\gamma}\sum_{i=1}^n M_{q_i}^\gamma(-\phi) + e^A e^{-\frac{\phi \nu_\gamma}{n\sqrt{\gamma}}} \times 3nA\sqrt{\gamma},
\end{align*}
where (a) follows by using $M_{q_i}^\gamma(-\phi) \leq M_{q_{\min}}^\gamma(-\phi)$ and $[\gamma q_i - \tilde C]^+e^{-\sqrt{\gamma }\phi q_i} \leq [\gamma q_i - \tilde C]^+e^{-\sqrt{\gamma }\phi \tilde C/\gamma}$; and (b) follows by using $\sum_{i=1}^n \EE[\gamma q_i - \tilde C]^+ \leq nA\sqrt{\gamma} + \gamma\EE[\|\bar{{\bf q}}\|] \le 3nA\sqrt{\gamma}$ (by Eq. \eqref{eq: jsq_mgfexist_center1storder}) and $\frac{1}{n}\sum_{i=1}^n M_{q_i}^\gamma(-\phi) \leq M_{q_{\min}}^\gamma(-\phi)$. Combining this with Eq. \eqref{eq: jsq_mgfbound_mgfzerodrift} and \eqref{eq: jsq_mgfbound_mgfsecondorder},
we get
\begin{align*}
    0&\leq \Big(-A + 2\phi e^{10A}\Big)\sum_{i=1}^n M_{q_i}^\gamma(-\phi) +\big( 3nAe^A+4\phi e^{3A}\big) e^{-\frac{\phi \nu_\gamma}{n\sqrt{\gamma}}} \\
    &\leq -\frac{A}{2} \sum_{i=1}^n M_{q_i}^\gamma(-\phi) + 4nAe^A e^{-\frac{\phi \nu_\gamma}{n\sqrt{\gamma}}},
\end{align*}
where the last inequality follows by choosing $\phi < \frac{A}{4e^{10A}}$. This yields
\begin{align}
\label{eq: jsq_mgfexist_sumexpobound2}
    \sum_{i=1}^n M_{q_i}^\gamma(-\phi) \leq 8ne^A e^{-\frac{\phi \nu_\gamma}{n\sqrt{\gamma}}}.
\end{align}
Now, by Jensen's inequality and  replacing $\phi/n$ with $n$, we get that for any $\phi < \frac{A}{4ne^{10A}}$,
\begin{align}
     M_{\vsum{q}}^\gamma(-\phi) \leq 8e^A e^{-\frac{\phi \nu_\gamma}{\sqrt{\gamma}}}.
\end{align}

\textbf{Upper bound in Lemma \ref{lem: jsq_mgfexist_c}:} Following similar arguments as in \eqref{eq: jsq_mgfexist_expo_rel2}, for $0<\phi<\phi_0$, we obtain
\begin{align*}
    \sum_{i=1}^n  \EE\left[ e^{\gamma^\alpha \phi(q_i + c_i -d_i)} \right] &\leq \sum_{i=1}^n \EE\left[ e^{\gamma^\alpha \phi(q_i + c_i)} \right] \\
    &\leq \sum_{i=1}^n \EE\Big[ e^{\gamma^\alpha \phi q_i } \Big( 1 +\phi \gamma^\alpha(\lambda_\gamma Y_i - \mu_{\gamma_i}) + \phi^2 \gamma^{2\alpha} A^2 \Big) \Big]\\
    & \stackrel{(a)}{\leq} \big(1+ \phi^2 \gamma^{2\alpha} A^2\big) \sum_{i=1}^n \EE\left[ e^{\gamma^\alpha \phi q_i } \right] + \phi \gamma^\alpha \left( \lambda_\gamma \EE\big[ e^{\gamma^\alpha \phi q_{\min} }\big] - \sum_{i=1}^n \mu_{\gamma_i} \EE\left[ e^{\gamma^\alpha \phi q_{i} }\right] \right)\\
    &\stackrel{(b)}{\leq} \big(1+ \phi^2 \gamma^{2\alpha} A^2\big) \sum_{i=1}^n \EE\left[ e^{\gamma^\alpha \phi q_i } \right] - C_f\phi \gamma^{2\alpha} \frac{\mu_{\gamma,\min}}{\langle {\bf 1},\boldsymbol \mu_\gamma \rangle}\sum_{i=1}^n  \EE\left[ e^{\gamma^\alpha \phi q_{i} }\right]\\
    & \stackrel{(c)}{\leq} \left( 1 - \phi \gamma^{2\alpha} \frac{C_f \mu_{\gamma,\min}}{2\langle {\bf 1},\boldsymbol \mu_\gamma \rangle}\right)\sum_{i=1}^n \EE\left[ e^{\gamma^\alpha \phi q_i } \right],
\end{align*}
where (a) follows by using $|c_i |\leq A$; (b) follows from the fact that, for any positive vector ${\bf x}$, we have
\begin{align*}
    \lambda_\gamma x_{\min} - \sum_{i=1}^n \mu_{\gamma_i} x_i \leq \left( \frac{\lambda_\gamma}{\langle {\bf 1},\boldsymbol \mu_\gamma \rangle} -1 \right) \sum_{i=1}^n \mu_{\gamma_i} x_i = \frac{\nu_\gamma}{\langle {\bf 1},\boldsymbol \mu_\gamma \rangle} \sum_{i=1}^n \mu_{\gamma_i} x_i \leq  -C_f \gamma^\alpha \frac{\mu_{\gamma,\min}}{\langle {\bf 1},\boldsymbol \mu_\gamma \rangle} \langle {\bf 1},{\bf x} \rangle,
\end{align*}
and (c) follows by taking $0<\phi < \frac{C_f \mu_{\gamma,\min}}{2\langle {\bf 1},\boldsymbol \mu_\gamma \rangle}$. Further, following similar arguments as in Eq. \eqref{eq: jsq_mgfexist_expo_rel3}, we get
\begin{align*}
   \sum_{i=1}^n  \EE\left[ e^{\gamma^\alpha\phi (q^+_i - u_i)} \right] 
   &= \sum_{i=1}^n\EE\left[e^{\gamma^\alpha\phi q_i}\right] -\gamma^\alpha\phi \sum_{i=1}^n\EE[u_i].
\end{align*}
Thus, using that $\sum_{i=1}^n  \EE\big[e^{\gamma^\alpha\phi (q^+_i - u_i)}\big]  =  \sum_{i=1}^n  \EE\big[ e^{\gamma^\alpha \phi(q_i + c_i -d_i)} \big]$ and Jensen's inequality, we get
\begin{align}
\label{eq: jsq_mgfexist_ht_sumbound}
    n\EE\Big[ \exp \Big\{ \frac{\gamma^\alpha \phi}{n}  \langle {\bf 1},{\bf q} \rangle \Big\} \Big] \leq \sum_{i=1}^n\EE\left[ e^{\gamma^\alpha \phi q_i } \right] \leq \frac{2\langle {\bf 1},\boldsymbol \mu_\gamma \rangle}{C_f \mu_{\gamma,\min}} \times \frac{1}{\gamma^\alpha}\sum_{i=1}^n   \EE[u_i].
\end{align}
From Eq. \eqref{eq: jsq_lemssc_l2drift}, we know that 
\begin{align*}
    	\Delta \| \q \|^2 &\leq nA^2+\|{\bf d}\|^2 + 2\langle \q - {\bf d},{\bf c} \rangle - 2 \langle \q, {\bf d}\rangle \leq nA^2 + 2\langle \q - {\bf d},{\bf c} \rangle - \langle \q, {\bf d}\rangle,
\end{align*}
where last inequality follows since ${\bf d} \leq \q$, element wise. 
By taking expectation on both sides and equating the drift to zero in steady state, (i.e., $\EE[\Delta \| \q \|^2 ] =0$), we obtain
\begin{align*}
    0&\leq  \EE \Big[nA^2  -\gamma\|\q\|^2 + 2(1-\gamma)\big( \lambda_\gamma q_{\min} -  \langle \boldsymbol \mu , \q \rangle \big)\Big]\\
    &\leq  nA^2 -\gamma \EE[\|\q\|^2] + 2(1-\gamma) \nu_\gamma \EE[q_{\min}] \\
    &\leq nA^2 -\gamma\EE[\|\q\|^2],
\end{align*}
where the last inequality follows by using $\nu_\gamma<0$. Thus,
$\gamma\EE[\|\q\|^2]\leq nA^2$ and so $\EE[\langle {\bf 1}, \q\rangle] \leq \sqrt{n}  \EE[\|\q\|] \leq nA$. It follows that
\begin{align*}
    \EE[\langle {\bf 1},{\bf u} \rangle ] = \gamma\EE[\langle {\bf 1}, \q\rangle] -\nu_\gamma \leq \sqrt{\gamma} nA + C_f \gamma^\alpha \leq \gamma^\alpha (nA + C_f). 
\end{align*}
Now, by substituting this in Eq. \eqref{eq: jsq_mgfexist_ht_sumbound} and replacing $\phi/n$ by $\phi$, we get that for any $0<\phi<\min\big\{\frac{\phi_0}{n}, \frac{C_f \mu_{\gamma,\min}}{2n\langle {\bf 1},\boldsymbol \mu_\gamma \rangle} \big\}$,
\begin{align*}
    \EE\Big[ \exp \Big\{ \gamma^\alpha \phi  \langle {\bf 1},{\bf q} \rangle \Big\} \Big]  \leq  \frac{2\langle {\bf 1},\boldsymbol \mu_\gamma \rangle}{n C_f \mu_{\gamma,\min}} \times (nA + C_f).
\end{align*}
Therefore, the results in Lemma \ref{lem: jsq_mgfexist} holds with
\begin{align*}
    \phi_1 & = \min\left\{\frac{\phi_0}{n}, \frac{1}{7nAe^{A}}, \frac{A}{4ne^{10A}} \frac{C_f \mu_{\gamma,\min}}{2n\langle {\bf 1},\boldsymbol \mu_\gamma \rangle}\right\},\\
    M_1 & = \min\left\{ 8e^A, \frac{2\langle {\bf 1},\boldsymbol \mu_\gamma \rangle}{n C_f \mu_{\gamma,\min}} \times (nA + C_f)\right\}.
\end{align*}

\subsection{Proof of lemmas for Theorem \ref{thm: jsq_fast}}
\label{app: jsq_fast}

Recall that $(\q^+, \q, {\bf d},{\bf u})$ denotes a vector of random variables distributed as the steady state of $({\bf q}(t+1),{\bf q}(t),{\bf d}(t),{\bf u}(t))$, with ${\bf q}^+ = {\bf q}+ {\bf c}- {\bf d}+ {\bf u}$, where ${\bf c}$ has the same distribution as ${\bf c}(t)$. Thus, ${\bf q}^+$ and ${\bf q}$ have the same distribution. Also, we have $|c_i|\leq A$ and $0\leq u_i \leq A$. 

\begin{proof}[Proof of Lemma \ref{lem: jsq_approximation_fast}]  
 \ From Lemma \ref{lem: jsq_mgfexist_c}, we have $\gamma^\alpha\EE[\vsum{q}] \leq \phi_1^{-1}M_1$. As $\EE[\vsum{u}] = \gamma \EE[\vsum{q}] - \nu_\gamma$, it implies that 
\begin{align}
\label{eq: jsq_ubound_fast}
    |\EE[\vsum{u}] +\nu_\gamma| \leq \gamma^{1-\alpha}\phi_1^{-1}M_1, &&  \EE[\vsum{u}] \leq (1+C_f + \phi_1^{-1}M_1)\gamma^\alpha.
\end{align}
Moreover, since $\vsum{u}\leq nA$, the same arguments as in Eq. \eqref{eq: ssq_approximation_fast_uterm} yield
\begin{align*}
    \Big| \EE\left[{e^{-\gamma^\alpha \phi  \vsum{u} }}\right]-\phi\gamma^\alpha \nu_\gamma -1\Big| &\leq nA(2\phi_1^{-1}M_1 + C_f+1) \phi^2 \gamma^{\min\{3\alpha,1\}} e^{n|\phi|A}.
\end{align*}
 This proves Lemma \ref{lem: jsq_approximation_fast_a} with $\overline K_f^u = nA(2\phi_1^{-1}M_1 + C_f+1)$.  For Lemma \ref{lem: jsq_approximation_fast_b}, following same arguments as in Eq. \eqref{eq: ssq_approximation_fast_cterm}, we obtain
\begin{align*}
    \left| \EE\left[{e^{\gamma^\alpha \phi  c }}\right] - \frac{\gamma^{2\alpha}\phi^2 \big(\sigma^2_\gamma+\nu_\gamma^2\big)}{2} - \gamma^\alpha \phi \nu_\gamma -  1 \right| 
&\leq n^3A^3|\phi|^3 \gamma^{3\alpha} e^{n|\phi| A}.
\end{align*}
 This proves Lemma \ref{lem: jsq_approximation_fast_b} with $\overline K_f^c = n^3A^3$. For Lemma \ref{lem: jsq_approximation_fast_c}, first note that $\EE[  e^{\gamma^\alpha \phi  \vsum{q} } ]\leq M_1$. Thus, we can use the same argument as in Eq. \eqref{eq: ssq_approximation_fast_dterm} to obtain
\begin{align*}
    \Big| \EE\big[{e^{\gamma^\alpha \phi  \langle {\bf 1},\q-{\bf d} \rangle }}\big] - \EE\big[{e^{\gamma^\alpha \phi  \vsum{q}}}\big] \Big| \leq \frac{2M_1}{\phi_1} |\phi| \gamma.
\end{align*}
 This proves Lemma \ref{lem: jsq_approximation_fast_c} with $\overline K_f^d = \frac{2M_1}{\phi_1}$. 
\end{proof}

\begin{proof}[Proof of Lemma \ref{LEM: JSQ_SSC_fast}: ]
By construction $q^+_{\| i} = \frac{1}{n} \vsum{q^+}$, and so, by using $q_i^+ u_i =0$,
\begin{align*}
    \vsum{u}\vsum{q^+} = \sum_{i=1}^n u_i \vsum{q^+} = n \sum_{i=1}^n u_i (q_i^+ - q^+_{\perp i}) = -\sum_{i=1}^n u_i q_{\perp i}^+ = -\langle {\bf u},{\bf q}_{\perp}^+ \rangle
\end{align*}
Then, by using $e^{x}-1 \leq |x|e^{[x]^+}$, we obtain
\begin{align*}
     \Big|\EE\Big[\big(e^{\gamma^\alpha \phi \vsum{q^+}} -1\big)\big(e^{-\gamma^\alpha \phi \vsum{u} }-1\big)\Big]\Big| &\leq \phi^2 \gamma^{2\alpha }\EE\Big[\vsum{u}\vsum{q^+} e^{\gamma^\alpha [\phi]^+ \vsum{q^+}} e^{\gamma^\alpha |\phi| \vsum{u} }\Big]\\
     &\stackrel{(a)}{\leq} \phi^2 e^{n|\phi|A}\gamma^{2\alpha }\EE\Big[\vsum{u}\vsum{q^+} e^{\gamma^\alpha [\phi]^+ \vsum{q^+}} \Big]\\
     &= -\phi^2 \gamma^{2\alpha }e^{n|\phi|A} \EE\Big[\langle {\bf u},{\bf q}_{\perp}^+\rangle e^{\gamma^\alpha [\phi]^+ \vsum{q^+}} \Big]\\
     &\stackrel{(b)}{\leq}  \phi^2 \gamma^{2\alpha }e^{n|\phi|A} \EE[\|{\bf u}\|^4]^{\frac{1}{4}} \EE[\|{\bf q}_{\perp}^+\|^2]^{\frac{1}{2}}\EE\Big[ e^{4\gamma^\alpha [\phi]^+ \vsum{q^+}} \Big] ^{\frac{1}{4}}\\
     &\stackrel{(c)}{\leq} \phi^2 \gamma^{2\alpha }e^{n|\phi|A} \big(nA^3 (1+C_f + \phi_1^{-1}M_1)\big)^{\frac{1}{4}} \gamma^{\frac{\alpha}{4}} M_\perp^{\frac{1}{2}} M_1^\frac{1}{4}\\
     & = \overline{K}_f^{qu}\phi^2 \gamma^{\frac{9\alpha}{4}}e^{n|\phi|A},
\end{align*}
where (a) follows by using $|\phi|<\phi_1<1$ and $\vsum{u}\leq nA$; (b) follows by using Cauchy-Schwarz inequality twice; (c) follows  from Lemma \ref{lem: jsq_mgfexist_c} for $[\phi]^+<\phi_1/4$, Theorem \ref{LEM: JSQ_SSC} and $\EE[\vsum{u}] \leq (1+C_f + \phi_1^{-1}M_1)\gamma^\alpha$ as given in Eq. \eqref{eq: jsq_ubound_fast} and $u_i \leq A$, so,
\begin{align*}
    \EE[\|{\bf u}\|^4] \leq nA^2\EE[\|{\bf u}\|^2] \leq nA^3 \EE[\vsum{u}] \leq nA^3 (1+C_f + \phi_1^{-1}M_1) \gamma^\alpha.
\end{align*}
This proves Lemma \ref{LEM: JSQ_SSC_fast_a}. For Lemma \ref{LEM: JSQ_SSC_fast_b},
\begin{align*}
    \langle \boldsymbol \phi, \q \rangle &= \langle \boldsymbol \phi, \q_{\|} \rangle + \langle \boldsymbol \phi, \q_{\perp}\rangle = \frac{1}{n} \langle {\bf 1}, \boldsymbol \phi \rangle \vsum{q} + \langle \boldsymbol \phi, \q_{\perp} \rangle = \frac{1}{n} \phi \vsum{q}+\langle \boldsymbol \phi, \q_{\perp}\rangle.
\end{align*}
Further,
\begin{align*}
   \frac{1}{n}\phi \vsum{q} + [\langle\boldsymbol \phi, \q_\perp \rangle]^+ = \langle \boldsymbol \phi, \q_{\|} \rangle + [\langle\boldsymbol \phi, \q_\perp \rangle]^+ \leq \max \big\{\langle \boldsymbol \phi, \q \rangle, \langle \boldsymbol \phi, \q_{\|} \rangle\big\} \leq \|\boldsymbol \phi\| \vsum{q}.
\end{align*}
Thus, by using $e^{x}-1\leq |x|e^{[x]^+}$, we obtain
\begin{align*}
 \Big| \EE\left[ {e^{\gamma^\alpha \langle\boldsymbol \phi, \q \rangle }} \right] - \EE\left[ {e^{\frac{1}{n} \gamma^\alpha\phi \vsum{q}}} \right] \Big| &= \Big| \EE\left[e^{\frac{1}{n} \gamma^\alpha\phi \vsum{q} } \big(e^{\gamma^\alpha \langle\boldsymbol \phi, \q_\perp \rangle } -1\big) \right] \Big|\\
 &\leq \gamma^\alpha \EE \left[ |\langle\boldsymbol \phi, \q_\perp \rangle| e^{\frac{1}{n} \gamma^\alpha\phi \vsum{q} } e^{\gamma^\alpha [\langle\boldsymbol \phi, \q_\perp \rangle]^+ } \right] \\
 &\stackrel{(a)}{\leq }\gamma^\alpha \|\boldsymbol \phi\| \EE\left[ \| \q_\perp\| e^{\gamma^\alpha \|\boldsymbol \phi\| \vsum{q}} \right]\\
 &\stackrel{(b)}{\leq }\gamma^\alpha \|\boldsymbol \phi\|\EE\left[ \| \q_\perp\|^2 \right]^{\frac{1}{2}} \EE\left[  e^{2\gamma^\alpha \|\boldsymbol \phi\| \vsum{q}} \right]^{\frac{1}{2}}\\
 &\stackrel{(c)}{\leq }  \gamma^\alpha \|\boldsymbol \phi\|\sqrt{M_\perp M_1},
\end{align*}
where (a) and (b) follows from Cauchy-Schwarz inequality; and (c) follows from Lemma \ref{lem: jsq_mgfexist_c} and Theorem \ref{LEM: JSQ_SSC}. This proves Lemma \ref{LEM: JSQ_SSC_fast_b}.
\end{proof}

\subsection{Proof of lemmas Theorem \ref{thm: jsq_crit}}
\label{app: jsq_crit}
Recall that $(\q^+, \q, {\bf d},{\bf u})$ denotes a vector of random variables distributed as the steady state of $({\bf q}(t+1),{\bf q}(t),{\bf d}(t),{\bf u}(t))$, with ${\bf q}^+ = {\bf q}+ {\bf c}- {\bf d}+ {\bf u}$, where ${\bf c}$ has the same distribution as ${\bf c}(t)$. Thus, ${\bf q}^+$ and ${\bf q}$ have the same distribution. Also, we have $|c_i|\leq A$ and $0\leq u_i \leq A$.  

\begin{proof}[Proof of Lemma \ref{lem: jsq_approximation_crit}: ]
Lemma \ref{lem: jsq_mgfexist_a} implies that $\sqrt{\gamma}\EE[\vsum{q}] \leq \phi_1^{-1}M_1$. It follows that
\begin{align}
\label{eq: jsq_ubound_crit}
    \EE[\vsum{u}] = \gamma \EE[\vsum{q}] -\nu_\gamma \leq \sqrt{\gamma}(\phi_1^{-1}M_1 + C_c+1).
\end{align}
 Then, by similar arguments as in Eq. \eqref{eq: ssq_lhs}, we obtain
\begin{align}
\label{eq: jsq_lhs}
   \Big| \MM{-\vsum{u}} +\phi\sqrt{\gamma} \EE[\vsum{u}] -1\Big| \leq nA (\phi_1^{-1}M_1 + C_c+1) e^{n|\phi|A} |\phi|^2 \gamma^{\frac{3}{2}}.
\end{align}
This proves Lemma \ref{lem: jsq_approximation_crit_a} with $\overline K_c^u = nA (\phi_1^{-1}M_1 + C_c+1)$. The result in Lemma \ref{lem: jsq_approximation_crit_b} is same as that in Lemma \ref{lem: jsq_approximation_fast_b}. Now, by Lemma \ref{lem: jsq_mgfexist_a}, we have that for any $\phi < \phi_1$,
\begin{align*}
    \MM{\vsum{q}} \leq M_1 e^{\frac{\phi\nu_\gamma}{\sqrt{\gamma}}} \leq M_1e^{[\phi]^+ ([C_c]^++1)} \leq  M_1e^{[C_c]^++1} = \overline M_c
\end{align*}
Now, we use similar arguments as in Eq. \eqref{eq: ssq_highermoment},\eqref{eq: ssq_rhs2} and \eqref{eq: ssq_approximation_crit_cd}, to obtain
\begin{align*}
	\left| (\MM{\langle {\bf 1}, {\bf q} - {\bf d} \rangle} - \MM{\vsum{q}})\MM{c} +\phi \gamma \frac{d}{d\phi} \MM{\vsum{q}} \right| \leq  \phi^2 \gamma^{\frac{3}{2}} \left( \frac{2}{\phi_1}+\frac{8}{\phi_1^2} \right) \overline M_c+ \frac{2nA\overline M_c}{\phi_1} \phi^2 \gamma^{\frac{3}{2}} e^{n|\phi|  A}.
\end{align*}
This proves Lemma \ref{lem: jsq_approximation_crit_c}  with $\overline K_c^d = \Big(\frac{2}{\phi_1}+\frac{8}{\phi_1^2}\Big)\overline M_c + \frac{2nA\overline M_c}{\phi_1}$.
\end{proof}

\begin{proof}[Proof of Lemma \ref{LEM: JSQ_SSC_crit}: ]
Note that for any $\phi<\phi_1$, Lemma \ref{lem: jsq_mgfexist_a} implies that
\begin{align*}
    \MM{\vsum{q}} \leq M_1 e^{\frac{\phi\nu_\gamma}{\sqrt{\gamma}}} \leq M_1e^{[\phi]^+ ([C_c]^++1)} \leq  M_1e^{[C_c]^++1} = \overline M_c.
\end{align*}
Now by same argument as in the proof of Lemma \ref{LEM: JSQ_SSC_fast_a}, for $\phi<\phi_1/4$,
\begin{align*}
     \Big|\EE\Big[\Big(e^{\sqrt{\gamma} \phi \vsum{q^+}} -1\Big)\Big(e^{\sqrt{\gamma} \phi \vsum{u} }-1\Big)\Big]\Big| &\leq \phi^2 \gamma e^{n|\phi|A} \big(nA^3 (1+C_c + \phi_1^{-1}M_1)\big)^{\frac{1}{4}} \gamma^{\frac{1}{8}} M_\perp^{\frac{1}{2}} \overline M_c^\frac{1}{4}\\
     & = \overline{K}_c^{qu}\phi^2 \gamma^{\frac{9}{8}}e^{n|\phi|A}.
\end{align*}
This proves Lemma \ref{LEM: JSQ_SSC_crit_a}. For Lemma \ref{LEM: JSQ_SSC_crit_b}, by using the same argument as for Lemma \ref{LEM: JSQ_SSC_fast_b},
\begin{align*}
 \Big|  \EE\Big[{e^{\sqrt{\gamma} \langle\boldsymbol \phi, \q \rangle }}\Big] -  \EE\Big[{e^{\frac{1}{n} \sqrt{\gamma}\phi \vsum{q} }}\Big] \Big| & \leq  \sqrt{\gamma} \|\boldsymbol \phi\|M_\perp^\frac{1}{2} \overline{M}_c^{\frac{1}{2}},
\end{align*}
This proves Lemma \ref{LEM: JSQ_SSC_crit_b} as $\gamma$ goes to zero.
\end{proof}

\subsection{Proof of lemmas for Theorem \ref{thm: jsq_slow}}
\label{app: jsq_slow}

Recall that $(\q^+, \q, {\bf d},{\bf u})$ denotes a vector of random variables distributed as the steady state of $({\bf q}(t+1),{\bf q}(t),{\bf d}(t),{\bf u}(t))$, with ${\bf q}^+ = {\bf q}+ {\bf c}- {\bf d}+ {\bf u}$, where ${\bf c}$ has the same distribution as ${\bf c}(t)$. Thus, ${\bf q}^+$ and ${\bf q}$ have the same distribution. Also, we have $|c_i|\leq A$ and $0\leq u_i \leq A$.

\begin{proof}[Proof of Lemma \ref{lem: jsq_approximation_slow}: ]
For any $0<\phi<n\phi_1$, combining $q^+_i u_i =0, \ \forall i$ and Lemma \ref{lem: jsq_mgfexist} implies that
\begin{align*}
    \sum_{i=1}^n\EE[u_i] &= \sum_{i=1}^n\EE\left[u_i e^{-\sqrt{\gamma}\phi q_i^+}\right] \leq A \sum_{i=1}^n\EE\left[e^{-\sqrt{\gamma}\phi q_i^+}\right] = A\sum_{i=1}^n M_{q_i}^\gamma(-\phi) \leq AM_1 e^{-\frac{\phi \nu_\gamma}{n\sqrt{\gamma}}}.
\end{align*}
Since this is true for any $\phi<n\phi_1$, we get
\begin{equation}
\label{eq: jsq_slow_ubound}
    \sum_{i=1}^n\EE[u_i] \leq AM_1 e^{-\frac{\phi_1 \nu_\gamma}{\sqrt{\gamma}}}
\end{equation}
Using this, we get that
\begin{align*}
     \big| \MM{ -\vsum{u}} -  1\big|
     = \Big| \EE\Big[ \Big(e^{-\sqrt{\gamma}\phi \vsum{u}} -1\Big) \Big]\Big| \leq \sqrt{\gamma}|\phi| e^{nA} \EE [\vsum{u}]  \leq \sqrt{\gamma}|\phi| e^{nA} \times AM_1 e^{-\frac{\phi_1 \nu_\gamma}{\sqrt{\gamma}}}.
\end{align*}
This proves Lemma \ref{lem: jsq_approximation_slow_a} with $\overline K_s^u = AM_1e^{nA}$. Thus, for any $\phi>-\phi_1/2$,
\begin{equation*}
    \limg \frac{1}{|\phi|\gamma} \big| \MM{ -\vsum{u}} -  1\big|e^{-\frac{\phi \nu_\gamma}{\sqrt{\gamma}}} = 0.
\end{equation*}
For Lemma \ref{lem: jsq_approximation_slow_b}, we use the same argument as in the proof of Lemma \ref{lem:ssq_slow_mgfapprox_b}. 
\end{proof}

\begin{proof}[Proof of Lemma \ref{LEM: JSQ_SSC_slow}: ]
By using Cauchy-Schwarz inequality,
\begin{align*}
    \Big|\EE\Big[ \Big(e^{\sqrt{\gamma} \phi \langle {\bf 1},\bar{ {\bf q}}^+ \rangle} -1\Big) \Big(e^{-\sqrt{\gamma} \phi \vsum{u} }-1\Big)\Big]\Big| &\leq \EE\left[ \Big(e^{\sqrt{\gamma} \phi \langle {\bf 1},\bar{ {\bf q}}^+ \rangle} -1\Big)^2\right]^{\frac{1}{2}} \EE\left[\Big(e^{-\sqrt{\gamma} \phi \vsum{u} }-1\Big)^2\right]^{\frac{1}{2}}\\
    & \stackrel{(a)}{\leq }  \EE\Big[ \Big(e^{2\sqrt{\gamma} \phi \langle {\bf 1},\bar{ {\bf q}}^+ \rangle} +1\Big)\Big]^{\frac{1}{2}}\times \sqrt{\gamma
    } |\phi| nA e^{nA} \EE\big[\vsum{u}\big]^{\frac{1}{2}}\\
    & \stackrel{(b)}{\leq } (M_1+1)^{\frac{1}{2}} \times \sqrt{\gamma
    } |\phi| nA e^{nA} \times \sqrt{AM_1} e^{-\frac{\phi_1 \nu_\gamma}{2\sqrt{\gamma}}}\\
    & = \overline{K}_f^{qu} \sqrt{\gamma}|\phi| e^{-\frac{\phi_1 \nu_\gamma}{2\sqrt{\gamma}}},
\end{align*}
where (a) follows by using $e^{x}-1 \leq |x|e^{|x|}$ and $\vsum{u} \leq nA$; and (b) follows from using Lemma \ref{lem: jsq_mgfexist} and Eq. \eqref{eq: jsq_slow_ubound}. This proves Lemma \ref{LEM: JSQ_SSC_slow_a}. For Lemma \ref{LEM: JSQ_SSC_slow_b}, note that $\langle\boldsymbol \phi, \bar{ {\bf q}} \rangle = \frac{1}{n}\phi\langle {\bf 1},\bar{ {\bf q}} \rangle  + \langle\boldsymbol \phi,  {\bf q}_{\perp} \rangle$, so we use same argument as in the proof Lemma \ref{LEM: JSQ_SSC_fast_b}.
\end{proof}

\section{Proof of claims for Theorem \jsqssc}
\label{app: jsq_ssc}

\begin{proof}[Proof of Claim \ref{clm: jsq_ssc_abandonment}: ]
Consider a given queue in JSQ-A, i.e., fix $i$ and consider the process $\{q_i(t)\}_{t=0}^\infty$. The evolution process is given by 
\begin{align*}
    q_i(t+1) &= q_i(t) + c_i(t) -d_i(t) + u_i(t)\\
    &=  q_i(t) + a(t)Y_i(t) -s_i(t) -d_i(t) + u_i(t).
\end{align*}
We create a coupled process $\tilde q_i(t)$ with the initial condition $\tilde q_i(0) = q_i(0)$, by replacing $Y_i(t)$ with $1$. Then, 
\begin{equation*}
    \tilde q_i(t+1) = \tilde q_i(t)+  a(t) -s_i(t) - \tilde d_i(t) + \tilde u_i(t).
\end{equation*}
Thus, $\{\tilde q_i(t)\}_{t=0}^\infty $ behaves like a SSQ-A. It can be easily checked that $\tilde q_i(t) \geq q_i(t)$ for all $t\geq 0$, using a similar argument as in Eq. \eqref{eq: ssq_mgfexist_coupling}. Then, the result for SSQ-A and Lemma \ref{lem: ssq_mgfexist} implies that  the system is stable, and in steady state, there exists $\phi_0$ and $M_0$ such that, for any $\phi \in (-\phi_0,\phi_0)$,
\begin{equation*}
    \MM{q_i} \leq \MM{\tilde q_i} \leq M_0 \exp \Bigg( \frac{|\phi| [\lambda_\gamma - \mu_{\gamma,i}]^+}{\sqrt{\gamma}}\Bigg).
\end{equation*}
Thus, the moments of the steady state queue lengths $q_i$ are finite for any $i$. Using the bound on the MGF given in the previous equation, we get
\begin{align*}
    \EE[e^{\gamma \phi_0 q_i}] \leq M_0 \exp(|\phi_0| [\lambda_\gamma - \mu_{\gamma,i}]^+) \leq  M_0  e^{A}.
\end{align*}
Thus, 
\begin{align*}
    \EE\left[e^{\frac{1}{n} \gamma  \phi_0 \langle \q, {\bf 1}\rangle }\right] \leq \prod_{i=1}^n  \EE[e^{\gamma \phi_0 q_i}]^\frac{1}{n} \leq  M_0  e^{A},
\end{align*}
and
\begin{equation*}
    \gamma^m\EE[\langle \q, {\bf 1}\rangle^m] \leq \frac{n^m m!}{\phi_0^m} M_0  e^{A} := E_m.
\end{equation*}
Note that the bound provided above is very loose, and one can find a much tighter bound. However, for our purposes, such a loose bound is enough.
\end{proof}

\begin{proof}[Proof of Claim \ref{clm: jsq_ssc_drift}: ]
We have,
\begin{align}
\label{eq: jsq_lemssc_l2drift}
	\Delta \| \q  \|^2 &= \|\q^+\|^2-\|\q \|^2 \nonumber\\
	&= \|\q^+ -{\bf u} \|^2- \|{\bf u} \|^2 + 2\langle \q^+ ,{\bf u} \rangle -\|\q \|^2\nonumber\\
	&\stackrel{(a)}{=} \|\q  +{\bf c}  - {\bf d}  \|^2 - \|{\bf u} \|^2  -\|\q \|^2\nonumber\\
	&\leq \|{\bf c}  - {\bf d} \|^2 + 2\langle \q ,{\bf c}  - {\bf d} \rangle\nonumber\\
	& = \|{\bf c} \|^2+ \|{\bf d} \|^2 + 2\langle \q  - {\bf d} ,{\bf c}  \rangle - 2 \langle \q , {\bf d} \rangle\nonumber\\
	&\stackrel{(b)}{\leq} nA^2+\|{\bf d} \|^2 + 2\langle \q  - {\bf d} ,{\bf c}  \rangle - 2 \langle \q , {\bf d} \rangle,
\end{align}
where (a) follows by using $q_i^+u_i  = 0$ for all $i$; and (b) follows by using $|c_i | \leq A$. Now, since ${\bf d} \sim Bin(\q,\gamma)$, we have $\EE[\|{\bf d} \|^2] = \gamma(1-\gamma)\langle {\bf q} ,{\bf 1} \rangle + \gamma^2 \|\q \|^2$, $\EE[\langle \q , {\bf d} \rangle] = \gamma\|\q \|^2$, and \begin{align*}
    \EE[\langle \q  - {\bf d} ,{\bf c}  \rangle|\q] &= (1-\gamma)\EE[\langle \q ,{\bf c}  \rangle|\q] \\
    & = (1-\gamma)\big( \EE[a] \EE[\langle \q ,{\bf Y}  \rangle|\q] - \langle \q ,\EE[{\bf s}] \rangle \\
    &= (1-\gamma) \big(\lambda_\gamma q_{\min}  - \langle \q ,\boldsymbol\mu_\gamma \rangle\big).
\end{align*}
Combining these with Eq. \eqref{eq: jsq_lemssc_l2drift}, we have
\begin{align*}
    \EE[\Delta \| \q  \|^2 |\q ] &\leq nA^2 + \gamma(1-\gamma)\langle {\bf q} ,{\bf 1} \rangle + \gamma^2 \|\q \|^2 + 2(1-\gamma) \big(\lambda_\gamma q_{\min}  - \langle \q ,\boldsymbol\mu_\gamma \rangle\big) - 2\gamma\|\q \|^2\\
    & \leq  nA^2 + \gamma\langle\q ,{\bf 1} \rangle + (\gamma^2-2\gamma) \|\q \|^2 + 2(1-\gamma) \big(\lambda_\gamma q_{\min}  - \langle \q ,\boldsymbol\mu_\gamma \rangle\big),
\end{align*}
where $q_{\min}  = \min_i q_i $. Next, we have
\begin{align*}
    \Delta \langle\q ,{\bf 1} \rangle^2 &= \langle\q^+,{\bf 1} \rangle^2 - \langle\q ,{\bf 1} \rangle^2\\
        & = \langle\q^+ -{\bf u} ,{\bf 1} \rangle^2- \langle{\bf u} ,{\bf 1} \rangle^2 + 2 \langle \q^+,{\bf 1} \rangle \langle{\bf u} ,{\bf 1} \rangle -\langle\q ,{\bf 1} \rangle^2\\
        & \geq \langle\q + {\bf c}  - {\bf d} ,{\bf 1} \rangle^2 - \langle\q ,{\bf 1} \rangle^2 - \langle{\bf u} ,{\bf 1} \rangle^2\\
        & \stackrel{(a)}{=} \langle {\bf c} - {\bf d}  ,{\bf 1} \rangle^2 + 2\langle {\bf q} ,{\bf 1}\rangle \langle  {\bf c}  -{\bf d}  ,{\bf 1} \rangle - n^2 A^2\\
        & = (c  - \langle{\bf d} ,{\bf 1} \rangle)^2 +2 \langle\q ,{\bf 1} \rangle(c  - \langle{\bf d} ,{\bf 1} \rangle) - n^2A^2\\
		&= c^2  +  \langle{\bf d} ,{\bf 1} \rangle^2  +2 c  \langle\q  - {\bf d} ,{\bf 1} \rangle  - 2\langle{\bf q} ,{\bf 1} \rangle\langle{\bf d} ,{\bf 1} \rangle - n^2A^2\\
		&\geq \langle{\bf d} ,{\bf 1} \rangle^2  +2 c  \langle\q  - {\bf d} ,{\bf 1} \rangle  - 2\langle{\bf q} ,{\bf 1} \rangle\langle{\bf d} ,{\bf 1} \rangle - n^2A^2,
\end{align*}
where (a) follows from using $\langle{\bf u} ,{\bf 1} \rangle \leq nA$ when $u_i  \leq A$ for all $i$.
Taking expectation on both sides, and using the fact that $c $ is independent of ${\bf q} $ and ${\bf d} $, and ${\bf d}  \sim Bin({\bf q} ,\gamma)$, ew obtain
\begin{align*}
    \EE\big[ \Delta \langle\q ,{\bf 1} \rangle^2 \,\big|\, \q \big] & \geq \gamma(1-\gamma)\langle\q ,{\bf 1} \rangle + \gamma^2 \langle\q ,{\bf 1} \rangle^2 + 2(1-\gamma)\nu_\gamma\langle\q ,{\bf 1} \rangle -2\gamma\langle\q ,{\bf 1} \rangle^2 - n^2A^2\\
    & \ge   (\gamma^2-2\gamma) \langle\q ,{\bf 1} \rangle^2+2\nu_\gamma(1-\gamma)\langle\q ,{\bf 1} \rangle  - n^2A^2.
\end{align*}
\end{proof}

\begin{proof}[Proof of Claim \ref{clm: jsq_ssc_driftterm}: ]
Substituting $\lambda_\gamma = \nu_\gamma + \sum_{i=1}^n \mu_{\gamma,i} $, we have
\begin{align*}
	\lambda_\gamma  q_{\min} - \langle {\bf q} ,\boldsymbol \mu_\gamma \rangle- \frac{\nu_\gamma}{n}\langle \q ,{\bf 1} \rangle  
	& = \sum_{i=1}^n \mu_{\gamma,i} (q_{\min}  -q_i ) + \nu_\gamma\left(q_{\min}  - \frac{1}{n}\langle\q ,{\bf 1} \rangle \right) \\
	& \leq \mu_{\gamma,\min} \sum_{i=1}^n(q_{\min}  -q_i ) + \nu_\gamma\left( q_{\min} - \frac{1}{n}\langle\q ,{\bf 1} \rangle \right) \\
	& = (n\mu_{\gamma,\min} + \nu_\gamma) \left(q_{\min}  - \frac{1}{n}\langle\q ,{\bf 1} \rangle \right)\\
	&\leq (n\mu_{\gamma,\min} - \nu_\gamma^-) \left( q_{\min} - \frac{1}{n}\langle\q ,{\bf 1} \rangle \right) \\
	&\stackrel{(a)}{\leq} -\frac{1}{n}(n\mu_{\gamma,\min} - \nu_\gamma^-) \|\q_{\perp} \|\\
	&\stackrel{(b)}{\leq} -\frac{1}{2}\mu_{\gamma,\min }\|\q_\perp \|,
\end{align*}
where (a) follows from using $q_{\min} \leq \frac{1}{n}\langle\q ,{\bf 1} \rangle - \frac{1}{n}\|\q_{\perp} \|$ and by assuming $n\mu_{\gamma,\min} - \nu_\gamma^- \geq 0$; and (b) follows by using $\nu^-_\gamma \leq \frac{1}{2}n\mu_{\gamma,\min}$, where $\nu_\gamma^- = \max\{0,-\nu_\gamma\}$.
\end{proof}

\begin{proof}[Proof of Claim \ref{clm: jsq_ssc_bounded}: ]
\begin{align*}
		|\Delta \|\q_{\perp} \| | &= \big| \|\q_{\perp}^+\| -  \|\q_{\perp} \| \big|\\
		&\stackrel{(a)}{\leq} \|\q_{\perp}^+ - \q_{\perp} \|\\
		&\stackrel{(b)}{\leq} \|\q^+ - \q \|\\
		&\stackrel{(c)}{\leq} \| {\bf c} \| + \| {\bf d}  \| + \|{\bf u} \|\\ &\stackrel{(d)}{\leq}2nA +  \langle {\bf d}  ,{\bf 1} \rangle \\
		&=: Z({\bf d} ),
	\end{align*}
	where (a) follows by the triangle-inequality for the $\ell_2$-norm; (b) follows by using non-expansive property of the projection; (c) follows by using the equation $\q^+ = \q  +{\bf c}  - {\bf d}  + {\bf u} $; and (d) follows by using $|c_i |\leq A$ and $u_i  \leq A $ for all $i$.
\end{proof}

\begin{proof}[Proof of Claim \ref{clm: jsq_ssc_thridorder}: ]
    For the first term, we have
	\begin{align*}
	    3\|\q_{\perp} \|^2\EE[\Delta \|\q_{\perp} \| | \q ]  &\leq \begin{cases}
	     -3\epsilon_0\|\q_{\perp} \|^2 \ &\text{ if } \ \|\q_{\perp} \| \geq L(\q ) \\
	    3\|\q_{\perp} \|^2 \EE[Z({\bf d})  | \q ]\ &\text{ if } \ \|\q_{\perp} \| < L(\q )
	    \end{cases}\\
	    &\leq -3\epsilon_0\|\q_{\perp} \|^2 + 3\big(2nA+ \gamma\langle\q ,{\bf 1} \rangle + \epsilon_0\big) L^2(\q ).
	\end{align*}
	So, in steady state, we have
	\begin{align}
	\label{eq: jsq_lemssc_thridpower1}
	    3\EE[\|\q_{\perp}\|^2\Delta \|\q_{\perp}\|] &\leq  -3\epsilon_0\EE[\|\q_{\perp}\|^2] + \frac{12}{(1-\gamma)^2\mu_{\gamma,\min}^2}\EE\left[\big(2nA+ \gamma\langle \q,{\bf 1} \rangle + \epsilon_0\big) (2nA+ \gamma\langle \q,{\bf 1} \rangle)^2\right] \nonumber\\
	    &\leq  -3\epsilon_0\EE[\|\q_{\perp}\|^2] + \frac{12}{(1-\gamma)^2\mu_{\gamma,\min}^2}\EE\left[\big(2nA+ \gamma\langle \q,{\bf 1} \rangle + \epsilon_0\big)^3 \right] \nonumber\\
	    &\stackrel{(a)}{\leq} -3\epsilon_0\EE\left[\|\q_{\perp}\|^2\right] + K_1,
	\end{align}
	where (a) follows by using $\gamma^m\EE[\langle \q,{\bf 1} \rangle^m]\leq E_m$, and from the fact that there exist such a constant $K_1$ independent of the value of $\gamma$. For the second term, in steady state, we have
	\begin{align}
	\label{eq: jsq_lemssc_thridpower2}
	    \EE\left[3\|\q_{\perp}\| \big(\Delta \|\q_{\perp}\|\big)^2\right] &\leq \EE\big[3\|\q_{\perp}\| Z^2({\bf d})\big]\nonumber\\
	    &\stackrel{(a)}{\leq} 3\EE[\|\q_{\perp}\|^2]^{\frac{1}{2}} \EE[Z^4({\bf d})]^{\frac{1}{2}}\nonumber\\
	    &\stackrel{(b)}{\leq} 3K_2\EE[\|\q_{\perp}\|^2]^{\frac{1}{2}},
	\end{align}
	where (a) follows from using Cauchy-Schwarz inequality; and (b) follows by using the fact that $\EE[\langle {\bf d},{\bf 1} \rangle^m]$ is bounded by a constant irrespective of the value of $\gamma$ as $\gamma^m\EE[\langle \q,{\bf 1} \rangle^m] \leq E_m$, and so $K_2$ is a bounded constant irrespective of the value of $\gamma$. Finally, we have
	\begin{align}
	\label{eq: jsq_lemssc_thridpower3}
	    \EE\left[\big(\Delta \|\q_{\perp}\|\big)^3\right] &\leq  \EE\big[Z^3({\bf d})\big] \leq 3K_3, 
	\end{align}
	where $K_3$ is a bounded constant irrespective of the value of $\gamma$ due to the fact that $\EE[\langle {\bf d},{\bf 1} \rangle^m]$ is bounded by a constant irrespective of the value of $\gamma$.
\end{proof}

\end{appendix}

\end{document}